%% file: ms.tex
\documentclass[11pt,fleqn]{article}%
\pdfoutput=1
\input{packages.tex}%
\input{colordefinition.tex}%
\newboolean{WithComments}\setboolean{WithComments}{false}%
\input{customcommands.tex}%
\newboolean{WithTUMHeader}\setboolean{WithTUMHeader}{false}%
\newboolean{UseHelveticaFont}\setboolean{UseHelveticaFont}{true}%
\newboolean{isTRAC}\setboolean{isTRAC}{false}%
\def\TRACVolume{???}%
\def\TRACDate{April 2018}%
\input{styletemplate.tex}%
\hypersetup{pdftitle = {H2 Pseudo-Optimal Reduction of Structured DAEs by Rational Interpolation},%
            pdfsubject = {},%
            pdfauthor = {Philipp Seiwald, Alessandro Castagnotto, Tatjana Stykel, Boris Lohmann},%
            pdfkeywords = {MOR, Model Order Reduction, DAE, Differential Algebraic Equation, Stability Preservation, Krylov Subspace Methods, H2 Pseudo-Optimality, Rational Interpolation},%
            pdfcreator = {},%
            pdfproducer = {}}%
\date{April 2018}%
\title{\HtwoText\ Pseudo-Optimal Reduction of Structured DAEs by Rational Interpolation}%
\newcommand{\runtitle}{\HtwoText\ Pseudo-Optimal Reduction of Structured DAEs by Rational Interpolation}%
\author{Philipp Seiwald\footnote{Corresponding author, e-mail: \href{mailto:philipp.seiwald@tum.de}{philipp.seiwald@tum.de}, present address: Technical University of Munich, Chair of Applied Mechanics, Boltzmannstr. 15, D-85748 Garching, Germany}\hspace*{0.7em}\footnote{Technical University of Munich, Chair of Automatic Control, Boltzmannstr. 15, D-85748 Garching, Germany} \quad Alessandro Castagnotto\footnotemark[2] \quad Tatjana Stykel\footnote{University of Augsburg, Department of Mathematics, Universit\"atsstr. 14, D-86159 Augsburg, Germany} \quad Boris Lohmann\footnotemark[2]}%
\begin{document}%
    %
    \maketitle%
    \input{abstract.tex}%
    \input{keywords.tex}%
    \input{declarationofinterest.tex}%
    %
    \input{introduction.tex}%
    \input{fundamentals.tex}%
    \input{modelReduction.tex}%
    \input{H2innerproduct.tex}%
    \input{pork.tex}%
    \input{curedspark.tex}%
    \input{numericalresults.tex}%
    \input{conclusions.tex}%
    %
    \printbibliography%
\end{document}

%% file: packages.tex
\usepackage[utf8]{inputenc}%
\usepackage[ngerman,english]{babel}%
\usepackage{ifthen}%
\usepackage{time}%
\usepackage{amssymb}%
\usepackage{amsfonts}%
\usepackage{amsthm}%
\usepackage{amsmath}%
\usepackage{mathtools}%
\usepackage{thmtools}%
\usepackage{units}%
\usepackage{siunitx}%
\usepackage{xfrac}%
\usepackage{cancel}%
\usepackage{color}%
\usepackage{xcolor}%
\usepackage{tikz}%
\usetikzlibrary{decorations.pathreplacing}%
\usetikzlibrary{shapes,arrows,patterns}%
\usepackage{pgfplots}%
\pgfplotsset{compat=newest}%
\pgfplotsset{plot coordinates/math parser=false}%
\pgfplotsset{yticklabel style={text width=2.5em,align=right}}%
\pgfplotsset{/pgf/number format/.cd, fixed,precision=6}%
\newlength\myheight
\newlength\mywidth
\usepackage[markup=nocolor]{changes}%
\usepackage{soul}%
\usepackage{pdfcomment}%
\usepackage{enumerate}%
\usepackage{fancyhdr}%
\usepackage{caption}%
\usepackage{subcaption}%
\usepackage{graphicx}%
\usepackage{rotating}%
\usepackage{pdfpages}%
\usepackage{epsfig}%
\usepackage{longtable}%
\usepackage{multirow}%
\usepackage{makecell}%
\usepackage{booktabs}%
\usepackage{colortbl}%
\usepackage{csquotes}%
\usepackage[backend=biber,style=numeric-comp,sorting=nyt,maxnames=99,giveninits=true]{biblatex}%
\DeclareNameAlias{sortname}{last-first}
\usepackage{listings}%
\usepackage[ruled,vlined,linesnumbered]{algorithm2e}%
\DontPrintSemicolon%
\usepackage{hyperref}%
\DeclareUrlCommand\doi{\def\UrlLeft##1\UrlRight{DOI:\nobreakspace\href{http://dx.doi.org/##1}{##1}}\urlstyle{rm}}%
\usepackage[noabbrev,capitalize]{cleveref}%
\crefformat{equation}{{#2{\color{black}(}#1{\color{black})}#3}}%
\Crefformat{equation}{{#2{\color{black}(}#1{\color{black})}#3}}%
\crefname{equation}{}{}%
\crefformat{enumeratei}{{#2{\color{black}(}#1{\color{black})}#3}}%
\Crefformat{enumeratei}{{#2{\color{black}(}#1{\color{black})}#3}}%
\crefname{enumeratei}{}{}%

%% file: colordefinition.tex
\definecolor{tumblue}{HTML}{0065BD}%
\definecolor{tumblue1}{HTML}{98C6EA}%
\definecolor{tumblue2}{HTML}{64A0C8}%
\definecolor{tumblue3}{HTML}{0073CF}%
\definecolor{tumblue4}{HTML}{005293}%
\definecolor{tumblue5}{HTML}{003359}%
\definecolor{tumgreen}{HTML}{A2AD00}%
\definecolor{tumorange}{HTML}{E37222}%
\definecolor{tumivory}{HTML}{DAD7CB}%
\definecolor{tumviolet}{HTML}{69085A}%
\definecolor{tumred}{HTML}{C4071B}%
\definecolor{bg-gray}{HTML}{E5E5E5}%
\definecolor{bg-red}{HTML}{FFDDDD}%
\definecolor{bg-green}{HTML}{DDFFDD}%
\definecolor{bg-blue}{HTML}{DDDDFF}%
\definecolor{text-gray}{HTML}{404040}%
\definecolor{text-red}{HTML}{AA0000}%
\definecolor{text-green}{HTML}{00AA00}%
\definecolor{text-blue}{HTML}{0000AA}%

%% file: customcommands.tex
\ifthenelse{\boolean{WithComments}}%
{\newcommand{\ToEditor}[1]{{\footnotesize\color{red}Note to the editor/reviewer: #1}}}%
{\newcommand{\ToEditor}[1]{}}%
\newcommand{\mcode}[1]{{\fontfamily{pcr}\selectfont #1}}%
\newtheorem{theorem}{Theorem}%
\newtheorem{lemma}{Lemma}%
\theoremstyle{definition}\newtheorem{definition}{Definition}%
\newtheorem{remark}{Remark}%
\newcommand{\subspace}{subspace}%
\renewcommand{\phi}{\varphi}
\newcommand{\ppp}{...}%
\newcommand{\ts}{\!}
\newcommand{\tin}{\ts\in\ts}
\newcommand{\teq}{\ts=\ts}%
\newcommand{\tneq}{\ts\neq\ts}%
\newcommand{\tto}{\ts\to\ts}%
\newcommand{\indextext}[1]{\mathrm{#1}}%
\newcommand{\fo}{n}
\newcommand{\ro}{q}
\newcommand{\is}{m}
\newcommand{\os}{p}
\newcommand{\strictlyproper}{\indextext{sp}}%
\newcommand{\improper}{\indextext{im}}%
\newcommand{\rsys}{\indextext{r}}
\newcommand{\esys}{\indextext{e}}
\newcommand{\Hsys}{\indextext{H}}
\newcommand{\Msys}{\indextext{M}}
\newcommand{\Fsys}{\indextext{F}}
\newcommand{\prim}{\indextext{P}}
\newcommand{\finite}{f}
\newcommand{\infinite}{\infty}
\newcommand{\idxleft}{l}
\newcommand{\idxright}{r}
\newcommand{\Hinf}{\mathcal{H}_\infty}
\newcommand{\Htwo}{\texorpdfstring{\mathcal{H}_2}{H2}}
\newcommand{\HtwoText}{\texorpdfstring{$\mathcal{H}_2$}{H2}}
\newcommand{\Ltwopm}{\texorpdfstring{\mathcal{L}_2^{(p,m)}}{L2(p,m)}}
\newcommand{\norm}[2]{\lVert #1 \rVert_{#2}}
\newcommand{\normHtwo}[1]{\lVert #1 \rVert_{\Htwo}}
\newcommand{\normLtwopm}[1]{\lVert #1 \rVert_{\Ltwopm}}
\newcommand{\normFrobenius}[1]{\lVert #1 \rVert_{F}}
\newcommand{\innerproductHtwo}[2]{\langle #1,\,#2 \rangle_{\Htwo}}
\newcommand{\assign}{:=}
\newcommand{\diff}{\mathrm{d}}
\newcommand{\dt}{\mathrm{d}t}
\newcommand{\grad}{\T{g}}
\newcommand{\Hess}{\T{H}}
\newcommand{\cost}{\mathcal{J}}
\newcommand{\Real}{\mathbb{R}}
\newcommand{\Complex}{\mathbb{C}}
\renewcommand{\Re}{\mathrm{Re}}%
\DeclareMathOperator{\Image}{\mathcal{R}}%
\DeclareMathOperator{\trace}{tr}
\DeclareMathOperator{\diag}{diag}
\newcommand{\transpose}{\indextext{T}}
\newcommand{\complexconjugate}[1]{\overline{#1}}
\newcommand{\complexconjugatetranspose}{\ast}
\newcommand{\T}[1]{\boldsymbol{\mathrm{#1}}}%
\def\tendeflatinlower#1{\expandafter\def\csname T#1\endcsname{\T{#1}}}%
\def\tendeflatinlowerall#1{\ifx#1\tendeflatinlowerall\else\tendeflatinlower{#1}\expandafter\tendeflatinlowerall\fi}%
\tendeflatinlowerall abcdefghijklmnopqrstuvwxyz\tendeflatinlowerall%
\def\tendeflatinupper#1{\expandafter\def\csname T#1\endcsname{\T{#1}}}%
\def\tendeflatinupperall#1{\ifx#1\tendeflatinupperall\else\tendeflatinupper{#1}\expandafter\tendeflatinupperall\fi}%
\tendeflatinupperall ABCDEFGHIJKLMNOPQRSTUVWXYZ\tendeflatinupperall%
\def\tendefgreeklower#1{\expandafter\def\csname T#1\endcsname{\T{\csname #1\endcsname}}}%
\def\tendefgreeklowerall#1{\ifx#1\tendefgreeklowerall\else\tendefgreeklower{#1}\expandafter\tendefgreeklowerall\fi}%
\tendefgreeklowerall {alpha}{beta}{gamma}{delta}{epsilon}{varepsilon}{zeta}{eta}{theta}{vartheta}{iota}{kappa}{lambda}{mu}{nu}{pi}{varpi}%
{xi}{rho}{varrho}{sigma}{varsigma}{tau}{upsilon}{phi}{varphi}{chi}{psi}{omega}\tendefgreeklowerall%
\def\tendefgreekupper#1{\expandafter\def\csname T#1\endcsname{\T{\csname #1\endcsname}}}%
\def\tendefgreekupperall#1{\ifx#1\tendefgreekupperall\else\tendefgreekupper{#1}\expandafter\tendefgreekupperall\fi}%
\tendefgreekupperall {Gamma}{Delta}{Theta}{Lambda}{Xi}{Pi}{Sigma}{Upsilon}{Phi}{Psi}{Omega}\tendefgreekupperall%
\newcommand{\Tzero}{\T{0}}%
\newcommand{\oft}{(t)}
\newcommand{\ofs}{(s)}
\newcommand{\Txr}{\Tx_{\rsys}}%
\newcommand{\Txt}{\Tx\oft}%
\newcommand{\Txp}{\dot{\Tx}}%
\newcommand{\Txpt}{\Txp\oft}%
\newcommand{\Tut}{\Tu\oft}%
\newcommand{\Tyt}{\Ty\oft}%
\newcommand{\TGs}{\TG \ofs}
\newcommand{\TGr}{\TG_\rsys}
\newcommand{\TGrs}{\TGr \ofs}
\newcommand{\TGe}{\TG_\esys}
\newcommand{\TGes}{\TGe \ofs}
\newcommand{\TGH}{\TG_\Hsys}
\newcommand{\TGHs}{\TGH \ofs}
\newcommand{\TGM}{\TG_\Msys}
\newcommand{\TGMs}{\TGM \ofs}
\newcommand{\TGF}{\TG_\Fsys}
\newcommand{\TGFs}{\TGF \ofs}
\newcommand{\TPs}{\TP \ofs}
\newcommand{\TGbot}{\TG_{\bot}}
\newcommand{\TGbots}{\TG_{\bot} \ofs}
\newcommand{\TGtimedomain}{\widehat{\TG}}
\newcommand{\TGHtimedomain}{\widehat{\TG}_\Hsys}
\newcommand{\TEW}{\widetilde{\TE}}
\newcommand{\TAW}{\widetilde{\TA}}
\newcommand{\TBW}{\widetilde{\TB}}
\newcommand{\TCW}{\widetilde{\TC}}
\newcommand{\TEr}{\TE_\rsys}
\newcommand{\TAr}{\TA_\rsys}
\newcommand{\TBr}{\TB_\rsys}
\newcommand{\TCr}{\TC_\rsys}
\newcommand{\TDr}{\TD_\rsys}
\newcommand{\TEre}{\TW^{\transpose}\TE\TV}
\newcommand{\TAre}{\TW^{\transpose}\TA\TV}
\newcommand{\TBre}{\TW^{\transpose}\TB}
\newcommand{\TCre}{\TC\TV}
\newcommand{\TEH}{\TE_\Hsys}
\newcommand{\TAH}{\TA_\Hsys}
\newcommand{\TBH}{\TB_\Hsys}
\newcommand{\TCH}{\TC_\Hsys}
\newcommand{\TJH}{\TJ_\Hsys}
\newcommand{\TNH}{\TN_\Hsys}
\newcommand{\TPH}{\TP_\Hsys}
\newcommand{\TQH}{\TQ_\Hsys}
\newcommand{\TEM}{\TE_\Msys}
\newcommand{\TAM}{\TA_\Msys}
\newcommand{\TBM}{\TB_\Msys}
\newcommand{\TCM}{\TC_\Msys}
\newcommand{\TEF}{\TE_\Fsys}
\newcommand{\TAF}{\TA_\Fsys}
\newcommand{\TBF}{\TB_\Fsys}
\newcommand{\TCF}{\TC_\Fsys}
\newcommand{\TXff}{\TX_{\finite\finite}}
\newcommand{\TXfi}{\TX_{\finite\infinite}}
\newcommand{\TXif}{\TX_{\infinite\finite}}
\newcommand{\TXii}{\TX_{\infinite\infinite}}
\newcommand{\TIfin}{\TI_{n_\finite}}
\newcommand{\TIinf}{\TI_{n_\infinite}}
\newcommand{\TIHfin}{\TI_{n_{\Hsys\finite}}}
\newcommand{\TIHinf}{\TI_{n_{\Hsys\infinite}}}

\newcommand{\TBfin}{\TB_\finite}
\newcommand{\TBinf}{\TB_\infinite}
\newcommand{\TCfin}{\TC_\finite}
\newcommand{\TCinf}{\TC_\infinite}
\newcommand{\TBHfin}{\TB_{\Hsys\finite}}
\newcommand{\TBHinf}{\TB_{\Hsys\infinite}}
\newcommand{\TCHfin}{\TC_{\Hsys\finite}}
\newcommand{\TCHinf}{\TC_{\Hsys\infinite}}
\newcommand{\Spec}{\T{\Pi}}
\newcommand{\Specleftfin}{\Spec^\finite_\idxleft}
\newcommand{\Specrightfin}{\Spec^\finite_\idxright}
\newcommand{\Specleftinf}{\Spec^\infinite_\idxleft}
\newcommand{\Specrightinf}{\Spec^\infinite_\idxright}
\newcommand{\SpecleftfinCT}{(\Spec^\finite_\idxleft)^\complexconjugatetranspose}
\newcommand{\SpecrightfinCT}{(\Spec^\finite_\idxright)^\complexconjugatetranspose}
\newcommand{\SpecHleftfin}{\Spec^\finite_{\Hsys \idxleft}}
\newcommand{\SpecHrightfin}{\Spec^\finite_{\Hsys \idxright}}
\newcommand{\SpecHleftfinCT}{(\Spec^\finite_{\Hsys \idxleft})^\complexconjugatetranspose}
\newcommand{\SpecHrightfinCT}{(\Spec^\finite_{\Hsys \idxright})^\complexconjugatetranspose}
\newcommand{\controllability}{\indextext{c}}%
\newcommand{\Gram}{\T{\Gamma}}%
\newcommand{\Gramcr}{\Gram^\controllability_\rsys}
\newcommand{\GramcM}{\Gram^\controllability_\Msys}
\newcommand{\GramcMinv}{(\GramcM)^{-1}}
\newcommand{\MEW}{\begin{bmatrix} \TIfin & \Tzero \\ \Tzero & \TN \end{bmatrix}}
\newcommand{\MAW}{\begin{bmatrix} \TJ & \Tzero \\ \Tzero & \TIinf \end{bmatrix}}
\newcommand{\MBW}{\begin{bmatrix} \TBfin \\ \TBinf \end{bmatrix}}
\newcommand{\MCW}{\begin{bmatrix} \TCfin & \TCinf \end{bmatrix}}
\newcommand{\pencilEA}{\lambda\,\TE-\TA}
\newcommand{\pencilErAr}{\lambda\,\TEr-\TAr}
\newcommand{\pencilEHAH}{\lambda\,\TEH-\TAH}
\newcommand{\sEmA}{s\,\TE-\TA}
\newcommand{\Sys}{\T{\Sigma}}
\newcommand{\Sysr}{\Sys_\rsys}
\newcommand{\SysM}{\Sys_\Msys}
\newcommand{\SysF}{\Sys_\Fsys}
\newcommand{\Realization}{(\TE,\,\TA,\,\TB,\,\TC)}
\newcommand{\RealizationD}{(\TE,\,\TA,\,\TB,\,\TC,\,\TD)}
\newcommand{\Realizationr}{(\TEr,\,\TAr,\,\TBr,\,\TCr)}
\newcommand{\RealizationH}{(\TEH,\,\TAH,\,\TBH,\,\TCH)}
\newcommand{\RealizationM}{(\TEM,\,\TAM,\,\TBM,\,\TCM)}
\newcommand{\RealizationF}{(\TEF,\,\TAF,\,\TBF,\,\TCF)}
\newcommand{\RealizationrTilde}{(\TEr,\,\TAr,\,\TBr,\,\TR,\,\TI_m)}
\newcommand{\Transfer}{\TC\left(\sEmA\right)^{-1}\TB}
\newcommand{\Transferr}{\TCr\left(s\,\TEr-\TAr\right)^{-1}\TBr}
\newcommand{\TransferM}{\TCM\left(s\,\TEM-\TAM\right)^{-1}\TBM}
\newcommand{\TransferF}{\TCF\left(s\,\TEF-\TAF\right)^{-1}\TBF}
\newcommand{\TransferSubspaceXY}[2]{\mathcal{G}(#1,\,#2)}%
\def\DAE{DAE}%
\def\DAEsys{\DAE\ system}%
\def\PORK{PORK}%
\def\PORKA{\PORK\ algorithm}%
\def\Weierstrass{Weierstraß}%
\def\WCF{\Weierstrass\ canonical form}%
\def\Jordan{Jordan}%
\def\JCF{\Jordan\ canonical form}%
\def\HTIP{\HtwoText\ inner product}%
\def\po{pseudo-optimal}%
\def\poy{\po ity}%
\def\HTO{\HtwoText\ optimal}%
\def\HTOy{\HtwoText\ optimality}%
\def\HTPO{\HtwoText\ \po}%
\def\HTPOy{\HtwoText\ \poy}%
\def\HTN{\HtwoText-norm}%
\def\Sylvester{Sylvester}%
\def\SE{\Sylvester\ equation}%
\def\gSE{generalized \SE}%

%% file: styletemplate.tex
\usepackage{helvet}%
\ifthenelse{\boolean{UseHelveticaFont}}%
{%
    \renewcommand{\sfdefault}{phv}%
}%
{%
    \renewcommand{\sfdefault}{cmr}%
}%
\fontfamily{\sfdefault}\selectfont%
\hypersetup{colorlinks = true,
            linkcolor = tumblue,
            citecolor = tumblue,
            urlcolor = tumblue,
            }%
\ifthenelse{\boolean{isTRAC}}%
{%
    \fancyhf{}\fancyhead[L]{{\small TUM Technical Reports on Automatic Control, Vol. \TRACVolume}} \fancyhead[R]{\thepage}\pagestyle{fancy}%
}%
{%
    \fancyhf{}\fancyhead[L]{{\small \runtitle}} \fancyhead[R]{\thepage}\pagestyle{fancy}%
}
\makeatletter%
\def\section{\@startsection{section}{1}{\z@}{-0.1ex plus -.1ex minus -.1ex}{0.1ex plus .05ex}{\large\bf}}%
\def\subsection{\@startsection{subsection}{2}{\z@}{-1.5ex}{0.01ex}{\normalsize\bf}}%
\def\subsubsection{\@startsection{subsubsection}{3}{\z@}{1.3ex \@plus .5ex \@minus .2ex}{-.5em \@plus -.1em}{\reset@font\normalsize\bfseries}}%
\makeatother%
\renewenvironment{abstract}{\centerline{\bf Abstract}\vspace{-2ex}\begin{quote}\small}{\par\end{quote}\vskip 1ex}%
\newcommand{\DrawTUMTercet}{%
    \tikzset{external/export next=false}%
    \begin{tikzpicture}[remember picture, overlay]%
        \node[shift={(1cm,-1cm)},outer sep=0,inner sep=0,anchor = north west] at (current page.north west) {%
            \color{tumblue}%
            \fontfamily{phv}\fontsize{8}{10.6}\selectfont%
            \begin{tabular}{@{}l@{}}%
                Chair of Automatic Control\\%
                Department of Mechanical Engineering\\%
                Technical University of Munich%
            \end{tabular}%
            \fontfamily{\sfdefault}\selectfont%
        };%
    \end{tikzpicture}%
}%
\newcommand{\DrawTUMLogo}{%
    \tikzset{external/export next=false}%
    \begin{tikzpicture}[remember picture, overlay]%
        \node[shift={(-1cm,-1cm)},outer sep=0,inner sep=0,anchor = north east] at (current page.north east) {%
            \resizebox{!}{1cm}{%
                \begin{tikzpicture}[y=0.80pt, x=0.80pt, yscale=-1, xscale=1, inner sep=0pt, outer sep=0pt]%
                    \begin{scope}[even odd rule,line width=0.800pt]%
                      \begin{scope}[fill=tumblue]%
                        \path[fill] (405,0) -- (370,0) -- (370,320) -- (290,320) -- (290,0) -- (255,0) -- (0,0) --
                            (0,70) -- (70,70) -- (70,390) -- (140,390) -- (140,70) -- (220,70) -- (220,390) -- 
                            (255,390) -- (405,390) -- (440,390) -- (440,70) -- (520,70) -- (520,390) -- (590,390) -- 
                            (590,70) --  (670,70) -- (670,390) -- (740,390) -- (740,35) -- (740,0) -- (405,0) -- cycle;
                      \end{scope}%
                    \end{scope}%
                \end{tikzpicture}%
            }%
        };%
    \end{tikzpicture}%
}%
\newcommand{\DrawTracTitleNote}{%
    \tikzset{external/export next=false}%
    \begin{tikzpicture}[remember picture, overlay]%
        \node[shift={(8cm,-1cm)},outer sep=0,inner sep=0,anchor = north west] at (current page.north west) {%
            \begin{tabular}{@{}c@{}}%
                \normalsize Technical Reports on Automatic Control\\[-0.2em]%
                \footnotesize Volume \TRACVolume, \TRACDate\\[-0.3em]%
                \footnotesize Available online at: \href{https://www.rt.mw.tum.de}{www.rt.mw.tum.de}%
            \end{tabular}%
        };%
    \end{tikzpicture}%
}%
\makeatletter%
\def\@maketitle{%
    \newpage%
    \ifthenelse{\boolean{WithTUMHeader}}{\DrawTUMTercet\DrawTUMLogo}{}%
    \ifthenelse{\boolean{isTRAC}}{\DrawTracTitleNote}{}%
    \vspace*{-0.5cm}%
    \begin{center}%
        \let \footnote \thanks%
        {\Large\bf \@title \par}%
        \vskip 1.5em%
        {%
            \large%
            \lineskip .5em%
            \begin{tabular}[t]{c}%
                \@author%
            \end{tabular}%
            \par%
        }%
    \end{center}%
    \par%
    \vskip 1.5em%
}%
\makeatother%
\makeatletter%
\renewcommand*{\@fnsymbol}[1]{\@alph{#1})}%
\makeatother%

%% file: abstract.tex
\begin{abstract}%
    In this contribution, we extend the concept of \HTIP\ and \HTPOy\ to dynamical systems modeled by differential-algebraic equations (DAEs). To this end, we derive projected Sylvester equations that characterize the \HTIP\ in terms of the matrices of the DAE realization. Using this result, we extend the $\Htwo$ pseudo-optimal rational Krylov algorithm for ordinary differential equations to the DAE case. This algorithm computes the globally optimal reduced-order model for a given subspace of $\Htwo$ defined by poles and input residual directions. Necessary and sufficient conditions for $\Htwo$ pseudo-optimality are derived using the new formulation of the \HTIP\ in  terms of tangential interpolation conditions. Based on these conditions, the cumulative reduction procedure combined with the adaptive rational Krylov algorithm, known as CUREd SPARK, is extended to DAEs. Important properties of this procedure are that it guarantees stability preservation and adaptively selects interpolation frequencies and reduced order. Numerical examples are used to illustrate the theoretical discussion. Even though the results apply in theory to general DAEs, special structures will be exploited for numerically efficient implementations.%
\end{abstract}%

%% file: keywords.tex
\begingroup%
    \small%
\endgroup\par%

%% file: declarationofinterest.tex
\begingroup%
    \small%
    \textbf{Declaration of interest:} %
    The work related to this contribution is supported by the German Research Foundation (DFG), grants LO~408/19-1 and STY~58/1-2.%
\endgroup\par%

%% file: introduction.tex
\section{Introduction}%
\label{sec:introduction}%
Systems of differential-algebraic equations (DAEs) represent a widespread formalism in the modeling of dynamical systems, e.g. constrained mechanical systems and electrical networks \cite{Kunkel2006,Riaza2008,Simeon13}. Due to the automatic, object-oriented generation of the model equations and the resulting redundancies in the descriptor variables, DAE systems often reach a very high order, thus making simulations and control design computationally challenging, if at all feasible. This motivates the use of model order reduction techniques for approximating large-scale system with reduced-order models which capture the relevant input-output dynamics and preserving fundamental properties of the original model such as stability. Due to the particular characteristics of DAEs, new model reduction techniques designed to work directly on the DAE system matrices are necessary. Their goal is to approximate the subsystem describing the dynamic evolution while preserving the subsystem describing the algebraic constraints.\par%
Several reduction methods have been developed in recent years to correctly deal with the algebraic part of the DAE, a survey of which is given in \cite{Benner2017}. Existing reduction approaches include the extension of the (low-rank, square-root) \emph{balanced truncation} algorithm as well as some variants thereof \cite{Stykel2004,morReiS10,morReiS11}. Amongst others, these methods bear the advantage of having a~priori error bounds which allow an appropriate selection of reduced order by inspecting proper and improper Hankel singular values. Balancing-related model reduction techniques are based on matrix equations, which in the large-scale setting can be solved using low-rank iterative methods, see \cite{Benner2013,Simoncini2016} for recent surveys. These methods may, however, have difficulties to converge, as, for example, for weakly damped mechanical systems.\par%
A robust and numerically efficient alternative to balanced truncation is given by \emph{rational interpolation} (also known as \emph{moment matching} or \emph{Krylov subspace methods}) \cite{Villemagne1987,Grimme1997,Gallivan2004TI,Gugercin2008}. Modified interpolatory subspace conditions were presented in \cite{Gugercin2013} to effectively reduce DAEs by means of tangential interpolation. Based on these conditions, a~fixed-point iteration was presented there that yields a~locally \HTO\space reduced model at convergence. This method has shown to work well in practice, even though no proof of convergence or stability preservation is available for general systems. In addition, the problem of finding a suitable reduced order still remains.\par%
In this contribution, we consider the problem of finding the globally optimal reduced-order model in a~subspace of $\Htwo$ defined by reduced poles and input residual directions. To this end, we characterize the \HTIP\ of two rational matrix-valued functions in terms of their state space realizations with singular descriptor matrices. We also develop an algorithm to efficiently construct a reduced-order model which solves the optimization problem. This algorithm is an extension and generalization of the model reduction method proposed in \cite{Wolf2014} for ordinary differential equations (ODEs) to DAEs. The resulting \HTPO \space rational Krylov algorithm is a~fundamental component of the fully adaptive reduction method known as CUREd SPARK \cite{Panzer2014} which can now be applied to DAE systems. To the authors' knowledge, this is the first rational interpolation algorithm for DAEs which guarantees stability preservation and adaptively selects interpolation frequencies and reduced order.\par%
The remainder of the paper is outlined as follows. \Cref{sec:fundamentals} revises fundamentals about linear DAEs, while \Cref{sec:MOR} introduces the model reduction problem and known methods for tangential interpolation. In \Cref{sec:H2innerproduct}, we derive a new formulation of the \HTIP\ for strictly proper transfer functions based on the solution of projected Sylvester equations. In \Cref{sec:pork}, this new result is used to extend the \HTPO\space rational Krylov algorithm to DAEs. In \Cref{sec:curedspark}, we briefly revise the fully adaptive reduction framework CUREd SPARK that will be ultimately used in numerical experiments. As the SPARK algorithm has not been extended to multiple-input and multiple-output (MIMO) models yet, the numerical examples in \Cref{sec:numericalresults} are restricted to the single-input and single-output (SISO) case. In addition, the test models considered here have special structure allowing for a~numerical efficient implementation. Note, however, that the main result of this contribution, namely the derivation of the \HTPO\ rational Krylov algorithm for DAEs using a special formulation of the $\mathcal{H}_2$ inner product 
is neither restricted to the SISO case nor to the special structure of the DAE system.\par%

%% file: fundamentals.tex
\section{Fundamentals of Linear DAEs}%
\label{sec:fundamentals}%
Consider a linear time-invariant control system%
\begin{equation}%
    \label{eq:LTI_DAE_Definition}%
    \TE\,\Txpt=\TA\,\Txt+\TB\,\Tut\;,\qquad \Tyt=\TC\,\Txt+\TD\,\Tut\;,%
\end{equation}%
with state $\Txt\tin\Real^{n}$, input $\Tut\tin\Real^{m}$, output $\Tyt\tin\Real^{p}$, and system matrices $\TE,\TA\tin\Real^{n\times n}$, $\TB\tin\Real^{n\times m}$, $\TC\tin\Real^{p\times n}$, and $\TD\tin\Real^{p\times m}$. Within this contribution, we consider the case of singular descriptor matrix $\TE$ such that \Cref{eq:LTI_DAE_Definition} describes a~\DAEsys\ (also called \emph{descriptor system}). While $\TGs\teq\TC(s\TE-\TA)^{-1}\TB+\TD$ denotes the \emph{transfer function} of \Cref{eq:LTI_DAE_Definition}, its specific \emph{realization} is abbreviated with $\Sigma\teq\RealizationD$.\par%
We assume that the matrix pencil $\pencilEA$ is \emph{regular}, i.\,e., $\det (\pencilEA)\neq 0$ for some $\lambda\tin\Complex$. In this case, there exist regular transformation matrices $\TP,\,\TQ \tin \Real^{n \times n}$ such that%
\begin{equation}%
    \label{eq:E_and_A_in_Weierstrass}%
    \TEW=\TP\,\TE\,\TQ=\MEW,\quad\TAW=\TP\TA\TQ=\MAW\;,%
\end{equation}%
where both $\TJ$ and $\TN$ are in real \JCF\ \cite{Gantmacher1959}. The matrix $\TN$ is nilpotent of index~$\nu$, i.\,e., $\TN^{\nu}\teq\T0$ and $\TN^{\nu-1}\neq\Tzero$. It determines the eigenvalue at infinity, whereas the eigenvalues of $\TJ$ are the finite eigenvalues of $\pencilEA$ . The identity matrices $\TIfin$ and $\TIinf$ are of dimension $n_{\finite}$ and $n_{\infinite}$ which denote the number of finite and infinite eigenvalues, respectively. The set of finite and infinite generalized eigenvalues of $\pencilEA$ is denoted by $\Lambda(\TE,\TA)$ or just by $\Lambda(\TA)$ if $\TE=\TI$. The representation \Cref{eq:E_and_A_in_Weierstrass} is called the \emph{\WCF}. Although $\TP$ and $\TQ$ will not be computed explicitly in the following, the Weierstra{\ss} canonical form is important for the theoretical analysis and derivations.\par%
Note that in the case of linear \DAEsys s, $\nu$ represents the \emph{differentiation index}~\cite{Mehrmann2006}, which is related to the structural complexity of the given problem. Though it will be used in the following considerations, e.g. to characterize different types of DAEs, an explicit knowledge of the value of $\nu$ is not required for -- nor computed during -- reduction.\par%
In the following, we restrict ourselves to the \emph{asymptotically stable} system \Cref{eq:LTI_DAE_Definition} for which all finite generalized eigenvalues of $\pencilEA$ have negative real part. Note that this is not restrictive, as unstable modes are always dominant and the reduction of unstable models generally requires the separation of stable and antistable part, reducing only the former and preserving the latter.\par%
A characteristic property of DAEs is that a \emph{polynomial contribution} to the transfer function $\TGs$ may be present. Let the pencil $\lambda \TE- \TA$ be in the Weierstra{\ss} canonical form \Cref{eq:E_and_A_in_Weierstrass} and let the matrices%
\begin{equation}%
    \label{eq:Transformation_into_WCF_all_matrices}%
    \TBW=\TP\,\TB=\MBW\;,\quad \TCW=\TC\,\TQ=\MCW%
\end{equation}%
be partitioned according to the block-diagonal form of $\TEW$ and $\TAW$. Then the transfer function $\TGs$ can additively be decomposed as%
\begin{equation}%
    \label{eq:DAE_transfer_function}%
    \TGs=\Transfer+\TD=\widetilde{\TC}(s\,\widetilde{\TE}-\widetilde{\TA})^{-1}\widetilde{\TB}+\TD = \TG^{\strictlyproper}(s) + \TP(s),%
\end{equation}%
where $\TG^{\strictlyproper}(s)=\TC_{f} (s\,\TIfin-\TJ)^{-1}\TB_{f}$ is the \emph{strictly proper} part and%
\begin{equation*}%
    \TP(s)=\TC_{\infty}(s\,\TN-\TIinf)^{-1}\TB_{\infty}+\TD= -\sum\limits_{k=0}^{\nu-1}\TC_{\infty}\TN^k\,\TB_{\infty}\,s^k +\TD%
\end{equation*}%
is the \emph{polynomial} part of $\TGs$.\par%
In many cases, the explicit computation of the \WCF\ \Cref{eq:E_and_A_in_Weierstrass} is not required. Instead, one can act separately on the strictly proper and polynomial parts of $\TGs$ by making use of the \emph{spectral projectors} \cite{Benner2017,Mehrmann2006}. The matrices%
\begin{equation}%
    \label{eq:Spectral_Projectors_finite}%
    \Specleftfin=\TP^{-1}\begin{bmatrix} \TIfin & \Tzero \\ \Tzero & \Tzero \end{bmatrix}\TP\qquad\mbox{and}\qquad\Specrightfin=\TQ\begin{bmatrix} \TIfin & \Tzero \\ \Tzero & \Tzero \end{bmatrix}\TQ^{-1}%
\end{equation}%
are called the spectral projectors onto the left and right deflating subspaces of $\pencilEA$ corresponding to the finite eigenvalues along the left and right deflating subspaces corresponding to the infinite eigenvalues. Similarly, the matrices%
\begin{equation}%
    \label{eq:Spectral_Projectors_infinite}%
    \Specleftinf=\TP^{-1}\begin{bmatrix} \Tzero & \Tzero \\ \Tzero & \TIinf \end{bmatrix}\TP\qquad\mbox{and}\qquad\Specrightinf=\TQ\begin{bmatrix} \Tzero & \Tzero \\ \Tzero & \TIinf \end{bmatrix}\TQ^{-1}%
\end{equation}%
are called the spectral projectors onto the left and right deflating subspaces of $\pencilEA$ corresponding to the infinite eigenvalues along the left and right deflating subspaces corresponding to the finite eigenvalues. Using these projectors, one can partition $\TGs=\TG^\strictlyproper(s)+\TPs$ with%
\begin{equation}%
    \label{eq:RealizationofstrictlyproperPart}%
    \TG^\strictlyproper(s)=\TC\Specrightfin(s\TE-\TA)^{-1}\TB=\TC(s\TE-\TA)^{-1}\Specleftfin\TB=\TC\Specrightfin(s\TE-\TA)^{-1}\Specleftfin\TB \quad%
\end{equation}%
for the strictly proper part and%
\begin{equation}%
\label{eq:RealizationofimproperPart}%
    \TPs=\TC\Specrightinf(s\TE-\TA)^{-1}\TB+\TD=\TC(s\TE-\TA)^{-1}\Specleftinf\TB+\TD=\TC\Specrightinf(s\TE-\TA)^{-1}\Specleftinf\TB+\TD%
\end{equation}%
for the polynomial part \cite{Gugercin2013}. Note that the numerical computation of the spectral projectors is in general ill-conditioned. However, for certain types of \DAEsys s, analytic expressions can be derived directly by exploiting the structure of the system matrices. Several examples, including semi-explicit systems of index 1, Stokes-like systems of index 2 and mechanical systems of index 1 and 3 are collected in \cite{Benner2017}. Since the spectral projectors are in general dense matrices, their explicit computation and storage should be avoided. Fortunately, they often inherit the block structure of the system matrices, therefore projector-vector products can be performed block-wise exploiting the sparsity of the system matrices and resulting in a more efficient implementation \cite{Benner2017}. Therefore, in the following we assume that the given system is of special structure and the respective spectral projectors are known.\par%
%
%

%% file: modelReduction.tex
\section{Model Reduction by Tangential Interpolation}%
\label{sec:MOR}%
In this contribution, we consider the reduction of a full-order model (FOM) given as a DAE system \Cref{eq:LTI_DAE_Definition} by means of tangential interpolation. Our goal is to design a reduced-order model (ROM)%
\begin{equation}%
    \TGr(s)=\Transferr + \TD_\rsys%
    \label{eq:ROM_general}%
\end{equation}%
with $\TEr,\TAr\tin\Real^{q\times q}$, $\TBr\tin\Real^{q\times m}$, $\TCr\tin\Real^{p\times q}$, and $\TDr\tin\Real^{p\times m}$ which satisfies the tangential interpolation conditions%
\begin{subequations}%
    \begin{align}%
        \TG(\sigma_i)\,\Tr_i &= \TGr(\sigma_i)\, \Tr_i\;, & & i = 1,\,\ppp\,,\,q,
        \label{eq:tangential_interpolation_conditions:1}\\
        \Tl_j^{\transpose}\,\TG(\mu_j) &= \Tl_j^{\transpose}\,\TGr(\mu_j)\;, & & j = 1,\,\ppp\,,\,q,
        \label{eq:tangential_interpolation_conditions:2}\\
        \Tl_i^{\transpose}\left(\left.\frac{\diff}{\diff s}\TG(s)\right|_{s=\sigma_i}\right)\Tr_i & =
        \Tl_i^{\transpose}\left(\left.\frac{\diff}{\diff s}\TGr(s)\right|_{s=\sigma_i}\right)\Tr_i, & & i = 1,\,\ppp\,,\,q,
        \label{eq:tangential_interpolation_conditions:3}%
    \end{align}%
    \label{eq:tangential_interpolation_conditions}%
\end{subequations}%
for complex frequencies $\sigma_i,\mu_j \tin\Complex $ as well as right and left tangential directions $\Tr_i\tin\Complex^{m}$ and \mbox{$\Tl_j\tin\Complex^{p}$}. A reduced model \Cref{eq:ROM_general} satisfying all three conditions in \Cref{eq:tangential_interpolation_conditions} is referred to as a \emph{bitangential Hermite interpolant}.\par%
The main advantage of the interpolatory approximation approach is the existence of numerically efficient methods which have been effectively used for reduction of large-scale systems, see \cite{Villemagne1987, Grimme1997, Antoulas2010, Gallivan2004TI, Gugercin2008, VanDooren2008, Beattie2009, Druskin2014, Panzer2013, Wolf2013} for ODEs and \cite{Gugercin2013,Ahmad2014,Benner2017} for the DAE case. Moreover, the interpolatory conditions allow to directly specify frequency regions in which the approximation quality should be higher. Further note that by means of Carleman bilinearization and quadratic bilinearization, tangential interpolation has also been extended to nonlinear models \cite{Benner2016,Ahmad2017,Goyal2015}. In the following, we discuss the tangential interpolation problem from two different -- though related -- perspectives: projective and non-projective. Both of them will play a role in \Cref{sec:pork}.\par%
\subsection{The Projective Framework}%
\label{subsec:MOR_projection}%
Traditionally, the model reduction problem is formulated in a projective framework, where the design parameters are projection matrices $\TV, \TW \tin \Real^{n\times q}$ spanning appropriate subspaces. The reduced-order model is obtained by approximating the state vector $\Tx\approx \widehat{\Tx} \teq \TV \Tx_\rsys$ and enforcing the Petrov-Galerkin condition $\TW^\transpose\left(\TE\TV\Txp_\rsys-\TA\TV\Txr-\TB \Tu \right) \teq 0$
in the form%
\begin{equation}%
    \TEr\Txp_\rsys = \TAr\Txr + \TBr\Tu, \qquad%
    \Ty_\rsys =\TCr\Txr + \TD_\rsys\Tu,%
    \label{eq:ROM_by_projection}%
\end{equation}%
where $\TEr\teq \TEre$, $\TAr \teq \TAre$, $\TBr\teq\TBre$, and $\TCr\teq\TCre$. In general, $\TD_\rsys\teq\TD$ is chosen and, hence, the reduction process can often be carried out without explicit consideration of $\TD$. Note that the choice $\TD_\rsys \teq \TD$ is necessary for bounded \HtwoText\  error but is in general non-optimal with respect to the $\Hinf$ error, see \cite{morFlaBG13,morCasBG17}.\par%
To enforce the tangential interpolation conditions \Cref{eq:tangential_interpolation_conditions}, the projection matrices can be constructed according to the following theorem.%
\begin{theorem}[\cite{Villemagne1987,Gallivan2004TI,Beattie2017}] \label{thm:tangential_interpolation}%
Consider a DAE \Cref{eq:LTI_DAE_Definition} with the transfer function $\TGs$. Let $\TGrs$ be the transfer function of the reduced-order model \Cref{eq:ROM_by_projection} obtained using the projection matrices $\TV$ and $\TW$. Furthermore, let $\sigma_i,\mu_j\tin\Complex$ be outside of the generalized spectra of the pencils $\pencilEA$ and $\pencilErAr$ and let $\Tr_i\tin \Complex^{m}$ and $\Tl_j\tin \Complex^{p}$ be nonzero vectors for $i,j=1,\,\ppp\,,\,q$.%
\begin{enumerate}%
    \item If%
          \begin{equation}\label{eq:input_Krylov}%
               (\TA-\sigma_i\TE)^{-1}\TB\,\Tr_i \in\Image(\TV), \hspace{2.5cm}i=1,\,\ppp\,,\,q\;,%
          \end{equation}%
        then $\TG(\sigma_i)\,\Tr_i \teq \TGr(\sigma_i)\,\Tr_i$ for $i=1,\,\ppp\,,\,q$.%
    \item If%
        \begin{equation}\label{eq:output_Krylov}%
            (\TA-\mu_j\TE)^{-\transpose}\TC^{\transpose}\,\Tl_j \in\Image(\TW), \hspace{2cm} j=1,\,\ppp\,,\,q\;,%
        \end{equation}%
        then $\Tl_j^{\transpose}\,\TG(\mu_j)\teq \Tl_j^{\transpose}\,\TGr(\mu_j)$ for $j=1,\,\ppp\,,\,q$.%
    \item If both \Cref{eq:input_Krylov} and \Cref{eq:output_Krylov} hold for $\sigma_i\teq\mu_i$, then, in addition,%
        \begin{equation*}%
            \Tl_i^{\transpose}\left(\left.\frac{\diff}{\diff s}\TG(s)\right|_{s=\sigma_i}\right)\Tr_i= \Tl_i^{\transpose}\left(\left.\frac{\diff}{\diff s}\TGr(s)\right|_{s=\sigma_i}\right)\Tr_i, \qquad i = 1,\,\ppp\,,\,q\;.%
        \end{equation*}%
\end{enumerate}%
\end{theorem}%
One of the beauties of this result is that it holds for both regular and singular matrices $\TE$ meaning that tangential interpolation can be achieved both for ODEs and DAEs by using \Cref{thm:tangential_interpolation}. %
In addition, necessary conditions for \HTOy\ can be derived in terms of tangential interpolation conditions as follows.%
\begin{theorem}[\cite{VanDooren2008}]\label{thm:H2optimality}%
Let a DAE system \Cref{eq:LTI_DAE_Definition} be given and let $\TGs$ denote its transfer function. Consider a~reduced-order model ${\TGrs\teq\sum\limits_{i=1}^{\ro}\frac{\Tc_{r,i}\Tb_{r,i}}{s-\lambda_{r,i}}}$ with distinct poles $\lambda_{r,i}\tin\Complex$ and input and output residual directions $\Tb_{r,i}^{\complexconjugatetranspose}\tin\Complex^{\is}$ and $\Tc_{r,i}\tin\Complex^{\os}$, respectively. If $\TGrs$ satisfies $\TGrs\teq\arg\min\limits_{\text{dim}(\widehat{\TG}_\rsys)=\ro}\norm{\TG-\widehat{\TG}_\rsys}{\Htwo}$ at least locally, then it holds%
\begin{subequations}%
    \begin{align}%
        \TG(-\complexconjugate{\lambda}_{r,i})\,\Tb_{r,i}^{\complexconjugatetranspose} &= \TGr(-\complexconjugate{\lambda}_{r,i})\,\Tb_{r,i}^{\complexconjugatetranspose}, \label{eq:H2OptimalityCondition_input}\\%
        \Tc_{r,i}^{\complexconjugatetranspose}\,\TG(-\complexconjugate{\lambda}_{r,i}) &= \Tc_{r,i}^{\complexconjugatetranspose}\,\TGr(-\complexconjugate{\lambda}_{r,i}), \label{eq:H2OptimalityCondition_output}\\%
        \Tc_{r,i}^{\complexconjugatetranspose}\left(\left.\frac{\diff}{\diff s} \TG(s)\right|_{s=-\complexconjugate{\lambda}_{r,i}}\right)\Tb_{r,i}^{\complexconjugatetranspose} &= \Tc_{r,i}^{\complexconjugatetranspose}\left(\left.\frac{\diff}{\diff s} \TGr(s)\right|_{s=-\complexconjugate{\lambda}_{r,i}}\right)\Tb_{r,i}^{\complexconjugatetranspose} \label{eq:H2OptimalityCondition_twosided}%
    \end{align}%
\end{subequations}%
for $i = 1,\,\ppp\,,\,q$.%
\end{theorem}%
Next, we demonstrate that also the duality between the projection matrices satisfying \Cref{eq:input_Krylov} and \Cref{eq:output_Krylov} and the solutions of particular generalized Sylvester equations holds, see \cite{Gallivan2004,morVan04,Wolf2014} for ODEs. To that end, we first introduce different representations for tangential interpolation data used throughout this contribution.%
\begin{definition}[Right tangential interpolation data]\label{def:right interpolation data}%
Given a DAE system \Cref{eq:LTI_DAE_Definition}, consider a set of distinct interpolation frequencies $\sigma_1,\ppp,\sigma_\ro\tin\Complex$, which do not belong to $\Lambda(\TE,\TA)$, and corresponding nonzero right tangential directions $\Tr_1,\,\ppp\,,\,\Tr_\ro\tin\Complex^{\is}$.
We refer to the set of pairs $\left\{\left(\sigma_i,\Tr_i\right)\right\}_{i=1}^{\ro}$ as the \emph{right tangential interpolation data}, as they define the right tangential interpolation conditions of the form \Cref{eq:tangential_interpolation_conditions:1}. An equivalent representation for the interpolation data is given by the matrices%
\begin{subequations}%
    \begin{align}%
        \TS^\prim &= \diag\left(\sigma_1,\,\ppp\,,\,\sigma_\ro\right)\in\Complex^{\ro\times\ro},\label{eq:Sp}\\%
        \TR ^{\prim}&=\left[\Tr_1,\,\ppp\,,\,\Tr_\ro\right] \in \Complex^{\is\times\ro}.\label{eq:Rp}%
    \end{align}%
    \label{eq:SpRp}%
\end{subequations}%
\end{definition}%
\begin{remark}%
All considerations in the following have a dual counterpart in terms of \emph{left tangential interpolation data} $\left\{\left(\mu_j,\Tl_j\right)\right\}_{j=1}^{\ro}$ with $\mu_1,\ppp,\mu_\ro\tin\Complex\setminus\Lambda(\TE,\TA)$ and $\Tl_1,\ppp,\Tl_\ro\tin\Complex^{\os}$. As the discussion is analogous, we omit it here for brevity.%
\end{remark}%
It will be convenient in the remainder of the paper to characterize the right tangential interpolation data not in terms of the special structured matrices in \Cref{eq:SpRp} but in terms of matrices $\TS$ and~$\TR$ that are connected to $\TS^\prim$ and $\TR^\prim$ by means of a~similarity transformation.%
\begin{definition}\label{thm:S}%
Given a DAE system \Cref{eq:LTI_DAE_Definition}, consider a diagonalizable matrix $\TS\tin\Complex^{\ro\times\ro}$ with distinct eigenvalues $\sigma_1,\,\ppp\,,\,\sigma_\ro$ satisfying $\Lambda(\TS)\cap\Lambda(\TE,\TA)\teq\emptyset$ and a matrix $\TR\tin\Complex^{\is\times\ro}$. Let $\TT\tin\Complex^{\ro\times\ro}$ be a~regular matrix such that $\TT \TS \TT^{-1} = \diag\left(\sigma_1,\,\ppp\,,\,\sigma_\ro\right)$. If all columns of $\TR \TT^{-1}\teq\left[\Tr_1,\,\ppp\,,\,\Tr_\ro\right]$ are nonzero, then the matrices $\TS$ and $\TR $ are an equivalent representation of the right interpolation data $\left\{\left(\sigma_i,\Tr_i\right)\right\}_{i=1}^{\ro}$ as in \Cref{def:right interpolation data}.%
\end{definition}%
%
%
%
%
%
\begin{remark}\label{remark:S_R_controllability}%
It can be shown that if $\TS$ and $\TR$ satisfy the assumptions in \Cref{thm:S}, then the pair $\left(-\TS^\complexconjugatetranspose,\TR ^{\complexconjugatetranspose}\right)$ is controllable, see \cite[Theorem~3.17, Corollary~3.18 and Theorem 4.33 part (iii)]{Seiwald2016}. This is worth mentioning at this point, as the results in \Cref{sec:pork} will often make use of this property.%
\end{remark}%
The representation of the right interpolation data in terms of the matrices $\TS$ and $\TR$ becomes of practical importance when parametrizing the corresponding projection matrix $\TV$ in terms of \emph{generalized Sylvester equations.}%
\begin{theorem}[Sylvester equation for right tangential interpolation of DAEs]%
\label{thm:Assembly_Krylov_to_Sylvester}%
Consider a DAE system \Cref{eq:LTI_DAE_Definition}. Let right tangential interpolation data $\left\{\left(\sigma_i,\Tr_i\right)\right\}_{i=1}^{\ro}$ as in \Cref{def:right interpolation data} be given. Then any projection matrix $\TV\tin\Complex^{\fo\times\ro}$ satisfying \Cref{eq:input_Krylov} solves the generalized sparse-dense Sylvester equation%
\begin{equation}%
    \TA \TV - \TE \TV \TS - \TB\TR = \Tzero \label{eq:Sylvester}%
\end{equation}%
with appropriate choice of $\TS\tin\Complex^{\ro\times\ro}$ and $\TR\tin\Complex^{\is\times\ro}$.%
\end{theorem}%
\begin{proof}%
The proof is obtained by construction. Let%
\begin{equation*}%
    \TV_i^\prim = (\TA-\sigma_i\TE)^{-1}\TB\,\Tr_i, \quad i = 1,\ppp,\ro.%
\end{equation*}%
Then we obtain from%
\begin{equation*}%
    \TA \TV_i^\prim -\sigma_i\TE \TV_i^\prim = \TB\,\Tr_i, \quad i = 1,\ppp,\ro,%
\end{equation*}%
that the so-called \emph{primitive basis} $\TV^{\prim}\teq\left[\TV_1^\prim,\,\ppp\,,\,\TV_\ro^{\prim}\right]$ satisfies%
\begin{equation*}%
\TA \TV^{\prim} - \TE \TV^{\prim} \TS^\prim - \TB \TR^\prim = \Tzero%
\end{equation*}%
with $\TS^\prim$ and $\TR^\prim$ as in \Cref{eq:SpRp}. Any other basis of the same subspace can be obtained through transformation $\TV \teq \TV^{\prim}\TT$ with a~regular matrix $\TT\tin\Complex^{\ro\times\ro}$, which ultimately results in the Sylvester equation \Cref{eq:Sylvester} with $\TS\teq\TT^{-1}\TS^\prim\TT$ and $\TR\teq\TR^\prim\TT$.%
%
%
%
%
\end{proof}%
A converse relation also holds true underlying the strong relationship between the solution of \Cref{eq:Sylvester} and the projection matrices satisfying \Cref{eq:input_Krylov}.%
\begin{theorem}%
Consider a DAE system \Cref{eq:LTI_DAE_Definition}. Let right tangential interpolation data $\left\{\left(\sigma_i,\Tr_i\right)\right\}_{i=1}^{\ro}$ be given in form of matrices $\TS$ and $\TR$ satisfying the assumptions in \Cref{thm:S}. Furthermore, let the reduced-order model \Cref{eq:ROM_by_projection} result from a projection with a~matrix $\TV\tin\Complex^{\fo\times\ro}$ solving the Sylvester equation \Cref{eq:Sylvester} and a~matrix $\TW\tin\Complex^{\fo\times\ro}$ chosen such that $\Lambda(\TS)\cap\Lambda(\TEr,\TAr)\teq\emptyset$. Then the reduced-order model \Cref{eq:ROM_by_projection} satisfies the right interpolation conditions \Cref{eq:tangential_interpolation_conditions:1}.%
\end{theorem}%
\begin{proof}%
Due to the condition ${\Lambda(\TS)\cap\Lambda(\TE,\TA)\teq\emptyset}$, the solution of \Cref{eq:Sylvester} exists and is unique \cite{Chu1987}. By \Cref{thm:S} there exists a nonsingular matrix $\TT$ such that $\TT \TS \TT^{-1} \teq \TS^\prim \teq \diag(\sigma_1,\ppp,\sigma_\ro)$ and $\TR\TT^{-1}\teq\TR ^\prim \teq \begin{bmatrix}\Tr_1,\ppp,\Tr_\ro\end{bmatrix}$. Therefore, \Cref{eq:Sylvester} is equivalent to%
\begin{equation*}%
    \TA \TV^\prim - \TE \TV^\prim \TS^\prim - \TB\TR^\prim = \Tzero,
\end{equation*}%
where $\TV ^\prim \teq \TV\TT^{-1}$. Then each column of $\TV^\prim$ has the form%
\begin{equation*}%
    \TV_i^\prim = (\TA-\sigma_i\TE)^{-1}\TB\,\Tr_i, \quad i = 1,\ppp,\ro.%
\end{equation*}%
Since $\TV_i^\prim\in \Image(\TV)$, it follows from \Cref{thm:tangential_interpolation} that the right tangential interpolation conditions \Cref{eq:tangential_interpolation_conditions:1} are fulfilled.%
\end{proof}%
Note that the interpolation conditions \Cref{eq:tangential_interpolation_conditions} can be extended to tangentially match also high order derivatives of $\TGs$ by spanning appropriate input/output tangential Krylov subspaces \cite{Villemagne1987,Grimme1997,Gallivan2004,Beattie2017}. Also in this case, the duality to a~respective Sylvester equation holds. For the sake of brevity and simpler notation, we do not present the results in this most general form but refer the interested reader to the more detailed discussion in \cite{Seiwald2016}.\par%
The results of \Cref{thm:tangential_interpolation} may mislead to think that there is no difference in approximating ODEs and DAEs by tangential interpolation. As discussed in \cite{Gugercin2013,Benner2017}, achieving tangential interpolation for DAEs may in fact not be enough to obtain acceptable approximations. Indeed, interpolation by itself cannot in general guarantee that the polynomial part $\TPs$ of the transfer function $\TG(s)$ is matched exactly by the reduced-order model. This may lead to a nonzero or even unbounded error at $s\tto\infty$, and hence an error system which is not in $\Htwo$, see, e.g., the example in \cite{Gugercin2013}.\par%
For this reason, a correct reduction of DAEs should always ensure $\TPs\teq\TP_\rsys(s)$. This can be achieved by decomposing the reduction problem into two subproblems: tangential interpolation of the strictly proper part $\TG^{\strictlyproper}(s)$ given in \Cref{eq:RealizationofstrictlyproperPart} and preservation of the polynomial part $\TPs$ given in \Cref{eq:RealizationofimproperPart}, as illustrated in \Cref{fig:overall_procedure}. To this aim, similarly to \cite{Gugercin2013}, we make use of the spectral projectors \Cref{eq:Spectral_Projectors_finite} and \Cref{eq:Spectral_Projectors_infinite} to act separately on realizations of the strictly proper part $\TG^\strictlyproper (s)$ and the polynomial part $\TPs$. As shown in \cite[Theorem 3.1]{Gugercin2013}, tangential interpolation of the strictly proper part can be achieved through projection of the basis matrices $\TV$ and $\TW$ as in \Cref{eq:input_Krylov} and \Cref{eq:output_Krylov} according to $\TV^{\strictlyproper} \teq \Specrightfin\TV$ and $\TW^{\strictlyproper} \teq (\Specleftfin)^{\transpose}\TW$.\par
\begin{figure}[!ht]%
    \centering%
    \input{generalframework.tex}%
    \caption{Sketch of DAE-aware model reduction. The spectral projectors are used to act on the strictly proper and polynomial parts separately. While tangential interpolation is used to approximate $\TG^\strictlyproper (s)$, balanced truncation is used to determine a minimal realization of $\TPs$. In a~final step, both subsystems are combined to the overall reduced model $\TGr(s)$. The interpolation frequency is indicated by a circle.\ToEditor{layout=spanning two-columns preferred; online-version=colored; print-version=grayscale}}%
    \label{fig:overall_procedure}%
\end{figure}
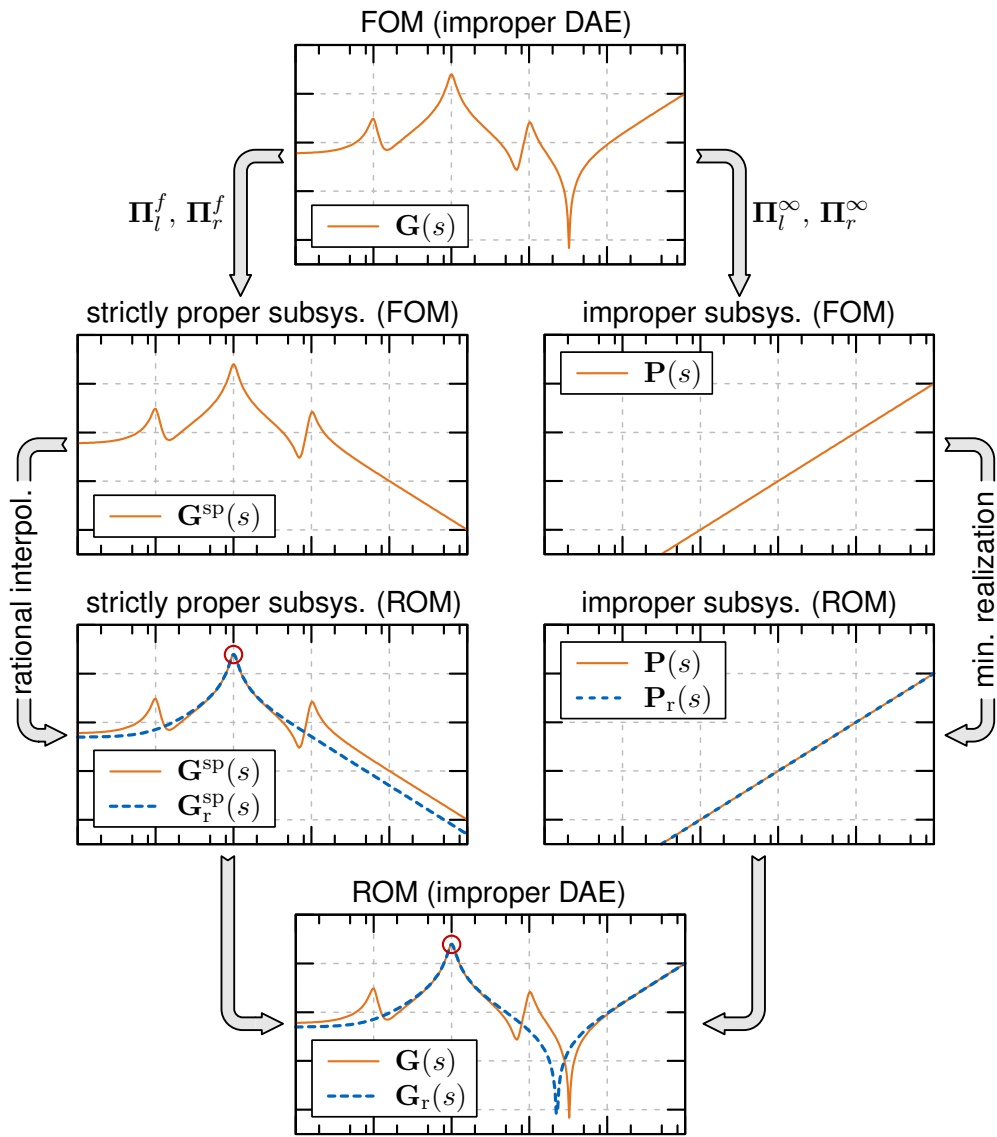%
In this contribution, we will derive a pseudo-optimal rational Krylov algorithm \cite{Wolf2013,Wolf2014} to approximate the strictly proper part $\TG^\strictlyproper (s)$ by $\TGr^\strictlyproper(s)$ with a~realization $(\TEr^\strictlyproper,\,\TAr^\strictlyproper,\,\TBr^\strictlyproper,\TCr^\strictlyproper,\Tzero)$, see \Cref{sec:pork}. For the polynomial part $\TPs$, the balanced truncation approach proposed in \cite{Stykel2004} can be applied to find its minimal realization $(\TEr^\improper,\,\TAr^\improper,\,\TBr^\improper,\,\TCr^\improper,\,\TDr^\improper)$ by projection%
\begin{equation*}%
    \TEr^\improper = (\TW^\improper)^{\transpose}\TE\TV^\improper,\;%
    \TAr^\improper = (\TW^\improper)^{\transpose}\TA\TV^\improper,\;%
    \TBr^\improper = (\TW^\improper)^{\transpose}\TB,\;%
    \TCr^\improper = \TC\TV^\improper,\;%
    \TDr^\improper = \TD.%
\end{equation*}%
The computation of the projection matrices $\TW^\improper$ and $\TV^\improper$ involves the numerical solution of two projected discrete-time Lyapunov equations. Exploiting characteristic properties of DAEs, the solutions of these equations can be determined by solving at most $(m+p)\nu$ (generally sparse) linear systems \cite{Stykel2008}. Note that there exists a~direct relation between the index $\nu$ of the DAE system \Cref{eq:LTI_DAE_Definition} and the dimension of the minimal realization of the polynomial part given by ${\dim (\TEr^\improper)\leq \min\lbrace\nu\,m,\,\nu\,p,\,n_\infinite\rbrace}$, see \cite{Stykel2004}. Since $n_\infinite$ is often rather high, while $\nu\, m$ and $\nu\, p$ are small in most technical applications, this approach can already achieve a significant reduction of the system dimension without any approximation error. Finally, the overall reduced-order model $\TGrs\teq\TGr^\strictlyproper (s)+\TP(s)$ is then obtained by
\begin{equation*}%
    \TEr = \begin{bmatrix} \TEr^\strictlyproper & \Tzero \\ \Tzero & \TEr^\improper\end{bmatrix},\;%
    \TAr = \begin{bmatrix} \TAr^\strictlyproper & \Tzero \\ \Tzero & \TAr^\improper\end{bmatrix},\;%
    \TBr = \begin{bmatrix} \TBr^\strictlyproper \\ \TBr^\improper\end{bmatrix},\;%
    \TCr = \begin{bmatrix} \TCr^\strictlyproper &  \TCr^\improper\end{bmatrix},\;%
    \TDr = \TDr^\improper.%
\end{equation*}%
Within a projective framework, this is equivalent to composing the overall projection matrices $\TV \teq \begin{bmatrix}\TV^\strictlyproper & \TV^\improper\end{bmatrix}$ and $\TW \teq \begin{bmatrix}\TW^\strictlyproper & \TW^\improper\end{bmatrix}$ as proposed in \cite{Gugercin2013}.%
\subsection{A Non-Projective Framework}%
\label{subsec:MOR_non-projective}%
For later discussions, it is important to notice that for ODEs the problem of tangential interpolation can be tackled from a non-projective perspective as well, see \cite{Wolf2014,Astolfi2007,Astolfi2010Mar,Astolfi2010Dec}. This can be done by introducing families of reduced transfer functions.%
\begin{definition}%
\label{dfn:Familiy_of_reduced_transfer_functions}%
Given a~DAE system \Cref{eq:LTI_DAE_Definition} and right tangential interpolation data $\TS\tin\Complex^{\ro\times\ro}$ and ${\TR\tin\Complex^{\is\times\ro}}$ as in \Cref{thm:S}, let $\TV$ solve the Sylvester equation \Cref{eq:Sylvester} and let $\TF\tin\Complex^{q\times m}$ satisfy%
\begin{equation}%
    \Lambda(\TS)\cap\Lambda(\TS+\TF\TR)\teq\emptyset.%
    \label{eq:condition_for_F}%
\end{equation}%
Then a \emph{family of reduced transfer functions} $\TGFs$ is defined as $\TGFs\teq\TransferF$ with%
\begin{equation}%
    \label{eq:Familiy_of_reduced_transfer_functions_F_system matrices}%
    \TEF=\TI_q\;,\qquad\TAF=\TS+\TF\,\TR\;,\qquad\TBF=\TF\;,\qquad\TCF=\TC\,\TV\;,%
\end{equation}%
and free parameter matrix $\TF$.%
\end{definition}%
Note that the reduced model \Cref{eq:Familiy_of_reduced_transfer_functions_F_system matrices} is not computed through the Pretrov-Galerkin projection as in \Cref{eq:ROM_by_projection}. In fact, the existence of a corresponding matrix $\TW$, such that a projection according to \Cref{eq:ROM_by_projection} leads to \Cref{eq:Familiy_of_reduced_transfer_functions_F_system matrices}, is not required. The following theorem establishes that for arbitrary choice of $\TF$ satisfying \Cref{eq:condition_for_F}, $\TGFs$ is a right tangential interpolant of $\TGs$ with respect to the interpolation data $\TS$ and $\TR$. It can be proved analogously to the ODE case \cite{Astolfi2007,Wolf2014}.%
\begin{theorem}%
\label{thm:parametrizedFamily_tangential_interpolation}%
Consider a~DAE system \Cref{eq:LTI_DAE_Definition} with transfer function $\TGs$ and let right tangential interpolation data $\TS\tin\Complex^{\ro\times\ro}$, $\TR\tin\Complex^{\is\times\ro}$ be given as in \Cref{thm:S}. For any matrix $\TF\tin\Complex^{q\times m}$ satisfying \Cref{eq:condition_for_F}, let $\TGFs$ denote the family of reduced transfer functions according to \Cref{dfn:Familiy_of_reduced_transfer_functions}. Then%
\begin{equation}%
    \TG(\sigma_i)\Tr_i = \TGF(\sigma_i)\Tr_i\qquad i=1,\,\ppp\,,\,q\;,%
    \label{eq:Tangential interpolation of SysF}%
\end{equation}%
for arbitrary $\TF$.%
\end{theorem}%
%
%

%% file: generalframework.tex
\begingroup%
  \makeatletter%
  \providecommand\color[2][]{%
    \errmessage{(Inkscape) Color is used for the text in Inkscape, but the package 'color.sty' is not loaded}%
    \renewcommand\color[2][]{}%
  }%
  \providecommand\transparent[1]{%
    \errmessage{(Inkscape) Transparency is used (non-zero) for the text in Inkscape, but the package 'transparent.sty' is not loaded}%
    \renewcommand\transparent[1]{}%
  }%
  \providecommand\rotatebox[2]{#2}%
  \ifx\svgwidth\undefined%
    \setlength{\unitlength}{382.67716535bp}%
    \ifx\svgscale\undefined%
      \relax%
    \else%
      \setlength{\unitlength}{\unitlength * \real{\svgscale}}%
    \fi%
  \else%
    \setlength{\unitlength}{\svgwidth}%
  \fi%
  \global\let\svgwidth\undefined%
  \global\let\svgscale\undefined%
  \makeatother%
  \begin{picture}(1,1.12592593)%
    \put(0,0){\includegraphics[width=\unitlength,page=1]{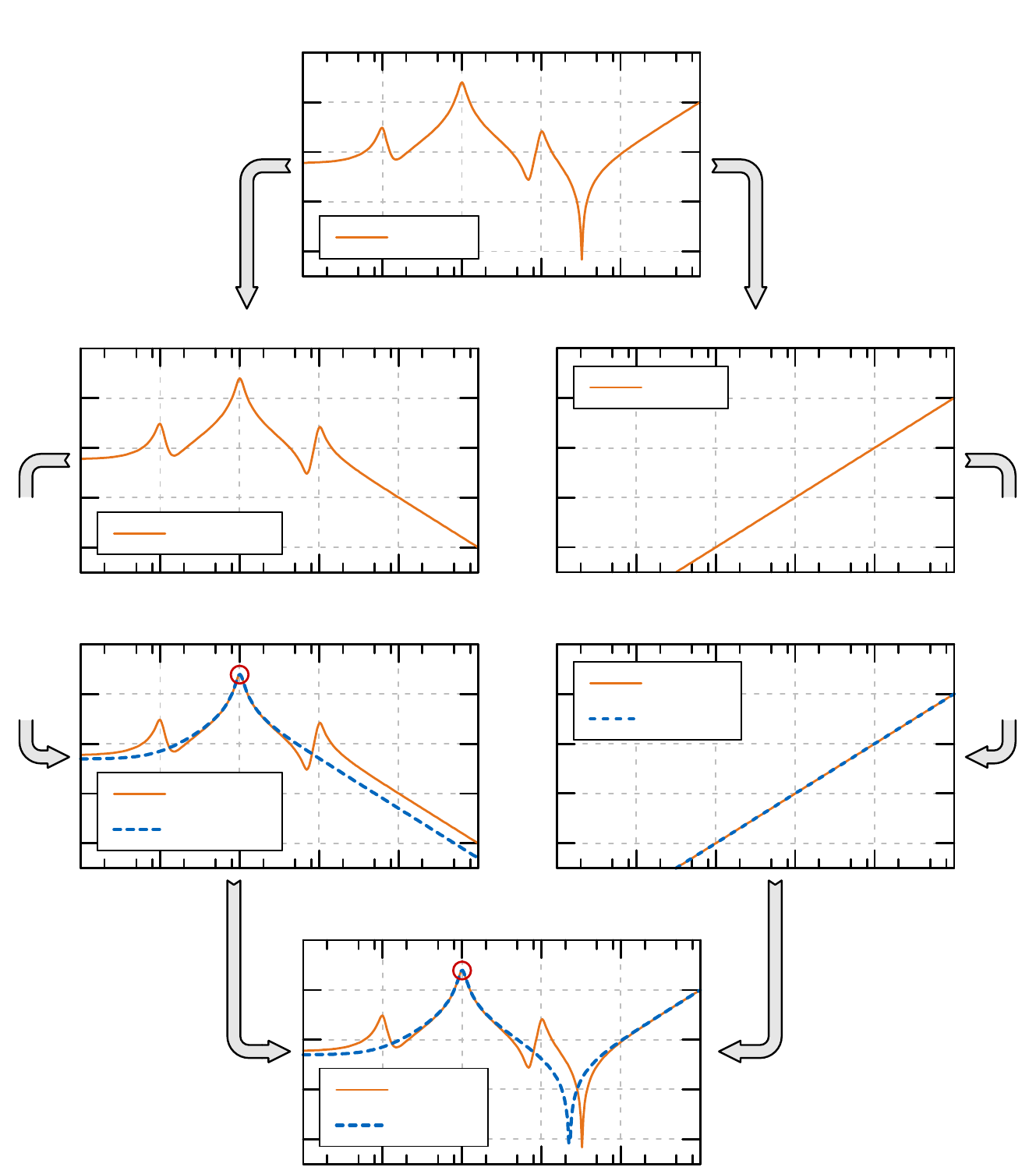}}%
    \put(0.35597091,1.08834218){\color[rgb]{0,0,0}\makebox(0,0)[lb]{\smash{FOM (improper DAE)}}}%
    \put(0.74214001,0.90102057){\color[rgb]{0,0,0}\makebox(0,0)[lb]{\smash{$\Specleftinf,\,\Specrightinf$}}}%
    \put(0.57490931,0.80269907){\color[rgb]{0,0,0}\makebox(0,0)[lb]{\smash{improper subsys. (FOM)}}}%
    \put(0.12752265,0.90102057){\color[rgb]{0,0,0}\makebox(0,0)[lb]{\smash{$\Specleftfin,\,\Specrightfin$}}}%
    \put(0.08648475,0.80257147){\color[rgb]{0,0,0}\makebox(0,0)[lb]{\smash{strictly proper subsys. (FOM)}}}%
    \put(0.34772112,0.2312591){\color[rgb]{0,0,0}\makebox(0,0)[lb]{\smash{ROM (improper DAE)}}}%
    \put(0.9797757,0.44189596){\color[rgb]{0,0,0}\rotatebox{90}{\makebox(0,0)[lb]{\smash{min. realization}}}}%
    \put(0.03219921,0.43429525){\color[rgb]{0,0,0}\rotatebox{90}{\makebox(0,0)[lb]{\smash{rational interpol.}}}}%
    \put(0.57394092,0.517012){\color[rgb]{0,0,0}\makebox(0,0)[lb]{\smash{improper subsys. (ROM)}}}%
    \put(0.08627567,0.51688441){\color[rgb]{0,0,0}\makebox(0,0)[lb]{\smash{strictly proper subsys. (ROM)}}}%
    \put(0.39007896,0.88802285){\makebox(0,0)[lb]{\smash{$\TGs$}}}%
    \put(0.63532159,0.7432818){\makebox(0,0)[lb]{\smash{$\TP(s)$}}}%
    \put(0.17554153,0.60225212){\makebox(0,0)[lb]{\smash{$\TG^{\strictlyproper}(s)$}}}%
    \put(0.39030293,0.06522458){\makebox(0,0)[lb]{\smash{$\TG(s)$}}}%
    \put(0.39030293,0.0309398){\makebox(0,0)[lb]{\smash{$\TGr(s)$}}}%
    \put(0.63532159,0.45759468){\makebox(0,0)[lb]{\smash{$\TP(s)$}}}%
    \put(0.63532159,0.4233099){\makebox(0,0)[lb]{\smash{$\TP_\rsys(s)$}}}%
    \put(0.17554153,0.35084985){\makebox(0,0)[lb]{\smash{$\TG^{\strictlyproper}(s)$}}}%
    \put(0.17554153,0.31656507){\makebox(0,0)[lb]{\smash{$\TGr^{\strictlyproper}(s)$}}}%
  \end{picture}%
\endgroup%

%% file: H2innerproduct.tex
\section{The \HtwoText\ Inner Product of DAEs}%
\label{sec:H2innerproduct}%
In this section, we present a formulation of the \HTIP\ of two strictly proper transfer functions in terms of the system matrices of their DAE realizations and the solution of \gSE s. Note that these results are not restricted to model reduction and may be used in different contexts. Although similar relations have been investigated in \cite{Stykel2006} and \cite{Wolf2014}, the formulation of the \HTIP\ of two strictly proper DAEs via \gSE s appears to be new. Its main advantage lies in the direct relationship between the inner product and the system matrices and will be the key for the results of \cref{sec:pork}.\par%
As a first step, we define $\Ltwopm$ as the Hilbert space of matrix-valued functions $\TF: \imath\,\Real \tto \Complex^{p \times m}$ that have bounded $\Ltwopm$-norm%
\begin{equation}%
    \normLtwopm{\TF}\assign\left(\frac{1}{2\,\pi}\int_{-\infty}^{\infty}\normFrobenius{\TF(\imath\,\omega)}^2 \,\diff \omega\right)^{\frac{1}{2}}\;,%
    \label{eq:L2-Definition}%
\end{equation}%
where $\normFrobenius{\cdot}$ denotes the Frobenius matrix norm. Furthermore, let $\Htwo$ denote the subspace of $\Ltwopm$ containing all rational functions that are analytic in the closed right half of the complex plane. It is important to note that a~rational function needs to be strictly proper in order to be in $\Htwo$, as otherwise the integral in \cref{eq:L2-Definition} is unbounded. This is why in the DAE context, the discussion must be restricted to the strictly proper transfer functions only. Also note that {asymptotic stability is a sufficient but not necessary condition for a DAE system to have the transfer function in $\Htwo$ \cite{Stykel2006}.\par%
Next we consider two asymptotically stable \DAEsys s with the strictly proper transfer functions $\TG(s)=\TC(s\TE-\TA)^{-1}\TB$ and $\TGH(s)=\TCH(s\TEH-\TAH)^{-1}\TBH$ and the impulse responses $\TGtimedomain(t)$ and $\TGHtimedomain(t)$ defined as inverse Laplace transforms of $\TG(s)$ and $\TGH(s)$, respectively. Then $\TG, \TGH\in\Htwo$ and the \HTIP\ of $\TG$ and $\TGH$ is defined as%
\begin{equation*}
    \innerproductHtwo{\TG}{\TGH}\assign\frac{1}{2\,\pi}\int_{-\infty}^{\infty}\trace\left(     \TG(\imath\,\omega)\,\TGH^{\complexconjugatetranspose}(\imath\,\omega) \right) \diff\omega     = \int_{0}^{\infty}\trace\left(\TGtimedomain(t)\,\TGHtimedomain(t)^{\complexconjugatetranspose}\right) \dt\;.
\end{equation*}%
The second relation follows immediately from Parseval's theorem \cite{Rudin1987}. Note that we use the notation in \WCF\ in order to exploit the strict properness of $\TGs$ and $\TGHs$. The following theorem is one of the main results of this contribution. It establishes a~formulation of the \HTIP\ of two transfer functions in terms of their DAE realizations.%
\begin{theorem}\label{thm:GenSylEq_for_X_and_Y}%
Consider two asymptotically stable DAE systems with the same number of input and output variables. Let their realizations be $\Realization$ and $\RealizationH$, respectively, and let their transfer functions $\TGs$ and $\TGHs$ be strictly proper. Furthermore, let $\Specleftfin,\,\Specrightfin$ and $\SpecHleftfin,\,\SpecHrightfin$ be the spectral projectors onto the left and right deflating subspaces of $\pencilEA$ and $\pencilEHAH$, respectively, corresponding to the finite eigenvalues. Then the \HTIP\ of $\TG$ and $\TGH$ is given by%
\begin{equation}%
    \label{eq:GenSylEq_innerProduct_X_and_Y}%
    \innerproductHtwo{\TG}{\TGH} = \trace \left(\TC\,\TX\,\TCH^{\complexconjugatetranspose}\right)=\trace \left(\TB^{\complexconjugatetranspose}\,\TY\,\TBH\right)\;,%
\end{equation}%
where $\TX$ and $\TY$ are the unique solutions of the projected Sylvester equations%
\begin{subequations}%
    \begin{align}%
        \label{eq:GenSylEq_for_X_DAEDAE}%
        &\TA\,\TX\,\TEH^{\complexconjugatetranspose}+\TE\,\TX\,\TAH^{\complexconjugatetranspose}+\Specleftfin\,\TB\,\TBH^{\complexconjugatetranspose}\,\SpecHleftfinCT = \Tzero\;, &\TX=\Specrightfin\,\TX\,\SpecHrightfinCT\;,\\%
        \label{eq:GenSylEq_for_Y_DAEDAE}%
        &\TA^{\complexconjugatetranspose}\,\TY\,\TEH+\TE^{\complexconjugatetranspose}\,\TY\,\TAH+\SpecrightfinCT\,\TC^{\complexconjugatetranspose}\,\TCH\,\SpecHrightfin = \Tzero\;,%
        &\TY=\SpecleftfinCT\,\TY\,\SpecHleftfin\;.%
    \end{align}%
    \label{eq:GenSylEq_for_DAEDAE}%
\end{subequations}%
\end{theorem}%
\begin{proof}%
Let the pencils $\pencilEA$ and $\pencilEHAH$ be transformed into the Weierstra{\ss} canonical form \Cref {eq:E_and_A_in_Weierstrass} and%
\begin{equation}%
    \label{eq:E_and_A_in_WeierstrassH}%
    \TPH\,\TEH\,\TQH=\begin{bmatrix} \TIHfin & \Tzero \\ \Tzero & \TNH\end{bmatrix}\;, \quad \TPH\,\TAH\,\TQH=\begin{bmatrix} \TJH & \Tzero \\ \Tzero & \TIHinf\end{bmatrix}\;,%
\end{equation}%
respectively, and let the input and output matrices be transformed as in \Cref{eq:Transformation_into_WCF_all_matrices} and%
\begin{equation}%
    \label{eq:Transformation_into_WCF_all_matricesH}%
    \TPH\,\TBH=\begin{bmatrix} \TBHfin \\ \TBHinf \end{bmatrix}\;,\quad \TCH\,\TQH=\begin{bmatrix} \TCHfin &\TCHinf \end{bmatrix}.%
\end{equation}%
Partitioning the matrix%
\begin{equation*}%
\TQ^{-1}\TX\TQH^{-*} = \begin{bmatrix} \TXff & \TXfi \\ \TXif & \TXii \end{bmatrix}%
\end{equation*}%
such that $\TXff\in\mathbb{C}^{n_f\times n_{\Hsys f}}$ and $\TXii\in\mathbb{C}^{n_{\infty}\times n_{\Hsys \infty}}$, the first equation in \Cref{eq:GenSylEq_for_X_DAEDAE} can be decoupled as%
\begin{subequations}%
    \begin{align}%
    \label{eq:Proof_GenSylEq_for_X_in_Weierstrass_Xff}%
    \TJ\,\TXff+\TXff\,\TJ^{\complexconjugatetranspose}_\Hsys+\TBfin\,\TB^{\complexconjugatetranspose}_{\Hsys \finite}&=\Tzero\;,\\%
    \label{eq:Proof_GenSylEq_for_X_in_Weierstrass_Xif}%
    \TXif+\TN\,\TXif\,\TJ^{\complexconjugatetranspose}_\Hsys&=\Tzero\;,\\%
    \label{eq:Proof_GenSylEq_for_X_in_Weierstrass_Xfi}%
    \TJ\,\TXfi\,\TN^{\complexconjugatetranspose}_\Hsys+\TXfi&=\Tzero\;,\\%
    \label{eq:Proof_GenSylEq_for_X_in_Weierstrass_Xii}%
    \TXii\,\TN^{\complexconjugatetranspose}_\Hsys+\TN\,\TXii&=\Tzero\;.%
    \end{align}%
    \label{eq:Proof_GenSylEq_for_X_in_Weierstrass}%
\end{subequations}%
Since both DAE systems are asymptotically stable, the Sylvester equation \Cref{eq:Proof_GenSylEq_for_X_in_Weierstrass_Xff} has a unique solution given by%
\begin{equation*}%
    \TXff\teq\int_0^{\infty} e^{\TJ\,t}\, \TBfin\,\TB^{\complexconjugatetranspose}_{\Hsys \finite}\,e^{\TJ^{\complexconjugatetranspose}_\Hsys\,t}  \dt.%
\end{equation*}%
The second equation in \Cref{eq:GenSylEq_for_X_DAEDAE} implies $\TXif\teq\Tzero$, $\TXfi\teq\Tzero$ and  $\TXii\teq\Tzero$, which obviously satisfy equations \Cref{eq:Proof_GenSylEq_for_X_in_Weierstrass_Xif}-\Cref{eq:Proof_GenSylEq_for_X_in_Weierstrass_Xii}.\par%
Furthermore, using again \Cref {eq:E_and_A_in_Weierstrass}, \Cref {eq:E_and_A_in_WeierstrassH}, \Cref{eq:Transformation_into_WCF_all_matrices}, \Cref{eq:Transformation_into_WCF_all_matricesH} and exploiting the strict properness of $\TGs$ and $\TGHs$, the impulse responses can be represented as%
\begin{equation*}%
    \TGtimedomain(t)=\TCfin\,e^{\TJ\,t}\,\TBfin\;,\qquad \TGHtimedomain(t)=\TCHfin\,e^{\TJ_\Hsys\,t}\,\TBHfin\;, \qquad t\geq 0\;.%
\end{equation*}%
Then we have%
\begin{equation*}%
    \begin{aligned}%
        \innerproductHtwo{\TG}{\TGH}&=\int_{0}^{\infty} \trace \left( \TCfin\,e^{\TJ\,t}\,\TBfin\,\TBHfin^\complexconjugatetranspose\,e^{\TJ_\Hsys^\complexconjugatetranspose\,t}\,\TCHfin^\complexconjugatetranspose \right)\dt\\%
        & = \trace \left(\TCfin \Bigl( \int_{0}^{\infty} e^{\TJ\,t}\,\TBfin\,\TBHfin^\complexconjugatetranspose\,e^{\TJ_\Hsys^\complexconjugatetranspose\,t}\dt\Bigr)\TCHfin^\complexconjugatetranspose \right)\\%
        & = \trace \Bigl(\begin{bmatrix} \TCfin &\TCinf \end{bmatrix} \begin{bmatrix} \TXff & \Tzero \\ \Tzero & \Tzero \end{bmatrix}\begin{bmatrix} \TCHfin^\complexconjugatetranspose \\ \TCHinf^\complexconjugatetranspose \end{bmatrix}\Bigr)= \trace \left( \TC\, \TX\, \TCH^\complexconjugatetranspose\right).%
    \end{aligned}%
\end{equation*}%
Thus, \Cref{eq:GenSylEq_for_X_DAEDAE} holds. Similar considerations lead to the dual result \Cref{eq:GenSylEq_for_Y_DAEDAE}.%
\end{proof}%
\begin{remark}%
A special case covered by \cref{thm:GenSylEq_for_X_and_Y} and of particular importance in \cref{sec:pork} is given when only one of the two systems involved in the inner product is a DAE, e.g., if $\TEH$ is nonsingular. In this case, the additional constraints in \Cref{eq:GenSylEq_for_X_DAEDAE} and \Cref{eq:GenSylEq_for_Y_DAEDAE} on $\TX$ and $\TY$ to ensure uniqueness of the solution are trivially satisfied.%
\end{remark}%
It is worth noting that $\TX$ and $\TY$ coincide with the proper controllability and observability Gramians of a DAE for the special case $\Realization\teq\RealizationH$. This shows that \eqref{eq:GenSylEq_for_X_DAEDAE} is a generalization of the results in \cite{Stykel2006}, where the squared \HTN\ of $\TGs$ reads%
\begin{equation}%
    \normHtwo{\TG}^2=\innerproductHtwo{\TG}{\TG}=\trace(\TC\,\TX\,\TC^{\complexconjugatetranspose})=\trace(\TB^{\complexconjugatetranspose}\,\TY\,\TB)\;.%
\end{equation}%
To conclude this section, finally note that the \HTIP\ of two strictly proper transfer functions $\TGs$ and $\TGMs$ can also be characterized in terms of the pole-residue representation of $\TGMs$ by applying the residue theorem.%
\begin{theorem}[\cite{Wolf2014}] \label{thm:Description_of_the_Htwo_inner_product_via_moments}%
Consider two DAE systems with the transfer functions $\TG,\TGM\tin\Htwo$. Let $\TGMs \teq \sum\limits_{i=1}^{\ro}\frac{\Tc_{\Msys i}\Tb_{\Msys i}}{s-\lambda_{Mi}}$ be the pole-residue representation, where $\lambda_{Mi}$ denote the poles of $\TGMs$ and $\,\Tb_{\Msys i}^\complexconjugatetranspose$ and $\Tc_{\Msys i}$ denote the corresponding input and output residual directions. Then the \HTIP\ of $\TG$ and $\TGM$ is given by%
\begin{equation}%
    \label{eq:Htwo_inner_product_via_moments_summation}%
    \innerproductHtwo{\TG}{\TGM}=\sum_{i=1}^q \trace\left(\TG(-\complexconjugate{\lambda}_{\Msys i})\,\Tb_{\Msys i}^{\complexconjugatetranspose}\,\Tc_{\Msys i}^\complexconjugatetranspose\right)\;. %
\end{equation}%
\end{theorem}%
This result can be extended to systems with multiple poles as well. As the notation becomes more involved, we refer to \cite{Wolf2014,Seiwald2016} for details.%
%
%

%% file: pork.tex
\section{\HtwoText\ Pseudo-Optimal Reduction of DAEs}%
\label{sec:pork}%
Using the results of \Cref{sec:H2innerproduct}, it is possible to extend the concept of \HTPO\space reduction and the related algorithms to DAEs. This type of reduction is strongly related to rational interpolation and \HTO\ reduction as introduced in \Cref{sec:MOR}. The resulting ROM will interpolate the FOM at the mirror images of the reduced poles tangentially along the reduced input residues. As we will discuss, these are necessary and sufficient conditions for the \emph{global} optimum within a subspace of $\Htwo$ defined by the set of reduced poles and input residues. Optimality within such a~subspace is what motivates the name ``pseudo''. As the original work \cite{Wolf2014} requires the regularity of $\TE$, several steps have to be adapted to deal with a singular descriptor matrix $\TE$. We start our discussion by introducing the notion of \HTPOy.%
\subsection{\HtwoText\ Pseudo-Optimality}%
\label{sec:pork_Htwo_Pseudo_optimality}%
In contrast to \HTO\ reduction \cite{Gugercin2008,Gugercin2013}, which is aimed at finding a \emph{locally} optimal reduced-order model amongst all reduced models of prescribed order $\ro$, we consider the problem of finding the \emph{global} optimum within a specific \subspace\ of $\Htwo$ defined as follows.%
\begin{definition}%
\label{dfn:Definition_Subspace_of_transfer_functions}%
Let $\TAM\in\Complex^{\ro\times\ro}$ and $\TBM\in\Complex^{\is\times\ro}$ be given such that $\TAM$ is diagonalizable and has all eigenvalues in the open left half-plane. We denote by $\TransferSubspaceXY{\TAM}{\TBM}$ the set of asymptotically stable, strictly proper transfer functions $\TG_{\Msys}\tin\Htwo$ such that%
\begin{equation}%
\TransferSubspaceXY{\TAM}{\TBM}\assign\left\lbrace \TG_{\Msys}(s)\;\left|\;\exists\, \TC_\Msys\in\Complex^{p\times \ro}: \TG_{\Msys}(s)= \TC_\Msys\,(s\,\TI_q-\TAM)^{-1}\,\TBM\right.\right\rbrace\subset\Htwo\;.%
\end{equation}%
\end{definition}%
\begin{remark}\label{remark:subspace diagonal}%
    It holds $\TransferSubspaceXY{\TAM}{\TBM}\teq\TransferSubspaceXY{\widehat{\TA}_\Msys}{\widehat{\TB}_\Msys}$, where $\widehat{\TA}_\Msys \teq \TT^{-1}\TAM\TT \teq \diag(\lambda_{\Msys 1},\,\ppp\,,\,\lambda_{\Msys q})$ and $\widehat{\TB}_\Msys \teq \TT^{-1} \TBM$ for a nonsingular matrix $\TT$. Therefore, all transfer functions in $\TransferSubspaceXY{\TAM}{\TBM}$ share the same poles and input residues.%
\end{remark}%
The reduced-order model within this subset which yields the smallest approximation error will be denoted as the \HtwoText\ pseudo-optimum \cite{Wolf2014}.%
\begin{definition}%
\label{dfn:Definition_Pseudo_Optimal_ROM}%
Consider a DAE system \Cref{eq:LTI_DAE_Definition} with a~strictly proper transfer function $\TGs$. Let  $\TAM\tin\Complex^{\ro\times\ro}$ be diagonalizable and  $\TBM\tin\Complex^{\ro\times\is}$. A reduced-order model $\TGrs$ is called \emph{\HTPO\ with respect to} $\TransferSubspaceXY{\TAM}{\TBM}$ if it satisfies%
\begin{equation}%
    \TGr(s)=\arg\;\min_{\TG_\Msys\in\TransferSubspaceXY{\TAM}{\TBM}}\normHtwo{\TG-\TG_\Msys}\;.%
\end{equation}%
\end{definition}%
Using the Hilbert projection theorem \cite{Rudin1987}, \Cref{thm:GenSylEq_for_X_and_Y} and  \Cref{thm:Description_of_the_Htwo_inner_product_via_moments}, it is possible to derive necessary and sufficient conditions for \HTPOy. This result is an~extension of that in {\cite{Beattie2012}, \cite[Theorem~4.19]{Wolf2014}} to the DAE case.%
\begin{theorem}%
\label{lem:PseudoOptimality}%
Consider a DAE system \Cref{eq:LTI_DAE_Definition} with a~strictly proper transfer function $\TG\tin\Htwo$. Let matrices $\TAM\tin\Complex^{\ro\times\ro}$ and $\TBM\tin\Complex^{\ro\times\is}$ be given such that $\TAM$ is diagonalizable. Then $\TGrs$ is the unique \HTPO\ reduced transfer function with respect to $\TransferSubspaceXY{\TAM}{\TBM}$ if and only if
\begin{equation}
\label{eq:Condition_PseudoOptimality_InnerProduct}
\innerproductHtwo{\TG-\TGr}{\TG_\Msys}=0\quad \mbox{for all} \quad\TG_\Msys\in\TransferSubspaceXY{\TAM}{\TBM}
\end{equation}
or, equivalently,%
\begin{equation}
\label{eq:Condition_PseudoOptimality_Moments}
\left(\TG(-\complexconjugate{\lambda}_{\Msys i})-\TGr(-\complexconjugate{\lambda}_{\Msys i})\right)\Tb_{\Msys i}^\complexconjugatetranspose=\Tzero\;,\qquad i=1,\,\ppp\,,\,q\;,
\end{equation}
holds, %
where $\lambda_{\Msys i}$ and $\,\Tb_{\Msys i}^\complexconjugatetranspose$ are the poles and input residual directions of any ${\TG_\Msys\tin\TransferSubspaceXY{\TAM}{\TBM}}$.
\end{theorem}%
\begin{proof}%
The relationship in \Cref{eq:Condition_PseudoOptimality_InnerProduct} follows directly from the Hilbert projection theorem on optimality within a subspace, see \cite{Rudin1987}. Based on this result and the new formulation of the \HTIP\ presented in \Cref{thm:GenSylEq_for_X_and_Y}, we can derive an equivalent characterization of \Cref{eq:Condition_PseudoOptimality_InnerProduct} in terms of tangential interpolatory conditions \Cref{eq:Condition_PseudoOptimality_Moments}.\par%
Denoting the error system by $\TG_e(s)\teq\TGs-\TGrs$ with a realization $(\TE_e,\TA_e,\TB_e,\TC_e)$, we obtain%
\begin{equation*}%
    \innerproductHtwo{\TG_e}{\TG_\Msys} = \trace\left(\TC_e\,\TX\,\TCM^{\complexconjugatetranspose}\right),%
\end{equation*}%
where $\TX$ solves the projected Sylvester equation%
\begin{equation}\label{eq:GenSylEq_for_X_ErrorMsys}%
    \TA_e\,\TX\,+\, \TE_e\,\TX\,\TAM^{\complexconjugatetranspose}+\Spec^\finite_{e\idxleft}\,\TB_e\,\TBM^{\complexconjugatetranspose} = \Tzero.%
\end{equation}%
Here, $\Spec^\finite_{e\idxleft}$ is the spectral projector onto the left subspace of $\lambda\,\TE_e-\TA_e$ corresponding to the finite eigenvalues. Since $\TAM$ is diagonalizable, there exists a~nonsingular matrix $\TT\tin\Complex^{\ro\times\ro}$ such that $\TT^{-1}\TAM\TT\teq \widehat{\TA}_{\Msys}\teq\diag(\lambda_{\Msys 1},\ppp,\lambda_{\Msys \ro})$. Then equation \Cref{eq:GenSylEq_for_X_ErrorMsys} can be written as%
\begin{equation}%
    \TA_e\,\widehat{\TX}\,+\,\TE_e\,\widehat{\TX}\,\widehat{\TA}_{\Msys}^{\complexconjugatetranspose}+\Spec^\finite_{e\idxleft}\,\TB_e\,\widehat{\TB}_{\Msys}^{\complexconjugatetranspose} = \Tzero,%
\end{equation}%
where $\widehat{\TX}\teq\TX\TT^{-\complexconjugatetranspose}$ and $\widehat{\TB}_{\Msys}\teq\TT^{-1}\TBM\teq\begin{bmatrix} \Tb_{\Msys 1}^{\complexconjugatetranspose}&\ppp& \Tb_{\Msys \ro}^{\complexconjugatetranspose}\end{bmatrix}^{\complexconjugatetranspose}$. Due to the diagonal structure of $\widehat{\TA}_\Msys$, every column $\widehat{\TX}_i$ of $\widehat{\TX}$ is given by%
\begin{equation*}%
    \widehat{\TX}_i = \left(-\complexconjugate{\lambda}_{\Msys i}\,\TE_e-\TA_e \right)^{-1}\Spec^\finite_{e\idxleft}\,\TB_e\,\Tb_{\Msys i}^{\complexconjugatetranspose}, \quad i=1,\ppp,\ro.%
\end{equation*}%
As the relation%
\begin{equation*}%
    0=\innerproductHtwo{\TG_e}{\TG_\Msys} = \trace(\TC_e\,\TX\,\TCM^{\complexconjugatetranspose}) = \trace(\TC_e\,\widehat{\TX}\,\widehat{\TC}_{\Msys}^{\complexconjugatetranspose})%
\end{equation*}
with $\widehat{\TC}_{\Msys}=\TC_{\Msys}\TT$ has to hold true for all $\TGM\tin\TransferSubspaceXY{\TAM}{\TBM}$, i.e., for all $\TCM\tin\Complex^{\os\times\ro}$, we conclude from%
\begin{equation}%
    \begin{aligned}%
        \TC_e\widehat{\TX} &= \TC_e \begin{bmatrix} \left(-\complexconjugate{\lambda}_{\Msys 1}\,\TE_e-\TA_e\right)^{-1}\Spec^\finite_{e\idxleft}        \,\TB_e\,\Tb_{\Msys 1}^{\complexconjugatetranspose} &\ppp & \left(-\complexconjugate{\lambda}_{\Msys \ro}\,\TE_e-\TA_e\right)^{-1}\Spec^\finite_{e\idxleft}\,\TB_e\,\Tb_{\Msys \ro}^{\complexconjugatetranspose}\end{bmatrix}\\%
        & = \begin{bmatrix} \TG_e\left(-\complexconjugate{\lambda}_{\Msys 1}\right)\,\Tb_{\Msys 1}^{\complexconjugatetranspose} &\ppp & \TG_e\left(-\complexconjugate{\lambda}_{\Msys \ro}\right)\,\Tb_{\Msys \ro}^{\complexconjugatetranspose}    \end{bmatrix}%
    \end{aligned}%
    \label{eq:Condition_PseudoOptimality_Proof}%
\end{equation}%
that \Cref{eq:Condition_PseudoOptimality_InnerProduct} is satisfied if and only if%
\begin{equation*}%
    \TG_e\left(-\complexconjugate{\lambda}_{\Msys i}\right)\,\Tb_{\Msys i}^\complexconjugatetranspose=\Tzero\;,\qquad i=1,\,\ppp\,,\,q\;.%
\end{equation*}%
Note that in \Cref{eq:Condition_PseudoOptimality_Proof} we used the strict properness of $\TGs$ and $\TGrs$.%
\end{proof}%
This result is central in that it describes \HTPO\ reduced models as tangential interpolants at the mirror images of the reduced eigenvalues with respect to the reduced input residual directions. In fact, it indicates how to construct \HTPO\ reduced models. Given reduced poles and input residual directions, we seek a reduced model satisfying the right tangential interpolation conditions \Cref{eq:Condition_PseudoOptimality_Moments}. Alternatively, given right tangential interpolation data $\TS$ and $\TR$, we seek a reduced model with respective poles and input residual directions. Therefore, the goal in the remainder of this section will be to develop an algorithm to construct such a reduced-order model.\par%
\begin{remark}\label{remark:HTPO_HTO}%
    Comparing \Cref{lem:PseudoOptimality} to \Cref{thm:H2optimality}, one can immediately see that the necessary and sufficient conditions \Cref{eq:Condition_PseudoOptimality_Moments} for \HTPOy\ correspond to the necessary conditions \Cref{eq:H2OptimalityCondition_input} for \HTOy. Therefore, every locally \HTO\ reduced-order model is also the global $\Htwo$ pseudo-optimum in the corresponding subspace. This relationship will be used in \Cref{sec:curedspark}, where the subspace $\TransferSubspaceXY{\TAM}{\TBM}$ will be optimized in order to obtain an \HTO\ reduced-order model through pseudo-optimal reduction.%
\end{remark}%
At this point, it still needs to be clarified how a model satisfying \Cref{eq:Condition_PseudoOptimality_Moments} can be constructed. It is well known that a transfer function $\TGrs$ allows an infinite amount of realizations. In the following, we will consider three particular ODE realizations of order $\ro$, namely $\SysM\teq\RealizationM$, $\SysF\teq\RealizationF$ and $\Sysr\teq\Realizationr$, each of which will play a specific role in the derivations. The key in the discussion will be to show that these systems are restricted system equivalent \cite{Mayo2007}, i.e., they share the same transfer function%
\begin{equation}%
    \TGrs=\Transferr=\TransferM=\TransferF,%
\end{equation}%
and, hence, the same properties with respect to tangential interpolation and pole-residue representation. In contrast to $\Sysr$, which represents the ROM we are searching for, the systems $\SysM$ and $\SysF$ are of theoretical interest only.\par %
\subsection{The \texorpdfstring{$\SysM$}{M} System}%
First, we introduce a system $\SysM$ of order $\ro$ whose transfer function $\TGMs$ belongs to a subspace defined in terms of the right tangential interpolation data $\TS$ and $\TR$.%
\begin{definition}\label{def:SysM}%
Let the right tangential interpolation data $\TS\tin\Complex^{\ro\times\ro}$ and $\TR\tin\Complex^{\is\times\ro}$ be given according to \Cref{thm:S}. Then $\SysM\teq\RealizationM$ is an ODE realization of order $q$ of a~transfer function $\TGMs$ defined by%
\begin{equation}%
    \label{eq:definition_of_M_system}%
    \TGM \in \TransferSubspaceXY{-\TS^\complexconjugatetranspose}{\TR ^\complexconjugatetranspose}.%
\end{equation}%
\end{definition}%
Recall the relations $\TS^\prim\teq\TT\,\TS\,\TT^{-1}\teq\diag(\sigma_1,\ppp,\sigma_\ro)$ and $\TR^\prim\teq\TR\,\TT^{-1}\teq \left[\Tr_1,\ppp,\Tr_\ro\right]$ from \Cref{thm:S}. Then by \Cref{remark:subspace diagonal} we have $\TransferSubspaceXY{-\TS^\complexconjugatetranspose}{\TR ^\complexconjugatetranspose}\teq\TransferSubspaceXY{-(\TS^\prim)^\complexconjugatetranspose}{(\TR^\prim) ^\complexconjugatetranspose}$, and, hence, by \Cref{lem:PseudoOptimality}, necessary and sufficient conditions for \HTPOy\ with respect to $\TransferSubspaceXY{-\TS^\complexconjugatetranspose}{\TR ^\complexconjugatetranspose}$ can be formulated as%
\begin{equation}%
    \label{eq:Condition_on_Msystem_for_HTPOy}%
    \left(\TG(\sigma_i)-\TGM(\sigma_i)\right)\Tr_i=\Tzero\;,\qquad i=1,\,\ppp\,,\,q\;.%
\end{equation}%
\begin{remark}\label{remark:positive shifts leads to stable ROM}%
Due to \Cref{eq:definition_of_M_system}, the relation $\lambda_i\teq-\complexconjugate{\sigma}_{\Msys i}$ holds for the eigenvalues of $\lambda\, \TEM-\TAM$. As a~consequence, the choice of interpolation points in the open right half-plane automatically enforces asymptotic stability of $\SysM$ for any choice of the free parameter matrix $\TCM$.%
\end{remark}%
\subsection{The \texorpdfstring{$\SysF$}{F} System}%
For given right tangential interpolation data $\TS\tin\Complex^{\ro\times\ro}$ and $\TR\tin\Complex^{\is\times\ro}$ as in \Cref{thm:S} and a~matrix $\TF\in\Complex^{\ro\times\is}$ satisfying \Cref{eq:condition_for_F}, we consider a~system $\SysF\teq\RealizationF$ with the transfer function $\TGFs\teq\TransferF$ as in \Cref{dfn:Familiy_of_reduced_transfer_functions} which is parametrized by $\TF$. Then by \Cref{thm:parametrizedFamily_tangential_interpolation}, the interpolation conditions \Cref{eq:Tangential interpolation of SysF} hold. Furthermore, it follows from \Cref{lem:PseudoOptimality} that the \HTPO\ reduced model with respect to $\TransferSubspaceXY{-\TS^\complexconjugatetranspose}{\TR ^\complexconjugatetranspose}$ must combine the properties of both realizations $\SysM$ and $\SysF$. Therefore, in the following we will link $\SysM$ and $\SysF$ together through an appropriate choice of remaining degrees of freedom such that $\TGFs\teq\TGMs$ holds. For this purpose, a~third auxiliary realization $\Sysr$ is required which will provide the \HTPO\ reduced model.%
\subsection{The \texorpdfstring{$\Sysr$}{r} System}%
Assume that $\SysM$ is asymptotically stable, $\TEM$ is nonsigular and the triple $(\TEM,\TAM,\TBM)$ is controllable, i.\,e., $\mbox{rank}\begin{bmatrix}\lambda\,\TEM-\TAM & \TBM\end{bmatrix}=q$ for all $\lambda\in\mathbb{C}$, then the generalized Lyapunov equation%
\begin{equation}%
    \label{eq:controllability_Gramian_ROM_Definition}%
    \TAM\,\GramcM\,\TEM^{\complexconjugatetranspose}+\TEM\,\GramcM\,\TAM^{\complexconjugatetranspose}+\TBM\,\TBM^{\complexconjugatetranspose}=\Tzero\;%
\end{equation}%
has a unique Hermitian positive definite solution $\GramcM$ which is the controllability Gramian of $\SysM$. Using this Gramian, we define a system $\Sysr\teq\Realizationr$ as follows.%
\begin{definition}\label{def:SysR}%
Let an~asymptotically stable system $\SysM\teq\RealizationM$ be given such that $\TEM$ is nonsingular and the triple $(\TEM,\TAM,\TBM)$ is controllable. Then $\Sysr\teq\Realizationr$ is defined as%
\begin{equation}%
    \label{eq:realization_ROM_depending_on_Msys}%
    \TAr=\TEr\GramcMinv\TEM^{-1}\TAM\,\GramcM\;,\quad \TBr=-\TEr\GramcMinv\TEM^{-1}\TBM\;,\quad\TCr=-\TCM\,\GramcM\;,%
\end{equation}%
where $\TEr$ is an~arbitrary nonsingular matrix and $\GramcM$ is the controllability Gramian of $\SysM$.%
\end{definition}%
Note that $\SysM$ and $\Sysr$ are restricted system equivalent and, as a~consequence, have the same transfer function. Indeed, we have%
\begin{equation}\label{eq:Gr_GM}%
    \begin{aligned}%
        \TGrs & = \TCr\left(s\TEr-\TAr\right)^{-1}\TBr \\%
              & = \TCM\,\GramcM\left(s\TEr-\TEr\GramcMinv\TEM^{-1}\TAM\,\GramcM\right)^{-1}\TEr\GramcMinv\TEM^{-1}\TBM \\%
              & = \TCM\,\GramcM\left(\TEr\GramcMinv\TEM^{-1} (s\TEM-\TAM)\GramcM\right)^{-1}\TEr\GramcMinv\TEM^{-1}\TBM \\%
              & = \TCM\left(s\TEM-\TAM\right)^{-1}\TBM = \TGMs.%
    \end{aligned}%
\end{equation}%
Moreover, the controllability of $(\TEM,\TAM,\TBM)$ implies the controllability of $(\TEr,\TAr,\TBr)$. In addition, it is possible to relate the matrix $\TAr$ directly to the interpolation data $\TS$ and $\TR$.%
\begin{lemma}\label{lem:Ar in terms of S_R}%
Let the matrices $\TS\tin\Complex^{\ro\times\ro}$ and $\TR\tin\Complex^{\is\times\ro}$ be given such that the pair $(-\TS^{\complexconjugatetranspose},\TR^{\complexconjugatetranspose})$ is controllable. Consider the system $\SysM=(\TEM,\TAM,\TBM,\TCM)$ with %
\begin{equation}%
    \TEM\teq\TI,\quad \TAM\teq-\TS^\complexconjugatetranspose, \quad \TBM\teq\TR^\complexconjugatetranspose%
    \label{eq:EM_AM_BM}%
\end{equation}%
and the system $\Sysr$ constructed according to \Cref{def:SysR}. Then it holds%
\begin{equation}%
    \TAr = \TEr\,\TS + \TBr\,\TR.%
    \label{eq:Ar in terms of SR}%
\end{equation}%
\end{lemma}%
\begin{proof}%
Multiplying the Lyapunov equation \Cref{eq:controllability_Gramian_ROM_Definition} from the left by $\TEr\,\GramcMinv\,\TEM^{-1}$ and taking into account \Cref{eq:realization_ROM_depending_on_Msys} and \Cref{eq:EM_AM_BM}, we obtain%
\begin{equation*}%
    0=(\TEr\,\GramcMinv\,\TEM^{-1}\,\TAM\,\GramcM)\,\TEM+\TEr\,\TAM^\complexconjugatetranspose+(\TEr\,\GramcMinv\,\TEM^{-1}\,\TBM)\, \TBM^\complexconjugatetranspose=\TAr-\TEr\,\TS - \TBr\,\TR.%
\end{equation*}%
Thus, \Cref{eq:Ar in terms of SR} is satisfied.%
\end{proof}%
Finally, we select the remaining degrees of freedom such that $\TGFs = \TGrs$, and, hence, by \Cref{eq:Gr_GM}, we get $\TGFs\teq\TGMs$. The quantities that still need to be specified are%
\vspace*{-0.5em}%
\begin{enumerate}\itemsep0pt%
    \item $\TF$, the parameter matrix in $\SysF$ satisfying \Cref{eq:condition_for_F},%
    \item $\TCM$, the output matrix in $\SysM$, or alternatively, due to \Cref{eq:realization_ROM_depending_on_Msys}, $\TCr\teq-\TCM \GramcM$, and%
    \item $\TEr$, a nonsingular matrix in $\Sysr$.%
\end{enumerate}%
\vspace*{-0.5em}%
Following the discussion so far, it is straightforward to see that the choice $\TF\assign\TEr^{-1}\TBr$ and $\TCr\assign\TC\,\TV$ with $\TV$ solving the Sylvester equation \Cref{eq:Sylvester} yields%
\begin{equation*}%
    \label{eq:connection_Fsys_rsys}%
    \begin{aligned}%
        \TGFs&=\TransferF=\TC\,\TV\left(s\,\TI-\TS-\TEr^{-1}\,\TBr\,\TR\right)^{-1}\TEr^{-1}\,\TBr\\%
        &=\TCr(s\,\TEr-\TEr\,\TS-\TBr\,\TR)^{-1}\TBr=\Transferr=\TGrs,%
    \end{aligned}%
\end{equation*}%
where we made use of \Cref{lem:Ar in terms of S_R}. Furthermore, the following lemma shows that the chosen matrix $\TF$ satisfies the spectral condition \Cref{eq:condition_for_F} once the interpolation points belong to the open right half-plane.%
\begin{lemma}\label{lem:spectral_condition_for_F}%
Let the matrices $\TS\tin\Complex^{\ro\times\ro}$ and $\TR\tin\Complex^{\is\times\ro}$ be given such that all eigenvalues of $\TS$ have positive real part and the pair $(-\TS^{\complexconjugatetranspose},\TR^{\complexconjugatetranspose})$ is controllable. Consider the system $\SysM$ as in \Cref{eq:EM_AM_BM} and the system $\Sysr$ constructed according to \Cref{def:SysR}. Then the matrix $\TF\teq\TEr^{-1}\,\TBr$ satisfies \Cref{eq:condition_for_F}.%
\end{lemma}%
\begin{proof}%
It follows from \Cref{lem:Ar in terms of S_R} and the definition of the systems $\SysM$ and $\Sysr$ that%
\begin{equation*}%
    \TS+\TF\,\TR=\TS+\TEr^{-1}\,\TBr\,\TR=\TEr^{-1}\,\TAr=\GramcMinv\TAM\,\GramcM=-\GramcMinv\,\TS^\complexconjugatetranspose\,\GramcM,%
\end{equation*}%
and, hence, the condition \Cref{eq:condition_for_F} is satisfied.%
\end{proof}%
According to \Cref{eq:Tangential interpolation of SysF}, $\TGFs$ tangentially interpolates $\TGs$ as encoded in $\TS$ and $\TR$. Moreover, we have $\TGM\tin\TransferSubspaceXY{-\TS^{\complexconjugatetranspose}}{\TR^\complexconjugatetranspose}$. Then we can conclude from \Cref{lem:PseudoOptimality} that $\TGrs$ is indeed the unique \HTPO\ reduced transfer function with respect to $\TransferSubspaceXY{-\TS^{\complexconjugatetranspose}}{\TR^\complexconjugatetranspose}$.\par%
In order to determine the reduced-order system $\Sysr$ given in \Cref{eq:realization_ROM_depending_on_Msys}, we need to solve the Lyapunov equation \Cref{eq:controllability_Gramian_ROM_Definition} which for the system $\SysM$ with the system matrices as in \Cref{eq:EM_AM_BM} takes the form%
\begin{equation*}%
    \TS^{\complexconjugatetranspose}\,\GramcM+\GramcM\,\TS-\TR^{\complexconjugatetranspose}\,\TR=\Tzero.%
\end{equation*}%
To avoid the inversion of $\GramcM$ in \Cref{eq:realization_ROM_depending_on_Msys}, we choose $\TEr=\GramcM$ and get the reduced-order model%
\begin{equation}%
    \TEr=\GramcM, \quad \TAr=\GramcM\TS-\TR^{\complexconjugatetranspose}\TR=-\TS^{\complexconjugatetranspose}\GramcM, \quad     \TBr=-\TR^{\complexconjugatetranspose}, \quad \TCr=\TC\TV.%
    \label{eq:reduced_model_Sigmar}%
\end{equation}%
This finally leads to the \emph{\HTPO\ rational Krylov} (PORK) algorithm for DAEs with strictly proper transfer function, see \Cref{alg:input_PORK}. It is worth noting that the PORK algorithm for general DAEs should be applied to the strictly proper part $(\TE,\TA, \Specleftfin\,\TB, \TC\,\Specrightfin)$. Due to the duality in linear control systems, there exists also an equivalent \emph{output} PORK algorithm which can be derived in a~similar way.\par %
\begin{algorithm}[ht]%
    \KwIn{$\Sys\teq\Realization$, $\TS$, $\TR$, $\TV$}%
    \KwOut{\HTPO\ reduced-order model $\Sysr\teq\Realizationr$}%
    solve the Lyapunov equation $\TS^{\complexconjugatetranspose}\,\GramcM+\GramcM\,\TS-\TR^{\complexconjugatetranspose}\,\TR=\Tzero$ for $\GramcM$\;%
    compute the reduced-order model $\TEr=\GramcM$, $\TAr = -\TS^{\complexconjugatetranspose}\,\GramcM$, $\TBr = -\TR^{\complexconjugatetranspose}$ and $\TCr = \TC\,\TV$\;%
    \caption{(Input) \PORKA\ for strictly proper \DAE s}%
    \label{alg:input_PORK}%
\end{algorithm}%
As a summary of the results so far, \cref{fig:relations_sysM_sysF_sysr} gives an overview of the relations between the three different realizations $\SysM$, $\SysF$ and $\Sysr$ used to obtain the \HTPO\ reduced order model.\par%
\begin{figure}[!htb]%
    \centering%
    \resizebox{0.9\textwidth}{!}{\input{./relations_sysM_sysF_sysr.tex}}%
    \caption{Illustration of the dependencies between $\SysM$, $\SysF$ and $\Sysr$ during the proof of \HTPOy. To be read from top to bottom. Adapted from \cite[Figure 4.3]{Seiwald2016}.\ToEditor{layout=spanning one-column preferred if readable; online-version=colored; print-version=grayscale}}%
    \label{fig:relations_sysM_sysF_sysr}%
\end{figure}%
\subsection{Properties and Advantages of \texorpdfstring{$\Htwo$}{H2} Pseudo-Optimality}\label{sec:PORK_properties}%
The PORK algorithm provides a~reduced-order model $\TGr(s)$ which is global \HtwoText \space optimum in the subspace $\TransferSubspaceXY{-\TS^\complexconjugatetranspose}{\TR^\complexconjugatetranspose}$. Naturally, the quality of this reduced model highly depends on the selection of $\TransferSubspaceXY{-\TS^\complexconjugatetranspose}{\TR^\complexconjugatetranspose}$, as the global optimum in a badly chosen subspace can be an arbitrarily bad approximation. For this reason, \HTPO\space reduction is most valuable in combination to a subspace optimization algorithm. In \Cref{sec:spark}, we will embed \Cref{alg:input_PORK} into a globally convergent trust region optimization to optimize the subspace and find a locally \HTO\space reduced-order model.\par%
Nonetheless, \HTPO \space reduction of DAEs by itself already bears a few advantageous properties we shall briefly revise, 
see \cite{Wolf2014} for ODE case.\par%
\newcounter{property}%
\stepcounter{property}%
\textbf{\Roman{property}. Error norm decomposition:} As already observed by Meier and Luenberger in \cite{Meier1967}, by the Hilbert projection theorem, the \HTPO\space reduced model satisfies%
\begin{equation}%
    \normHtwo{\TG-\TGr}^2 = \normHtwo{\TG}^2-\normHtwo{\TGr}^2,%
    \label{eq:h2 norm decomposition}%
\end{equation}%
since $\innerproductHtwo{\TG-\TGr}{\TGr} = 0 $ holds.%
This implies that minimization of the reduction error can be achieved by maximizing the norm of the reduced model, which is far easier to compute, as it will be shown in \Cref{sec:spark}.\par%
\stepcounter{property}%
\textbf{\Roman{property}. Reduction of the search space for \HTO \space reduction:} As all \HTO \space reduced models are also \HTPO (see \Cref{remark:HTPO_HTO}), an optimization within the set of \HTPO\space models is non-restrictive.\par%
\stepcounter{property}%
\textbf{\Roman{property}. Explicit dependence of $\TV$:} The PORK algorithm constructs \HTPO \space reduced models depending explicitly on a projection matrix $\TV$ that satisfies the Sylvester equation \Cref{eq:Sylvester}. Even though the advantage of this may not be obvious at first, the dependence on $\TV$ will be crucial to adaptively select the reduced order, as it will be presented in \Cref{sec:cure}. In addition, \HTPO\space reduction will ensure monotonic decrease of the error as we increase the reduced order.\par%
\stepcounter{property}%
\textbf{\Roman{property}. Relation to the ADI iteration for Lyapunov equations:} It has been shown in \cite{Wolf2013_at,Wolf2013_ADI} that the Alternative Direction Implicit (ADI)  method \cite{Penzl2000,Li2002,Kuerschner2016} for finding low-rank solutions of large-scale sparse Lyapunov equations can be interpreted as a two-step approach using \HTPO\space reduction. As a detailed discussion of this goes beyond the scope of this contribution, we limit ourselves to indicating this connection, which may be useful in solving projected Lyapunov equations for DAEs.\par%
\stepcounter{property}%
\textbf{\Roman{property}. Stability preservation:} A selection of $\TS$ with eigenvalues in the open right half-plane yields an~\emph{asymptotically stable} \HTPO\ reduced-order model by construction, as its eigenvalues will be the mirror images of the eigenvalues of $\TS$, see \Cref{remark:positive shifts leads to stable ROM}. This trivially provides a~stability preserving model reduction. Moreover, as we are not only fixing interpolation frequencies but also reduced eigenvalues, an appropriate choice of $\TS$ becomes twice as important.\par%

%% file: relations_sysM_sysF_sysr.tex
\begingroup%
  \makeatletter%
  \providecommand\color[2][]{%
    \errmessage{(Inkscape) Color is used for the text in Inkscape, but the package 'color.sty' is not loaded}%
    \renewcommand\color[2][]{}%
  }%
  \providecommand\transparent[1]{%
    \errmessage{(Inkscape) Transparency is used (non-zero) for the text in Inkscape, but the package 'transparent.sty' is not loaded}%
    \renewcommand\transparent[1]{}%
  }%
  \providecommand\rotatebox[2]{#2}%
  \ifx\svgwidth\undefined%
    \setlength{\unitlength}{436.53543307bp}%
    \ifx\svgscale\undefined%
      \relax%
    \else%
      \setlength{\unitlength}{\unitlength * \real{\svgscale}}%
    \fi%
  \else%
    \setlength{\unitlength}{\svgwidth}%
  \fi%
  \global\let\svgwidth\undefined%
  \global\let\svgscale\undefined%
  \makeatother%
  \begin{picture}(1,0.75324675)%
    \put(0,0){\includegraphics[width=\unitlength,page=1]{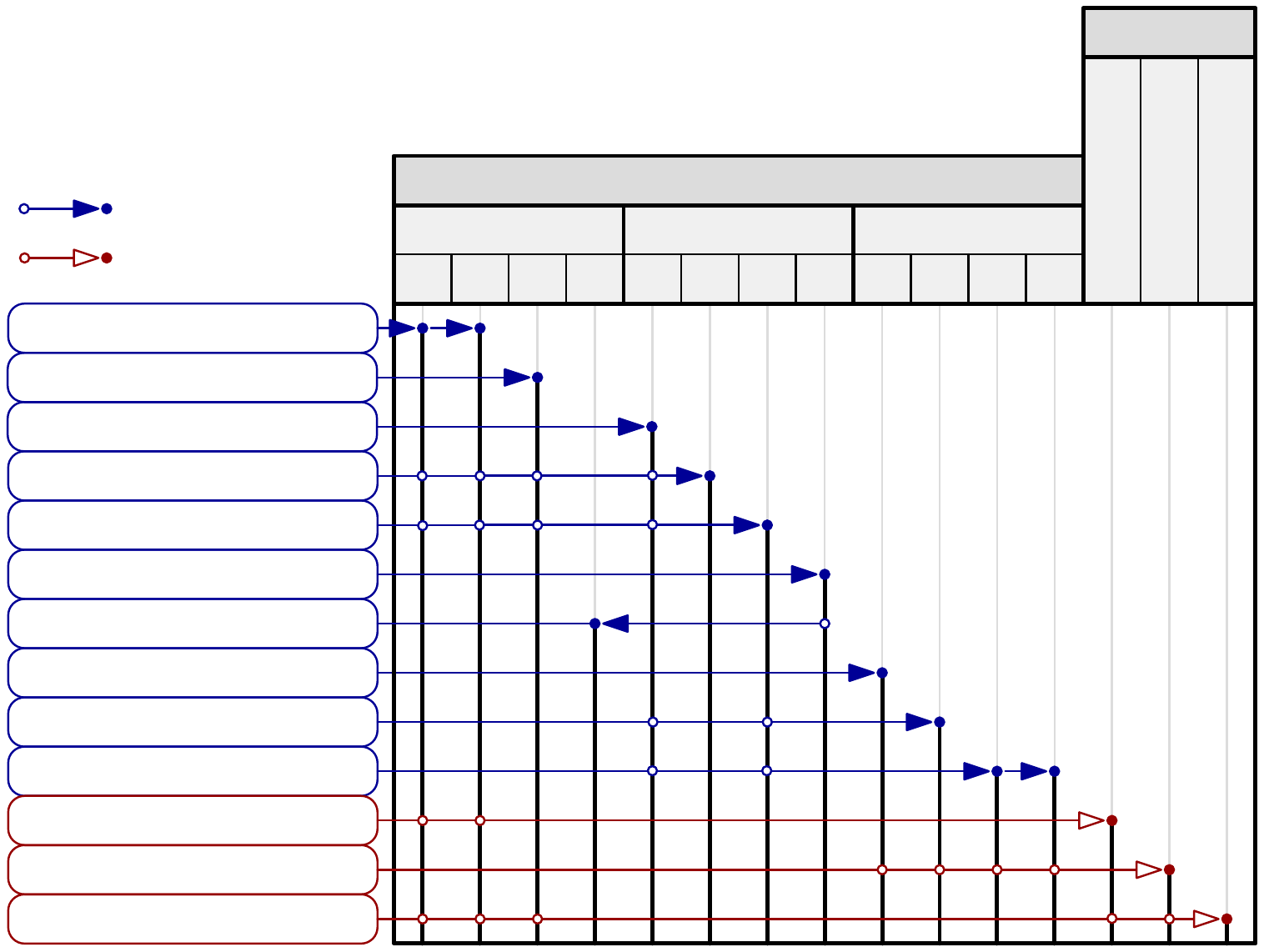}}%
    \put(0.38961038,0.56493501){\color[rgb]{0,0,0}\makebox(0,0)[lb]{\smash{$\SysM$}}}%
    \put(0.31493505,0.52597394){\color[rgb]{0,0,0}\makebox(0,0)[lb]{\smash{$\TEM$}}}%
    \put(0.57142859,0.56493501){\color[rgb]{0,0,0}\makebox(0,0)[lb]{\smash{$\Sysr$}}}%
    \put(0.5032468,0.52597394){\color[rgb]{0,0,0}\makebox(0,0)[lb]{\smash{$\TEr$}}}%
    \put(0.63961039,0.52597394){\color[rgb]{0,0,0}\makebox(0,0)[lb]{\smash{$\TCr$}}}%
    \put(0.75324675,0.56493501){\color[rgb]{0,0,0}\makebox(0,0)[lb]{\smash{$\SysF$}}}%
    \put(0.88961034,0.52597397){\color[rgb]{0,0,0}\rotatebox{90}{\makebox(0,0)[lb]{\smash{asymp. stab.}}}}%
    \put(0.93506485,0.52597397){\color[rgb]{0,0,0}\rotatebox{90}{\makebox(0,0)[lb]{\smash{tang. interp.}}}}%
    \put(0.98051953,0.52597397){\color[rgb]{0,0,0}\rotatebox{90}{\makebox(0,0)[lb]{\smash{$\Htwo$ pseudo-opt.}}}}%
    \put(0.012987,0.37013021){\color[rgb]{0,0,0}\makebox(0,0)[lb]{\smash{$\TAr= \TEr\,\GramcMinv\,\TAM\,\GramcM$}}}%
    \put(0.01298736,0.25324724){\color[rgb]{0,0,0}\makebox(0,0)[lb]{\smash{$\TCM= -\TCr\,\GramcMinv $}}}%
    \put(0.01256818,0.40909105){\color[rgb]{0,0,0}\makebox(0,0)[lb]{\smash{choose $\TEr$: $\det(\TEr)\neq0$}}}%
    \put(0.012987,0.3311692){\color[rgb]{0,0,0}\makebox(0,0)[lb]{\smash{$\TBr= -\TEr\,\GramcMinv\,\TBM$}}}%
    \put(0.01298702,0.17532472){\color[rgb]{0,0,0}\makebox(0,0)[lb]{\smash{$\TAF= \TS+\TEr^{-1}\,\TBr\,\TR$}}}%
    \put(0.01298701,0.13636387){\color[rgb]{0,0,0}\makebox(0,0)[lb]{\smash{$\TBF=\TEr^{-1}\,\TBr$, $\TCF= \TC\,\TV$}}}%
    \put(0.4512987,0.52597394){\color[rgb]{0,0,0}\makebox(0,0)[lb]{\smash{$\TCM$}}}%
    \put(0.54870131,0.52597394){\color[rgb]{0,0,0}\makebox(0,0)[lb]{\smash{$\TAr$}}}%
    \put(0.59415588,0.52597394){\color[rgb]{0,0,0}\makebox(0,0)[lb]{\smash{$\TBr$}}}%
    \put(0.68181818,0.52597394){\color[rgb]{0,0,0}\makebox(0,0)[lb]{\smash{$\TEF$}}}%
    \put(0.77272726,0.52597394){\color[rgb]{0,0,0}\makebox(0,0)[lb]{\smash{$\TBF$}}}%
    \put(0.81818182,0.52597394){\color[rgb]{0,0,0}\makebox(0,0)[lb]{\smash{$\TCF$}}}%
    \put(0.36038962,0.52597394){\color[rgb]{0,0,0}\makebox(0,0)[lb]{\smash{$\TAM$}}}%
    \put(0.40584416,0.52597394){\color[rgb]{0,0,0}\makebox(0,0)[lb]{\smash{$\TBM$}}}%
    \put(0.72727274,0.52597394){\color[rgb]{0,0,0}\makebox(0,0)[lb]{\smash{$\TAF$}}}%
    \put(0.11038961,0.5815756){\color[rgb]{0,0,0.58823529}\makebox(0,0)[lb]{\smash{assignment}}}%
    \put(0.11038961,0.54261457){\color[rgb]{0.58823529,0,0}\makebox(0,0)[lb]{\smash{consequence}}}%
    \put(0.52597403,0.60389605){\color[rgb]{0,0,0}\makebox(0,0)[lb]{\smash{Realization}}}%
    \put(0.8961039,0.72077919){\color[rgb]{0,0,0}\makebox(0,0)[lb]{\smash{$\TGrs$}}}%
    \put(0.01298701,0.01623406){\color[rgb]{0,0,0}\makebox(0,0)[lb]{\smash{\Cref{lem:PseudoOptimality}}}}%
    \put(0.01298701,0.48701279){\color[rgb]{0,0,0}\makebox(0,0)[lb]{\smash{$\TEM=\TI_q$, $\TAM=-\TS^\complexconjugatetranspose$}}}%
    \put(0.01298701,0.44805181){\color[rgb]{0,0,0}\makebox(0,0)[lb]{\smash{$\TBM=\TR^\complexconjugatetranspose$}}}%
    \put(0.01298698,0.29220814){\color[rgb]{0,0,0}\makebox(0,0)[lb]{\smash{$\TCr= \TC\,\TV$ }}}%
    \put(0.01298701,0.09740264){\color[rgb]{0,0,0}\makebox(0,0)[lb]{\smash{$\Re\lbrace\sigma_i\rbrace>0$, $\lambda_{\Msys i}=-\complexconjugate{\sigma}_i$}}}%
    \put(0.01298684,0.21428589){\color[rgb]{0,0,0}\makebox(0,0)[lb]{\smash{$\TEF= \TI_q$ }}}%
    \put(0.01298701,0.05519502){\color[rgb]{0,0,0}\makebox(0,0)[lb]{\smash{\Cref{thm:parametrizedFamily_tangential_interpolation}}}}%
  \end{picture}%
\endgroup%

%% file: curedspark.tex
\section{Adaptive, Stability Preserving Reduction of the Strictly Proper Subsystem}%
\label{sec:curedspark}%
One of the main challenges in model reduction by rational interpolation is an~appropriate selection of reduction parameters to obtain a high-fidelity approximation while preserving fundamental properties such as stability. The PORK algorithm represents merely a~reduction tool that needs to be fed with appropriate tangential interpolation data $\TS$ and $\TR$.\par%
To this end, we will integrate PORK within existing methods from \cite{Panzer2013,Panzer2014} that require no additional modification in the DAE setting once the strictly proper part of the transfer function is available. Their description in the following is hence limited to a brief introduction, while details can be found in the original publications. The optimization of $\TS$ and $\TR$ to obtain an \HTO \space reduced model of prescribed order $\ro$ will be discussed in \Cref{sec:spark}. In addition, we will adaptively choose the reduced order by means of \emph{cumulative reduction} (CURE), which will be discussed in \Cref{sec:cure}. The overall reduction method resulting from \Cref{sec:spark} and \Cref{sec:cure} will then be used in \Cref{sec:numericalresults} for numerical experiments. As in the previous section, we assume that the DAE system \Cref{eq:LTI_DAE_Definition} has a~strictly proper transfer function $\TG(s)$.\par%
\subsection{Optimal Choice of \texorpdfstring{$\TS$}{S} and \texorpdfstring{$\TR$}{R}}%
\label{sec:spark}%
The goal is now to find optimal interpolation data $\TS$ and $\TR$ that minimize the $\Htwo$ approximation error $\normHtwo{\TG-\TGr}$. Exploiting \HTPOy \space and the relation \Cref{eq:h2 norm decomposition}, this problem can be equivalently written as%
\begin{equation}%
    \TGr = \arg\min_{\TS,\TR}\left(-\normHtwo{\widehat{\TG}_{\rm r}}^2 \right)\quad \text{subject to} \quad \widehat{\TG}_{\rm r}\in\TransferSubspaceXY{-\TS^\complexconjugatetranspose}{\TR^\complexconjugatetranspose}.%
    \label{eq:optimization}%
\end{equation}%
A similar problem \cite{Beattie2007,Gugercin2008} has been addressed in the context of \HTO \space reduction for SISO  and MIMO models  and arbitrary reduced orders $\ro$, see \cite{Beattie2009,Beattie2012}. In the context of \HTPO \space reduction, this problem has been investigated in \cite{Panzer2013,Panzer2014} for SISO models and $\ro\teq2$. Our discussion will go along the lines of the latter, embedding \Cref{alg:input_PORK} into the \emph{stability preserving adaptive rational Krylov} (SPARK) algorithm of \cite{Panzer2014}. Its advantages lie in the efficient implementation with real arithmetics of a globally convergent trust-region optimization for the optimal choice of $\TS\tin\Complex^{2\times2}$. As this algorithm works for SISO models only, we limit our discussion to this case and use $\Tb,\Tc^{\transpose}\tin\Real^\fo$ as input and output vectors to make this restriction evident. The optimal choice of $\TR$ in the context of \HTPO\ reduction has been addressed in \cite{Beattie2012} for the ODE case using a~pole-residue formulation but has not been extended to work with PORK, yet.\par%
Due to \HTPOy\ and the relation \cref{eq:h2 norm decomposition}, the cost functional in \Cref{eq:optimization} can be equivalently formulated  as%
\begin{equation}%
    \cost(\TS) = - \normHtwo{\widehat{\TG}_{\rm r}}^2 = - \Tc_{\rm r}\Gramcr\Tc_{\rm r}^*,%
    \label{eq:spark-cost}%
\end{equation}%
where $\Tc_{\rm r} = \Tc\TV$ and $\Gramcr$ is the controllability Gramian of the reduced model. To efficiently implement an optimization strategy, explicit knowledge of gradient $\grad$ and Hessian $\Hess$ with respect to the optimization variables is required. In the SISO case, the \subspace\ $\TransferSubspaceXY{-\TS^\complexconjugatetranspose}{\TR^\complexconjugatetranspose}$ to be optimized is completely defined by the eigenvalues of $\TS$.
Thus, we define the set of optimization variables as being $\sigma_i\teq\lambda_i\left(\TS\right)$, $i \teq 1,\dots,\ro$, in accordance with \cref{thm:S}.\par%
To simplify the problem, we first restrict our consideration to reduced models of order $q\teq2$. The limitation will be lifted in \Cref{sec:cure}. Therefore, the number of optimization variables reduces to two complex scalars $\sigma_1,\sigma_2\tin\Complex$. Furthermore, we are interested exclusively in shift parameters in the open right half-plane that are closed under conjugation. For this choice of shifts, the \HTPO\space reduced-order model is guaranteed to be real and asymptotically stable. Hence, we can parametrize the problem by two real scalar parameters $a,b\tin\Real$ satisfying $\sigma_{1,2} \teq a \pm \sqrt{a^2-b}$. The derivation of expressions for $\grad$ and $\Hess$ is based on the following results.%
\begin{lemma}[\cite{Panzer2014}]\label{lm:spark-V}%
    Let $a,b\tin\Real$ be positive numbers and $\sigma_{1,2} \teq a \pm \sqrt{a^2-b}\notin\Lambda(\TE,\TA)$ with $\sigma_1\tneq\sigma_2$. Let $\TA_{\sigma_i} := \left(\TA-\sigma_i\TE\right)$. Then the matrix%
    \begin{equation}%
        \TV\teq\begin{bmatrix}\frac{1}{2}\TA_{\sigma_1}^{-1}\,\Tb + \frac{1}{2}\TA_{\sigma_2}^{-1}\,\Tb & \TA_{\sigma_2}^{-1}\,\TE\,\TA_{\sigma_1}^{-1}\,\Tb\end{bmatrix}\tin\Real^{\fo\times2}%
        \label{eq:V-spark}%
    \end{equation}%
    with $\TA_{\sigma_i} = \left(\TA-\sigma_i\,\TE\right)$ is a real basis of the rational Krylov subspace $\Image\left(\begin{bmatrix}\TA_{\sigma_1}^{-1}\,\Tb & \TA_{\sigma_2}^{-1}\,\Tb \end{bmatrix}\right)$ and satisfies the Sylvester equation \Cref{eq:Sylvester} with%
    \begin{equation}%
        \TS = \begin{bmatrix} a & 1 \\ a^2 - b & a \end{bmatrix}, \quad \TR = \begin{bmatrix} 1 & 0 \end{bmatrix}.%
        \label{eq:Sylvester-spark}%
    \end{equation}%
\end{lemma}%
\begin{lemma} \label{lm:spark-ROM}%
    The \HTPO\space reduced-order model \Cref{eq:reduced_model_Sigmar} resulting from $\TV$ and $\TS$, $\TR$ as in \Cref{eq:V-spark} and \Cref{eq:Sylvester-spark}, respectively, is given by%
    \begin{equation}%
        \TEr = \frac{1}{4ab}\begin{bmatrix}(a^2+b) & -a \\ -a & \; 1 \end{bmatrix},\quad%
        \TAr = \frac{1}{4a}\begin{bmatrix}-2a & 1 \\ -1 & 0 \end{bmatrix},\quad%
        \Tb_{\rm r} = \begin{bmatrix}-1 \\ 0 \end{bmatrix}, \quad%
        \Tc_{\rm r} = \Tc\TV.%
        \label{eq:ROM HTPO ab}%
    \end{equation}%
    In addition, the reduced controllability Gramian satisfies%
    \begin{equation}%
        \Gramcr = \begin{bmatrix} 4a & 4a^2 \\ 4a^2 & 4a(a^2+b) \end{bmatrix}.%
        \label{eq:Gamma_r}%
    \end{equation}%
\end{lemma}%
\begin{proof}%
The matrices \Cref{eq:ROM HTPO ab} and \Cref{eq:Gamma_r} can be obtained by straightforward computations.%
\end{proof}%
Note that the restriction $\sigma_1\tneq\sigma_2$ in \Cref{lm:spark-V} can be lifted easily. It was introduced here to simplify the notation.\par%
Making use of \Cref{lm:spark-V}, \Cref{lm:spark-ROM} and \Cref{eq:spark-cost}, the derivation of%
\begin{equation*}%
    \grad = \begin{bmatrix}\frac{\partial \cost}{\partial a} &\frac{\partial \cost}{\partial b}\end{bmatrix} \quad \text{and} \quad%
    \Hess = \begin{bmatrix}%
                \frac{\partial^2 \cost}{\partial a^2}&\frac{\partial^2 \cost}{\partial a\partial b}\\[2mm]%
                \frac{\partial^2 \cost}{\partial a\partial b}&\frac{\partial^2 \cost}{\partial b^2}%
            \end{bmatrix}%
\end{equation*}%
results from straightforward though lengthier computations. As the expressions match the ODE case, we refer to the original publication \cite[Section~4.4.3]{Panzer2014}. Using this relations, the optimization problem can be solved efficiently through a structured trust-region solver, as summarized in \cref{alg:SPARK}.%
\begin{algorithm}%
    \KwIn{$\Sys\teq\Realization$; initial shifts $\sigma_1,\sigma_2\tin\Complex^+$, closed under conjugation}%
    \KwOut{\HTO\ $\Sysr\teq (\TEr, \TAr, \Tb_{\rm r}, \Tc_{\rm r})$ of order $\ro\teq2$}%
    $(a,b) \gets \left( \frac{\sigma_1+\sigma_2}{2}, \sigma_1\,\sigma_2 \right)$ \;%
    \While{not converged}{%
        compute $\cost$, $\grad$, and $\Hess$ at $(a,b)$ \;%
        compute the trust-region step \;%
        update $(a,b)$\;%
    }%
    compute $\TV$ according to \cref{eq:V-spark} with optimal $\sigma_1^{opt}\teq a+\sqrt{a^2-b}$, $\sigma_2^{opt}\teq a-\sqrt{a^2-b}$ \;%
    return reduced order model $\Sysr$ as in \cref{eq:ROM HTPO ab} with optimal $(a,b)$ \;%
    \caption{SPARK for strictly proper \DAE s (adapted from \cite{Panzer2014})}%
    \label{alg:SPARK}%
\end{algorithm}%
\subsection{Adaptive Selection of Reduced Order \texorpdfstring{$\ro$}{q}}%
\label{sec:cure}%
The previous subsection dealt with the problem of finding an \HTO\space reduced SISO model of order $\ro\teq2$ by means of \HTPO\space reduction. Clearly, a reduced order of two might be restrictive for most problems. Nonetheless, it bears the advantage of reducing the search space for optimization to two real scalars $a,b$. So what if $\ro\teq2$ is not enough?\par%
The answer to this question was given in \cite{Panzer2013} by using a decomposition of the error system introduced in \cite{Wolf2011,Wolf2012} and proposing a cumulative reduction scheme. As this method is applicable to MIMO models as well, we will switch back to the matrices $\TB\tin\Real^{\fo\times m}$ and $\TC\tin\Real^{p\times \fo}$. The following result can be proved analogously to the ODE case \cite{Wolf2012}.%
\begin{theorem} \label{thm:Error Factorization CURE}%
Consider a DAE system \cref{eq:LTI_DAE_Definition} with a strictly proper transfer function $\TGs$. Let interpolation data $\TS\tin\Complex^{\ro\times\ro}$ and $\TR\tin\Complex^{\is\times\ro}$ be given according to \cref{thm:S}. Consider the reduced-order model \cref{eq:ROM_by_projection} with the transfer function $\TGrs$ obtained through projection with matrices $\TW$ and $\TV$ such that $\TV$ solves the Sylvester equation \Cref{eq:Sylvester} and  $\TW$ satisfies $\det\left(\TW^{\transpose}\TE\TV\right)\tneq0$. Define $\TB_{\bot}\teq\bigl(\TI-\TE\TV(\TW^{\transpose}\TE\TV)^{-1}\TW^{\transpose}\bigr)\TB$. Then the error system $\TGes \teq \TGs-\TGrs$ can be factorized as%
\begin{equation}%
    \TGes = \TGbots\,\widetilde{\TG}_{\rm r}(s),%
    \label{eq:cure-factorization}%
\end{equation}%
where $\TGbots$ and $\widetilde{\TG}_{\rm r}(s)$ have realizations $\Sys_\bot\teq(\TE,\TA,\TB_{\bot}, \TC, \Tzero)$ and $\widetilde{\Sys}_\rsys\teq\RealizationrTilde$, respectively.%
\end{theorem}%
In \Cref{eq:cure-factorization}, the factor $\TGbots$ is high-dimensional with a modified input matrix $\TB_{\bot}$, while the factor $\widetilde{\TG}_{\rm r}(s)$ is small dimensional with a modified output matrix $\TR$ and unity
feedthrough matrix. The interpretation of \Cref{eq:cure-factorization} becomes even more evident in the context of \HTPO\space reduction. One can show \cite{Panzer2013} that if $\TGrs$ is an \HTPO\space approximation of $\TGs$, then $\widetilde{\TG}_{\rm r}(s)$ is all-pass. %
%
%
%
%
As $\widetilde{\TG}_{\rm r}(s)$ has no influence on the amplitude response of $\TGes$, the dynamics of the error are represented by $\TGbots$.
This motivates the CURE scheme derived in the following. \par%
Assume that reduction by rational interpolation has been performed and the approximation error $\TGes\teq\TGs -\TGrs$ is larger (with respect to some norm) than desired. Instead of conducting a new reduction with different parameters or a higher reduced order, CURE further improves the approximation quality by adding an additional term coming from the reduction of $\TGbots$ according to%
\begin{subequations}%
    \begin{align}%
        \TG(s) &= \TGr(s) + \TGe(s), \nonumber\\%
        &= \TGr(s) + \TGbot^{(1)}(s)\widetilde{\TG}_{\rm r}(s),  \label{eq:cure:1} \\%
        &= \TGr(s) + \left(\TGr^{(2)}(s) + \TGe^{(2)}(s) \right)\widetilde{\TG}_{\rm r}(s), \nonumber\\%
        &= \underbrace{\left(\TGr(s) + \TGr^{(2)}(s)\widetilde{\TG}_{\rm r}(s)\right)}_{\TG_{\rsys,tot}^{(2)}(s)} + \underbrace{\TGbot^{(2)}(s)\left(\widetilde{\TG}_{\rm r}^{(2)}(s)\widetilde{\TG}_{\rm r}(s)\right)}_{\TG_{\esys,tot}^{(2)}(s)}.%
        \label{eq:cure:2}%
    \end{align}%
    \label{eq:cure}%
\end{subequations}%
Noticing that \Cref{eq:cure:2} has the same structure as \Cref{eq:cure:1}, this procedure can be repeated until the error $\TG^{(k)}_{\esys,tot}(s)$ is small enough.\par%
The matrices of the overall reduced model $\TG_{\rsys,tot}^{(k)}$ can be assembled easily at every step. Detailed expressions are given in \cite[Theorem 4.2]{Panzer2014}. The overall reduced order results from the sum of each individual reduction $\ro_{tot}\teq\sum\ro_j$. If \HTPO\space reduction is conducted in each step of CURE, then following properties hold.%
\begin{lemma}[\cite{Wolf2014}]%
\label{thm:CUREdSPARK_H2pseudooptimality_properties}%
If in each step of CURE \cref{eq:cure}, $\TGr^{(k)}$ is an \HTPO\space reduced model of $\TGbot^{(k-1)}$, then%
\vspace*{-1em}%
\begin{enumerate}\itemsep0pt%
    \item the reduction error $\normHtwo{\TG^{(k)}_{\esys,tot}}$ decreases monotonically with $k$,%
    \item the overall reduced model $\TG_{\rsys,tot}^{(k)}$ is \HTPO\ as well.%
\end{enumerate}%
\end{lemma}%
The combination of results from \cref{sec:spark} and \cref{sec:cure} finally yields the reduction procedure CUREd SPARK given in \cref{alg:CUREd SPARK}. For all details on implementation, we refer to \cite{Panzer2014,Castagnotto2017}. Typically, stopping criteria for the CURE iteration include the (estimated) approximation error falling below a given threshold or stagnation of the reduced model.\par%
\begin{algorithm}%
    \KwIn{$\Sys\teq\Realization$; initial shifts $\sigma_1,\sigma_2\in\Complex^+$ closed under conjugation}%
    \KwOut{\HTPO\ reduced model $\Sysr\teq\Realizationr$ of order $\ro$}%
    $(a,b) \gets \left( \frac{\sigma_1+\sigma_2}{2}, \sigma_1\,\sigma_2 \right)$ \;%
    $\TG_{r,tot}^{(0)}(s)\gets \left[\;\right]$, $\TG_{\bot}^{(0)}(s) \gets \TGs$, $k\gets 0$ \;%
    \While{CURE not converged}{
        $k = k+1$ \;%
        \While{SPARK not converged}{%
            compute $\cost$, $\grad$, and $\Hess$ at $(a,b)$ w.r.t. $\TG_{\bot}^{(k-1)}(s)$ \;%
            compute the trust-region step \;%
            update $(a,b)$\;%
        }%
        compute $\TV^{(k)}$ according to \cref{eq:V-spark} with optimal $\sigma_1^{opt}=a+\sqrt{a^2-b}$ and $\sigma_2^{opt}=a-\sqrt{a^2-b}$\;%
        return $\TG_{r}^{(k)}(s)$ as in \cref{eq:ROM HTPO ab} with optimal $(a,b)$ \;%
        assemble $\TG_{r,tot}^{(k)}(s)$ according to \cite[Theorem 4.2]{Panzer2014}\;%
        compute $\TB_\bot^{(k)} = \TB_\bot^{(k-1)} - \TE\TV^{(k)}\left(\TEr^{(k)}\right)^{-1}\TBr^{(k)}$ and define $\TG_{\bot}^{(k)}(s)$ according to \cref{thm:Error Factorization CURE}\;%
    }%
    $\ro\gets 2\,k$ \;%
    \caption{CUREd SPARK for strictly proper \DAE s (adapted from \cite{Panzer2014})}%
    \label{alg:CUREd SPARK}%
\end{algorithm}%
%
%

%% file: numericalresults.tex
\section{Numerical Results}%
\label{sec:numericalresults}%
In this section, we apply the reduction scheme as presented in \Cref{fig:overall_procedure} to DAE benchmark models of different indices and structures. Special structure of the system matrices is exploited to efficiently compute the projectors $\Specleftfin$ and $\Specrightfin$ required for determining the strictly proper and polynomial parts of the transfer function. To reduce the strictly proper part $\TG^\strictlyproper(s)$, we apply the CUREd SPARK method derived in \Cref{sec:curedspark}, where we monitor the convergence of the \HtwoText\ norm of the reduced model as the stopping criterion. Every trust-region optimization has been started at $a\teq b\teq10^{-4}$. For the polynomial part $\TP(s)$, we determine a~minimal realization as described in \Cref{subsec:MOR_projection}.\par%
Note that all algorithms are independent of the specific index and structure of the model as long as analytic expressions for the spectral projectors $\Specleftfin$ and $\Specrightfin$ are known. The numerical expe\-ri\-ments were conducted using sss and sssMOR, open-source MATLAB toolboxes for analysis and reduction of large-scale linear models \cite{Castagnotto2017}. The functions \mcode{DAEmor}, \mcode{cure} and \mcode{spark} used for the following results are available for download within sssMOR. \ToEditor{By the time of submission, the functions might still not be available for download. They will be however before publication.}\par%
\subsection{Semi-Explicit Index 1 DAE}%
\label{sec:numericalresults_index1SE}%
As a first test case, we consider a semi-explicit DAE of index~1. Such systems typically arise in computational fluid dynamics and power systems modelling \cite{Benner2017}. Their structure is given by%
\begin{equation}%
    \label{eq:SEIndex1_Structure}%
    \begin{bmatrix} \TE_{11} & \TE_{12} \\ \Tzero & \Tzero \end{bmatrix}\begin{bmatrix} \dot{\Tx}_1 \\ \dot{\Tx}_2\end{bmatrix}=\begin{bmatrix} \TA_{11} & \TA_{12} \\ \TA_{21} & \TA_{22} \end{bmatrix}\begin{bmatrix}\Tx_1 \\ \Tx_2\end{bmatrix}+\begin{bmatrix}%
    \TB_1 \\ \TB_2\end{bmatrix}\Tu\;,\qquad \Ty=\begin{bmatrix} \TC_1 & \TC_2 \end{bmatrix}\begin{bmatrix}\Tx_1 \\ \Tx_2\end{bmatrix}\;,%
\end{equation}%
where the matrices $\TE_{11}$ and $\TA_{22}-\TA_{21}\,\TE_{11}^{-1}\,\TE_{12}$ are assumed to be nonsingular. Note that system \Cref{eq:SEIndex1_Structure} has a proper transfer function $\TGs$ and, hence, the polynomial part of $\TGs$ is constant, i.\,e., $\TPs\equiv\TP$. The analytic expressions for the spectral projectors are given in \cite{Benner2017}.\par%
In the following, we investigate the DAE system ``BIPS/97'' created by the Brazilian Electrical Energy Research Center (CEPEL). It is available online from the MOR Wiki \cite{RommesMORWikiPowersystems} under the name MIMO46. Since this system is of MIMO-type, we restrict ourselves to the channel $G_{42,42}$, i.\,e., $u_{42}\rightarrow y_{42}$. Furthermore, simple row and column reordering has been applied in order to obtain the semi-explicit structure given in \Cref{eq:SEIndex1_Structure}. The model is of order $\fo\teq13250$, whereby  $n_{\finite}\teq1664$. As $\TE_{11}\teq\TI_{n_{\finite}}$ and $\TE_{12}\teq\Tzero$ holds for this model, $\Specleftfin$ and $\Specrightfin$ simplify to%
\begin{equation}%
    \Specleftfin = \begin{bmatrix} \TI_{n_{\finite}} & -\TA_{12}\,\TA_{22}^{-1} \\ \Tzero & \Tzero\end{bmatrix}\;,\qquad \Specrightfin = \begin{bmatrix} \TI_{n_{\finite}} & \Tzero \\ -\TA_{22}^{-1}\,\TA_{21} & \Tzero\end{bmatrix}\;,%
\end{equation}%
and $\TPs\teq-\TC_2\,\TA_{22}^{-1}\,\TB_2$. Note that the strictly proper part of $\TGs$ in \Cref{eq:RealizationofstrictlyproperPart} only involves terms of the form $\Specleftfin\TB$ and $\TC\,\Specrightfin$. The explicit computation of $\Specleftfin$ and $\Specrightfin$, which are generally dense matrices, can be hence avoided by implementing the projector-vector product in terms of matrix-vector multiplications and sparse linear solves. In sssMOR, this is done by creating function handles through the function \mcode{projInd1se} called by \mcode{DAEmor}.\par%
\Cref{image:ind1se_freqresp} shows the result of model reduction after 16 CURE iterations. The polynomial part is constant and its minimal realization is of order 1. Therefore, the resulting reduced order is $\ro = 16\cdot 2 + 1 =33$. Note that for this special case, the constant contribution $\TPs\teq-\TC_2\,\TA_{22}^{-1}\,\TB_2$ could be added as a~feedthrough term without increasing the order of the reduced model by one. Here, we prefer to illustrate the more general strategy presented in \cref{fig:overall_procedure}.\par%
\Cref{image:ind1se_h2norm_of_Gr_and_stopping_criterion} shows the improvement of the reduced-order model over the iterations in CURE. \Cref{image:cure_norm_of_ROM} depicts the \HTN\ of the reduced-order model $\TG_{r,tot}^{(k)}(s)$ showing a monotonic increase and stagnation. \Cref{image:cure_stopping_criterion} depicts the relative increase of the \HtwoText-norm of the reduced model over the iterations, which has been used as stopping criterion. At $k\teq16$ the desired tolerance of $10^{-6}$ is achieved. The frequency response of the error system $\TG_{e,tot}^{(k)}(s)\teq\TGs-\TG_{r,tot}^{(k)}(s)$ is shown in \Cref{image:ind1se_error}. As one can see, the overall error decreases over the iterations.\par%
Finally, we stress out the main advantage of the proposed procedure: the reduced model in \Cref{image:ind1se_freqresp} is an \HTPO\ approximation obtained without any prior specification of reduced order and interpolation frequencies. In addition, the reduced-order model preserves stability, as expected from theoretical considerations.\par%
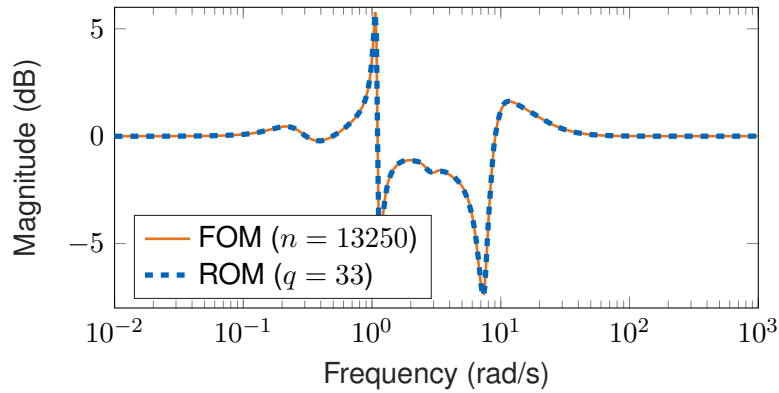
\begin{figure}[htb]%
    \centering
    \pgfplotsset{yticklabel style={text width=1.5em,align=right}}%
    \input{ind1se_freqresp}%
    \caption{Reduction of the BIPS/97 (MIMO46) benchmark model.\ToEditor{layout=spanning one-column preferred; online-version=colored; print-version=grayscale}}%
    \label{image:ind1se_freqresp}%
\end{figure}%
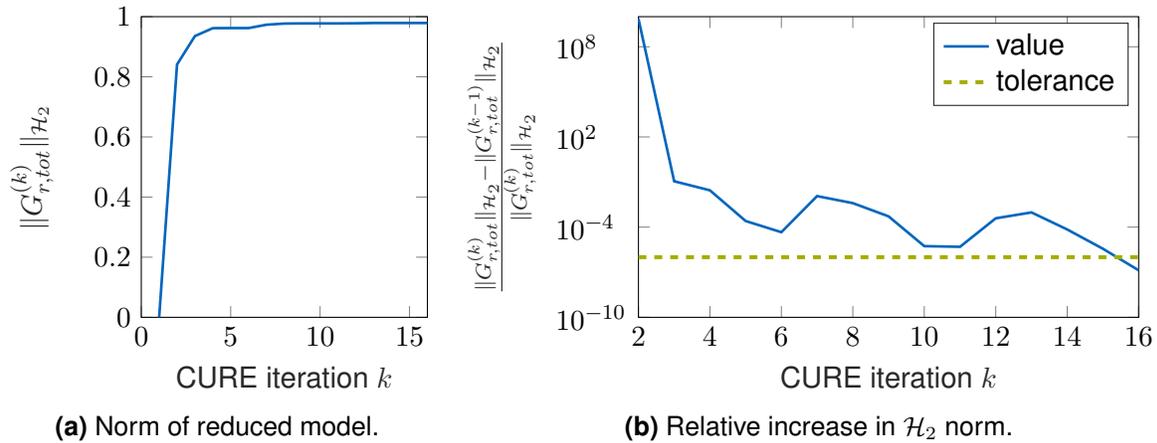
\begin{figure}[htb]%
    \centering%
    \captionsetup[subfigure]{justification=centering}%
    \begin{minipage}[b]{.4\linewidth}%
        \setlength{\myheight}{4cm}%
        \setlength{\mywidth}{4cm}%
        \centering\pgfplotsset{yticklabel style={text width=2.0em,align=right}}\input{ind1se_h2norm}%
        \subcaption{Norm of reduced model.}\label{image:cure_norm_of_ROM}%
    \end{minipage}%
    \hfill%
    \begin{minipage}[b]{.6\linewidth}%
        \setlength{\myheight}{4cm}%
        \setlength{\mywidth}{7cm}%
        \pgfplotsset{yticklabel style={text width=2.5em,align=left}}\input{ind1se_stopcrit}%
        \subcaption{Relative increase in \HtwoText\ norm.}\label{image:cure_stopping_criterion}%
    \end{minipage}%
    \caption{Adaptation of the reduced model for the BIPS/97 (MIMO46) model during CURE.\ToEditor{layout=spanning two-columns preferred or separated into two figures; online-version=colored; print-version=grayscale}}%
    \label{image:ind1se_h2norm_of_Gr_and_stopping_criterion}%
\end{figure}%
\begin{figure}[!ht]%
    \centering%
    \pgfplotsset{yticklabel style={text width=2.5em,align=right}}%
    \input{ind1se_error}%
    \caption{Magnitude plot of the error over the CURE iterations for the  BIPS/97 (MIMO46) model.\ToEditor{layout=spanning one-column preferred; online-version=colored; print-version=grayscale}}%
    \label{image:ind1se_error}%
\end{figure}%
\subsection{Stokes-Like Index 2 DAE}%
\label{sec:numericalresults_index2}%
As a second example, we consider a Stokes-like system of index~2. These systems arise in computational fluid dynamics where the flow of an incompressible fluid is modeled by the Navier-Stokes equation. Linearization and discretization in space by the finite element method leads to a DAE system of the form%
\begin{equation}%
    \label{eq:StokesLikeIndex2_Structure}%
    \begin{bmatrix} \TE_{11} & \Tzero \\ \Tzero & \Tzero \end{bmatrix}\begin{bmatrix} \dot{\Tx}_1 \\ \dot{\Tx}_2\end{bmatrix}=\begin{bmatrix} \TA_{11} & \TA_{12} \\ \TA_{21} & \Tzero \end{bmatrix}\begin{bmatrix}\Tx_1 \\ \Tx_2\end{bmatrix}+\begin{bmatrix}%
    \TB_1 \\ \TB_2\end{bmatrix}\Tu\;,\qquad%
    \Ty=\begin{bmatrix} \TC_1 & \TC_2 \end{bmatrix}\begin{bmatrix}\Tx_1 \\ \Tx_2\end{bmatrix}\;.%
\end{equation}%
If the matrices $\TE_{11}$ and $\TA_{21}\,\TE_{11}^{-1}\,\TA_{12}$ are nonsingular, then the DAE \Cref{eq:StokesLikeIndex2_Structure} is of index~2 and analytic expressions for the spectral projectors exist \cite{Benner2017}.\par%
In our experiments, we use a computer generated model as described in \cite[p.~34ff.]{Schmidt2007} of dimension $n=19039$ and dynamical order $n_\finite\teq12640$. Since for this model $\TE_{11}\teq\TI_{n_\finite}$ holds, the spectral projectors take the form%
\begin{equation}%
    \label{eq:StokesLikeIndex2_SpectralProjectors}%
    \Specleftfin=\begin{bmatrix} \TK & -\TK\,\TA_{11}\,\TA_{12}\left(\TA_{21}\,\TA_{12}\right)^{-1}\\ \Tzero & \Tzero \end{bmatrix}\;,\quad\Specrightfin=\begin{bmatrix} \TK & \Tzero \\ -\left( \TA_{21}\,\TA_{12}\right)^{-1}\TA_{21}\,\TA_{11}\,\TK & \Tzero\end{bmatrix}\;,%
\end{equation}%
where $\TK\teq \TI_{n_\finite}-\TA_{12}\left(\TA_{21}\,\TA_{12}\right)^{-1}\TA_{21}$. Note that the explicit computation of the dense matrices $\Specleftfin$ and $\Specrightfin$ is not necessary. Instead, sparse matrix-vector multiplications and linear solves are performed, which significantly reduces memory consumption.\par%
\Cref{image:ind2stokes_freqresp} shows the reduction result after only four CURE iterations. Since the transfer function of the original model is strictly proper, the resulting reduced-order model is of ODE-type and has the order $q\teq8$.\par%
\begin{figure}[htb]%
    \centering%
    \pgfplotsset{yticklabel style={text width=2.0em,align=right}}%
    \input{ind2stokes_freqresp}%
    \caption{Reduction of the Stokes-like benchmark model.\ToEditor{layout=spanning one-column preferred; online-version=colored; print-version=grayscale}}%
    \label{image:ind2stokes_freqresp}%
\end{figure}
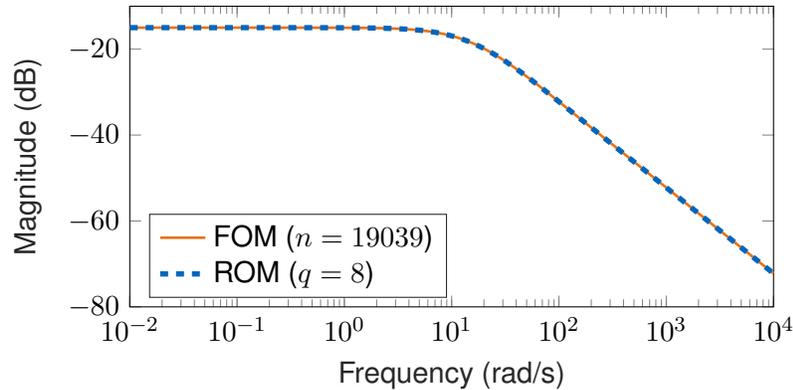%
As in the previous example, we observe in \Cref{image:ind2stokes_error} a~steady decrease of the overall error during the CURE iteration. Due to the simpler dynamics in this case, the CURE iteration converges fast, see \Cref{image:ind2stokes_stopcrit},
and no significant improvement was noted after four steps. Again, note that the reduced model in \Cref{image:ind2stokes_freqresp} was obtained fully automatically while preserving stability.\par%
%
%
%
%
\begin{figure}[htb]%
    \centering%
    \pgfplotsset{yticklabel style={text width=2.5em,align=right}}%
    \input{ind2stokes_error}%
    \caption{Magnitude plot of the error over CURE iterations for the Stokes-like model.\ToEditor{layout=spanning one-column preferred; online-version=colored; print-version=grayscale}}%
    \label{image:ind2stokes_error}%
\end{figure}%
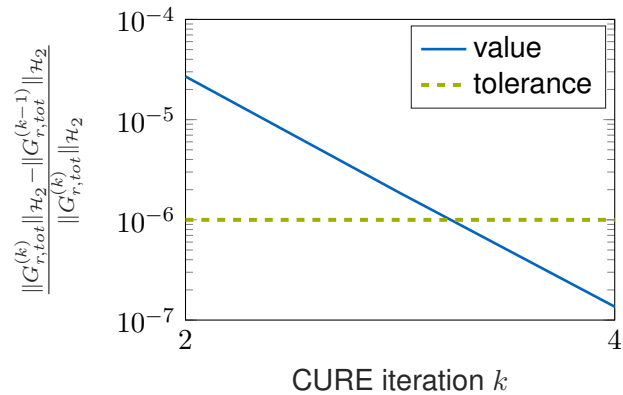
\begin{figure}[htb]%
    \centering%
    \setlength{\myheight}{4cm}%
    \setlength{\mywidth}{6cm}%
    \pgfplotsset{yticklabel style={text width=2.5em,align=right}}%
    \input{ind2stokes_stopcrit}%
    \caption{Relative increase in \HtwoText-norm for the Stokes-like model.\ToEditor{layout=spanning one-column preferred; online-version=colored; print-version=grayscale}}%
    \label{image:ind2stokes_stopcrit}%
\end{figure}%
%
%
%
%
%

%% file: ind1se_freqresp.tex
%
%
\definecolor{mycolor1}{rgb}{0.00000,0.40000,0.74000}%
\begin{tikzpicture}

\begin{axis}[%
width=8.558cm,
height=4cm,
at={(0cm,0cm)},
scale only axis,
xmode=log,
xmin=0.01,
xmax=1000,
xlabel style={font=\color{white!15!black}},
xlabel={Frequency (rad/s)},
ymin=-8,
ymax=6,
ylabel style={font=\color{white!15!black}},
ylabel={Magnitude (dB)},
axis background/.style={fill=white},
legend style={at={(0.03,0.03)}, anchor=south west, legend cell align=left, align=left, draw=white!15!black}
]
\addplot [color=blue!10!orange, line width=1.0pt]
  table[row sep=crcr]{%
0.01	0.000227421193452098\\
0.0102334021219164	0.000239055616560962\\
0.0104722518988843	0.00025131108908523\\
0.0107166764803286	0.000264222937003004\\
0.0109668059833687	0.000277828668724411\\
0.0112227735620851	0.000292168121233305\\
0.0114847154784029	0.000307283616579614\\
0.0117527711746295	0.000323220129512292\\
0.0120270833476851	0.000340025466938566\\
0.0123077980250667	0.000357750460190127\\
0.0125950646425836	0.00037644917107101\\
0.0128890361239089	0.000396179112495884\\
0.0131898689619867	0.000417001485032572\\
0.0134977233023394	0.000438981430432952\\
0.0138127630283201	0.000462188303307495\\
0.014135155848354	0.000486695962538156\\
0.0144650733852165	0.000512583083617119\\
0.0148026912673951	0.000539933493733045\\
0.0151481892225835	0.000568836531104023\\
0.0155017511733577	0.000599387430557412\\
0.0158635653350859	0.000631687737263677\\
0.0162338243161228	0.000665845750623322\\
0.0166127252203429	0.000701977000997725\\
0.0170004697520672	0.000740204761192841\\
0.017397264323438	0.000780660596104296\\
0.0178033201643011	0.000823484952635437\\
0.0182188534346517	0.000868827793782648\\
0.0186440853397049	0.000916849279706752\\
0.0190792422476527	0.000967720499826674\\
0.0195245558101686	0.00102162425958798\\
0.0199802630857255	0.00107875592678535\\
0.0204466066657912	0.00113932434107208\\
0.0209238348039698	0.00120355279288064\\
0.0214122015481573	0.00127168007625553\\
0.0219119668757815	0.00134396162189705\\
0.0224233968321985	0.00142067071673029\\
0.0229467636723194	0.0015020998166599\\
0.0234823460055428	0.00158856196033669\\
0.0240304289440697	0.00168039229135673\\
0.0245913042546805	0.00177794969816419\\
0.0251652705140539	0.00188161858027274\\
0.025752633267712	0.00199181075150015\\
0.0263537051926739	0.00210896749080337\\
0.0269688062639069	0.00223356175177921\\
0.0275982639246618	0.00236610054409164\\
0.0282424132607844	0.00250712750002018\\
0.0289015971790951	0.00265722563999019\\
0.0295761665899325	0.00281702035296594\\
0.0302664805939569	0.00298718260791772\\
0.0309729066733141	0.0031684324139108\\
0.0316958208872612	0.00336154254751666\\
0.0324356080723581	0.00356734256716579\\
0.0331926620473319	0.00378672313576977\\
0.0339673858227221	0.00402064067394284\\
0.0347601918154198	0.00427012236720939\\
0.0355715020682139	0.00453627155227516\\
0.0364017484744614	0.00482027350844432\\
0.0372513730080021	0.00512340168187628\\
0.0381208279584389	0.00544702437075398\\
0.0390105761719099	0.00579261190319839\\
0.0399210912974805	0.00616174433579707\\
0.0408528580392856	0.00655611970939185\\
0.0418063724145576	0.006977562891806\\
0.0427821420176762	0.00742803504253509\\
0.0437806862903817	0.00790964373634611\\
0.0448025367982949	0.00842465377786612\\
0.045848237513891	0.00897549874598342\\
0.046918345106078	0.0095647932990783\\
0.0480134292365345	0.0101953462819514\\
0.0491340728629636	0.010870174661116\\
0.0502808725494248	0.0115925183282466\\
0.0514544387839093	0.0123658557948749\\
0.0526553963033276	0.0131939208090828\\
0.0538843844260822	0.0140807199172358\\
0.055142057392403	0.0150305509858005\\
0.0564290847126254	0.01604802269734\\
0.0577461515235982	0.0171380750216771\\
0.0590939589534097	0.0183060006582568\\
0.0604732244946265	0.0195574674279837\\
0.0618846823862439	0.0208985415868039\\
0.0633290840045511	0.0223357120063948\\
0.0648071982631197	0.0238759151535041\\
0.0663198120221268	0.0255265607779163\\
0.067867730507233	0.0272955581786007\\
0.069451777738237	0.0291913428976472\\
0.0710727969677342	0.0312229036399619\\
0.0727316511300146	0.0333998091804983\\
0.0744292233004376	0.0357322349534003\\
0.0761664171655289	0.0382309889608956\\
0.0779441575040495	0.0409075365651313\\
0.0797633906792928	0.0437740236251917\\
0.0816250851428723	0.0468432973480763\\
0.0835302319502678	0.0501289240924087\\
0.0854798452884041	0.0536452032263852\\
0.0874749630155442	0.0574071759540986\\
0.089516647213783	0.061430627857729\\
0.0916059847544371	0.0657320836299607\\
0.09374408787663	0.0703287922413831\\
0.0959320947793824	0.0752387004233543\\
0.0981711702275219	0.0804804120214887\\
0.100462506171734	0.086073130296853\\
0.102807322383087	0.0920365797717342\\
0.105206867102362	0.0983909036111657\\
0.107662417704549	0.105156531853004\\
0.110175281378839	0.112354014995113\\
0.112746795824495	0.120003816549852\\
0.115378329962966	0.128126057106086\\
0.118071284666619	0.13674020128045\\
0.120827093504478	0.14586467756442\\
0.123647223505372	0.155516419633639\\
0.126533175938894	0.165710316016498\\
0.12948648711459	0.1764585533426\\
0.132508729199795	0.187769836565506\\
0.135601511056563	0.199648467858582\\
0.138766479098131	0.212093264305089\\
0.142005318165368	0.225096293371256\\
0.14531975242369	0.238641404725938\\
0.148711546280895	0.252702537717549\\
0.152182505326439	0.267241786364322\\
0.155734477292613	0.282207208854057\\
0.159369353038178	0.297530377437297\\
0.163089067554933	0.313123678491978\\
0.166895600997802	0.328877393210846\\
0.170790979738943	0.344656618639338\\
0.174777277446468	0.360298128717514\\
0.178856616188346	0.375607327448682\\
0.183031167562061	0.390355512520787\\
0.187303153850644	0.404277747506464\\
0.191674849205682	0.417071731593578\\
0.196148580857943	0.428398150992076\\
0.200726730356257	0.437883083939476\\
0.205411734835307	0.445123092306972\\
0.210206088313016	0.449693640740106\\
0.215112343018217	0.451161405434223\\
0.220133110749303	0.44910083328806\\
0.225271064264598	0.443114957885639\\
0.230528938705171	0.432859959748823\\
0.235909533050864	0.418072296253448\\
0.241415711610302	0.398596492757351\\
0.247050405545683	0.374411004729913\\
0.25281661443315	0.345649099260981\\
0.258717407859592	0.312611640892301\\
0.264755927056707	0.275769135244525\\
0.270935386573205	0.235751415878726\\
0.277259075986048	0.193324844431086\\
0.283730361651621	0.149358576761504\\
0.290352688497781	0.104782989060833\\
0.297129581857733	0.0605444231271386\\
0.304064649346707	0.0175607756371291\\
0.311161582782436	-0.0233179273428355\\
0.318424160150465	-0.0613408292791781\\
0.325856247615323	-0.0958776288929179\\
0.333461801578636	-0.126429673852535\\
0.341244870785289	-0.152631586150417\\
0.349209598478727	-0.174245080384169\\
0.357360224606579	-0.191147060674371\\
0.365701088077749	-0.203314136947968\\
0.374236629072198	-0.21080548456786\\
0.382971391404628	-0.213745596518896\\
0.39191002494334	-0.212308046244476\\
0.40105728808555	-0.206700966496619\\
0.410418050290471	-0.197154599959255\\
0.419997294671531	-0.183911009246234\\
0.42980012064908	-0.167215847937751\\
0.439831746665023	-0.147311979519963\\
0.450097512960805	-0.124434672288636\\
0.46060288442024	-0.0988080797397981\\
0.471353453478691	-0.0706427235523287\\
0.482354943100147	-0.0401337188244705\\
0.493613209823792	-0.00745951059704645\\
0.505134246881676	0.0272190789093232\\
0.51692418738916	0.0637596650450399\\
0.528989307609815	0.102039127072503\\
0.541336030296538	0.141954578463282\\
0.55397092811064	0.183424603040594\\
0.566900727120744	0.226390875502623\\
0.580132310383338	0.270820289338675\\
0.593672721606913	0.316707728779692\\
0.607529168901607	0.364079645684639\\
0.621709028616383	0.412998640603363\\
0.636219849265749	0.463569303933203\\
0.651069355548146	0.515945655554448\\
0.666265452458115	0.570340639932652\\
0.681816229494448	0.627038305202739\\
0.697729964966554	0.68640954372034\\
0.71401513040134	0.748932637345939\\
0.73068039505295	0.8152203945376\\
0.747734630517759	0.886056487713656\\
0.765186915457083	0.962444861676464\\
0.783046540430119	1.04567806201368\\
0.801323012839689	1.13743349940851\\
0.820026061993413	1.23991185485027\\
0.839165644283016	1.35604054631068\\
0.858751948484518	1.48978019771674\\
0.878795401182132	1.64659860755911\\
0.899306672318762	1.83422472345862\\
0.920296680876041	2.06388288231872\\
0.941776600686952	2.35236418648839\\
0.963757866384109	2.72553099019454\\
0.986252179486878	3.22391070230746\\
1.00927151463057	3.90834972688898\\
1.03282812594103	4.83695159258799\\
1.05693455355799	5.76583423442045\\
1.08160363031071	4.61628340113224\\
1.10684848854941	-0.609822501092321\\
1.13268256713615	-5.04417930027726\\
1.15911961859889	-5.46850638548221\\
1.18617371645248	-4.57297422944238\\
1.21385926269063	-3.76384127211645\\
1.24219099545262	-3.16729927241909\\
1.27118399686903	-2.7306814241169\\
1.30085370109057	-2.40393658138208\\
1.33121590250431	-2.15331403164104\\
1.36228676414165	-1.95685824697318\\
1.39408282628258	-1.8000610443428\\
1.42662101526074	-1.67307603906128\\
1.45991865247398	-1.56903535835879\\
1.49399346360526	-1.483027844388\\
1.52886358805873	-1.41146533298118\\
1.5645475886161	-1.35167909343809\\
1.60106446131832	-1.30165532849538\\
1.63843364557798	-1.25985619792245\\
1.6766750345277	-1.22509450036174\\
1.71580898561	-1.19644346500184\\
1.75585633141447	-1.17317167871822\\
1.79683839076772	-1.15469814479731\\
1.83877698008233	-1.14056415484313\\
1.88169442497056	-1.1304181352188\\
1.9256135721292	-1.12400889392866\\
1.9705578015018	-1.12118321298929\\
2.01655103872475	-1.1218853818886\\
2.06361776786386	-1.12615793750614\\
2.11178304444824	-1.13414370166122\\
2.16107250880838	-1.14608901687987\\
2.21151239972549	-1.16234701568951\\
2.26312956839953	-1.18337779155957\\
2.31595149274315	-1.20973905323308\\
2.37000629200933	-1.24205545173636\\
2.42532274176035	-1.28094641278004\\
2.48193028918626	-1.32688108199765\\
2.53985906878073	-1.37991878099806\\
2.59913991838293	-1.43929792271817\\
2.65980439559376	-1.50288737382782\\
2.72188479457518	-1.56665801801245\\
2.78541416324177	-1.62457520116488\\
2.85042632085343	-1.66948509074561\\
2.9169558760188	-1.69525635818229\\
2.98503824511873	-1.69942153005398\\
3.05470967115997	-1.6846392980228\\
3.1260072430687	-1.65784572014536\\
3.19896891543454	-1.62775121427199\\
3.27363352871525	-1.60235153436633\\
3.35004082991313	-1.58746484420359\\
3.42823149373397	-1.58618606290885\\
3.50824714423979	-1.59877297211672\\
3.59013037700707	-1.62288485079623\\
3.67392478180207	-1.65456867989357\\
3.75967496578547	-1.6900673156062\\
3.84742657725851	-1.72745947828429\\
3.93722632996348	-1.76684689689332\\
4.02912202795135	-1.80904453304452\\
4.12316259102975	-1.85404682822249\\
4.21939808080503	-1.90065181476752\\
4.31787972733202	-1.94754648457168\\
4.41865995638594	-1.99466811835022\\
4.5217924173707	-2.04350961100216\\
4.62733201187869	-2.09644775555352\\
4.73533492291712	-2.15598986217498\\
4.8458586448165	-2.22439908485558\\
4.95896201383722	-2.30365175624056\\
5.07470523949047	-2.39561726160915\\
5.19314993659021	-2.50233553687773\\
5.31435915805324	-2.62623360653324\\
5.4383974284648	-2.77023997872305\\
5.56533077842765	-2.93787264299621\\
5.69522677971282	-3.13337654132107\\
5.82815458123085	-3.3619316704723\\
5.96418494584247	-3.62990772325413\\
6.10339028802862	-3.94507871195941\\
6.24584471243962	-4.31649935063015\\
6.39162405334401	-4.75305217062089\\
6.54080591499825	-5.25807788387315\\
6.69346971295866	-5.81743705043053\\
6.84969671635745	-6.3903961686364\\
7.00957009116563	-6.91554639790103\\
7.17317494446561	-7.28811084947656\\
7.34059836975721	-7.32706128098399\\
7.51192949332097	-6.87398712796875\\
7.68725952166374	-5.97900129073661\\
7.86668179007158	-4.8808428898318\\
8.05029181229598	-3.78842168860186\\
8.23818733139961	-2.79258291116867\\
8.43046837178897	-1.92151459385051\\
8.62723729246145	-1.18179271613313\\
8.82859884149515	-0.567244957548224\\
9.03466021181053	-0.0641268862503072\\
9.24553109823357	0.343084951367131\\
9.46132375589077	0.668805049133875\\
9.68215305996709	0.925954760845949\\
9.90813656685868	1.12613532559709\\
10.1393945767529	1.27960535822942\\
10.3760501976691	1.39514958022752\\
10.6182294109938	1.48000264907387\\
10.866061138546	1.53990098959481\\
11.119677311207	1.57929815164025\\
11.3792129391532	1.60173504753801\\
11.6448061837269	1.61023305152614\\
11.9165984309856	1.60751362293437\\
12.1947343669674	1.59600130813515\\
12.4793620547131	1.57774765990913\\
12.7706330130864	1.55440100009657\\
13.0687022974335	1.52724038977568\\
13.373728582125	1.49723983272155\\
13.6858742450252	1.46513370818996\\
14.0053054539322	1.43147186203592\\
14.3321922550357	1.39666324774827\\
14.6667086634397	1.36101043780956\\
15.0090327557974	1.32473738170421\\
15.359346765109	1.28801187404985\\
15.7178371777316	1.25096339802185\\
16.0846948326536	1.21369666273286\\
16.4601150230855	1.17630120140943\\
16.8442976004232	1.1388576165444\\
17.2374470806362	1.10144123612151\\
17.6397727531405	1.06412397768207\\
18.0514887922111	1.02697511190509\\
18.4728143709963	0.990061439539776\\
18.9039737781922	0.953447212073512\\
19.3451965374405	0.917193978374435\\
19.7967175295133	0.881360439382535\\
20.2587771173502	0.846002335396235\\
20.7316212740123	0.811172363350113\\
21.2155017136245	0.776920112581056\\
21.7106760253726	0.743292007868964\\
22.2174078106288	0.71033125234552\\
22.7359668232772	0.678077767177184\\
23.2666291133146	0.646568128448981\\
23.8096771738036	0.615835504060005\\
24.3654000912547	0.585909594785518\\
24.9340936995188	0.556816584202033\\
25.5160607372718	0.528579102143444\\
26.1116110091746	0.501216205926288\\
26.7210615507944	0.474743382856655\\
27.3447367973758	0.44917257658698\\
27.9829687565512	0.424512238805798\\
28.6360971850812	0.400767406613262\\
29.3044697697214	0.377939804866074\\
29.9884423123103	0.356027971863924\\
30.6883789191764	0.33502740605883\\
31.4046521949675	0.314930731018049\\
32.1376434410028	0.295727875638928\\
32.8877428582551	0.277406266554481\\
33.6553497550709	0.25995102972845\\
34.4408727597382	0.243345198371668\\
35.2447300380159	0.227569924484838\\
36.0673495157403	0.212604691526246\\
36.9091691066278	0.198427525907232\\
37.7706369453937	0.185015205229542\\
38.6522116263126	0.172343461395724\\
39.554362447347	0.160387176945539\\
40.4775696599732	0.149120573196275\\
41.422324724839	0.138517388990361\\
42.3891305733878	0.12855104907663\\
43.3785018755899	0.119194821368705\\
44.3909653139217	0.11042196253166\\
45.4270598637405	0.102205851543941\\
46.4873370802026	0.0945201110625249\\
47.5723613918789	0.0873387165842894\\
48.6827104012228	0.0806360935437044\\
49.8189751920516	0.07438720261622\\
50.9817606442042	0.0685676136075123\\
52.1716857555435	0.0631535684015362\\
53.3893839714736	0.0581220335155968\\
54.6355035221488	0.0534507428694577\\
55.9107077675529	0.0491182314186091\\
57.2156755506325	0.045103860330779\\
58.5511015586724	0.0413878344006191\\
59.9176966931062	0.0379512124018483\\
61.3161884479577	0.0347759110702538\\
62.7473212971158	0.0318447033962812\\
64.2118570906476	0.0291412118840324\\
65.7105754603627	0.0266498974053937\\
67.2442742348425	0.024356044245468\\
68.8137698641567	0.0222457418992122\\
70.4198978544929	0.0203058641407188\\
72.0635132129305	0.0185240458467257\\
73.7454909025955	0.0168886580156368\\
75.466726308439	0.0153887813832498\\
77.2281357138864	0.0140141789973881\\
79.0306567886135	0.012755268075745\\
80.8752490877046	0.0116030914352897\\
82.7628945624634	0.0105492887474516\\
84.6945980831458	0.00958606784102748\\
86.6713879738923	0.00870617624476257\\
88.6943165601471	0.00790287313311694\\
90.7644607288536	0.00716990181236945\\
92.882922501725	0.00650146285916446\\
95.0508296218951	0.00589218800023437\\
97.2693361542618	0.0053371147997797\\
99.5396230998423	0.00483166220017591\\
101.862899024469	0.00437160694211867\\
104.240400702156	0.00395306087240772\\
106.673393773486	0.00357244913185147\\
109.163173419361	0.00322648920256869\\
111.711065050482	0.00291217078456866\\
114.318425012915	0.00262673646645092\\
116.986641310131	0.00236766315550934\\
119.717134341897	0.00213264423891309\\
122.511357660412	0.00191957246026057\\
125.370798744092	0.00172652351363568\\
128.296979789415	0.00155174037897363\\
131.291458521247	0.00139361844504343\\
134.355829022083	0.0012506914860958\\
137.491722580642	0.00112161857150444\\
140.700808560269	0.00100517199049541\\
143.984795287601	0.000900226264489639\\
147.345430961983	0.000805748296653031\\
150.784504586105	0.000720788674417969\\
154.303846918356	0.000644474099444491\\
157.905331447418	0.000576000876835968\\
161.590875389592	0.000514629357048559\\
165.362440709418	0.000459679195020145\\
169.222035164104	0.000410525275340908\\
173.171713372335	0.000366594150286382\\
177.213577908036	0.000327360848685436\\
181.34978041965	0.000292345934369118\\
185.582522775552	0.000261112719604216\\
189.914058236194	0.000233264567558365\\
194.346692653602	0.000208442244581974\\
198.882785698881	0.000186321305899947\\
203.524752118358	0.000166609515677815\\
208.275063019051	0.000149044313939469\\
213.136247184144	0.000133390349207301\\
218.110892419152	0.000119437097856559\\
223.201646929523	0.000106996589901499\\
228.411220730383	9.59012579661567e-05\\
233.742387089182	8.600192176293e-05\\
239.197984002024	7.716591598096e-05\\
244.780915704444	6.92753650368879e-05\\
250.49415421745	6.22256043416986e-05\\
256.340740929651	5.59237445407927e-05\\
262.323788216312	5.02873728774247e-05\\
268.446481096197	4.52433839492927e-05\\
274.712078927081	4.07269311911503e-05\\
281.123917140846	3.66804897265661e-05\\
287.685409019059	3.30530210940319e-05\\
294.400047510003	2.97992304694824e-05\\
301.271407088116	2.68789073885287e-05\\
308.303145656828	2.42563414186558e-05\\
315.499006495809	2.18998048869776e-05\\
322.862820253673	1.97810953319929e-05\\
330.398506987186	1.78751310823547e-05\\
338.110078248068	1.61595939209202e-05\\
346.001639218511	1.46146134132183e-05\\
354.077390896527	1.32224881273775e-05\\
362.341632332315	1.19674393871213e-05\\
370.798762916817	1.08353937567424e-05\\
379.453284723694	9.8137908677089e-06\\
388.309804905961	8.89141359766994e-06\\
397.37303814856	8.0582379481764e-06\\
406.647809178185	7.30530028755781e-06\\
416.139055331671	6.62457992045863e-06\\
425.851829184341	6.00889516731339e-06\\
435.791301239703	5.45181133760532e-06\\
445.96276268191	4.94755923146669e-06\\
456.371628192476	4.49096291024602e-06\\
467.023438832733	4.07737564059516e-06\\
477.92386499356	3.7026231094894e-06\\
489.078709413959	3.36295296322956e-06\\
500.493910270095	3.05499004555703e-06\\
512.175544336424	2.77569660200482e-06\\
524.129830220606	2.52233687382657e-06\\
536.363131673925	2.2924456475639e-06\\
548.881960978968	2.083800208663e-06\\
561.69298241638	1.89439537127433e-06\\
574.803015812535	1.72242122358083e-06\\
588.219040169996	1.56624326657271e-06\\
601.948197382727	1.4243846781884e-06\\
615.997796038017	1.29551047331383e-06\\
630.375315307128	1.17841334170388e-06\\
645.08840892677	1.0720009555326e-06\\
660.144909273488	9.75284636640006e-07\\
675.552831533164	8.87369163609863e-07\\
691.320377967813	8.07443656962057e-07\\
707.455942281988	7.34773361166784e-07\\
723.968114091088	6.68692321552314e-07\\
740.865683493957	6.08596726596837e-07\\
758.157645752212	5.53938985036302e-07\\
775.853206078783	5.04222342994372e-07\\
793.961784538228	4.5899606042127e-07\\
812.493021061405	4.17851065838206e-07\\
831.456780577206	3.80416006455357e-07\\
850.863158264058	3.46353730307975e-07\\
870.722484923992	3.153580519043e-07\\
891.045332482152	2.87150886246177e-07\\
911.842519614658	2.61479608501967e-07\\
933.125117507824	2.38114706834755e-07\\
954.904455751808	2.16847622309861e-07\\
977.1921283718	1.97488826026731e-07\\
1000	1.79866048614411e-07\\
};
\addlegendentry{FOM ($n=13250$)}

\addplot [color=mycolor1, dashed, line width=2.0pt]
  table[row sep=crcr]{%
0.01	-0.00417557713432224\\
0.0102334021219164	-0.00413657936381374\\
0.0104722518988843	-0.00409573766903311\\
0.0107166764803286	-0.00405296661112666\\
0.0109668059833687	-0.00400817678280477\\
0.0112227735620851	-0.00396127463027204\\
0.0114847154784029	-0.00391216226733574\\
0.0117527711746295	-0.00386073728138467\\
0.0120270833476851	-0.00380689253088265\\
0.0123077980250667	-0.00375051593404049\\
0.0125950646425836	-0.00369149024829656\\
0.0128890361239089	-0.00362969284019627\\
0.0131898689619867	-0.00356499544531596\\
0.0134977233023394	-0.0034972639177612\\
0.0138127630283201	-0.00342635796882326\\
0.014135155848354	-0.00335213089432\\
0.0144650733852165	-0.00327442929012859\\
0.0148026912673951	-0.0031930927553873\\
0.0151481892225835	-0.00310795358281718\\
0.0155017511733577	-0.00301883643559191\\
0.0158635653350859	-0.00292555801009624\\
0.0162338243161228	-0.00282792668396535\\
0.0166127252203429	-0.00272574214863124\\
0.0170004697520672	-0.00261879502567034\\
0.017397264323438	-0.00250686646606854\\
0.0178033201643011	-0.00238972773157356\\
0.0182188534346517	-0.00226713975712327\\
0.0186440853397049	-0.00213885269331521\\
0.0190792422476527	-0.00200460542778061\\
0.0195245558101686	-0.00186412508417082\\
0.0199802630857255	-0.00171712649742847\\
0.0204466066657912	-0.00156331166375612\\
0.0209238348039698	-0.00140236916366767\\
0.0214122015481573	-0.00123397355619796\\
0.0219119668757815	-0.00105778474223561\\
0.0224233968321985	-0.000873447294679567\\
0.0229467636723194	-0.000680589752812278\\
0.0234823460055428	-0.000478823878052258\\
0.0240304289440697	-0.000267743867860291\\
0.0245913042546805	-4.69255241878609e-05\\
0.0251652705140539	0.000184074627573258\\
0.025752633267712	0.000425720273677769\\
0.0263537051926739	0.00067849630625743\\
0.0269688062639069	0.000942909882446233\\
0.0275982639246618	0.00121949155485986\\
0.0282424132607844	0.00150879648071639\\
0.0289015971790951	0.00181140571776141\\
0.0295761665899325	0.00212792761628192\\
0.0302664805939569	0.00245899931753054\\
0.0309729066733141	0.00280528837031304\\
0.0316958208872612	0.00316749447886029\\
0.0324356080723581	0.00354635139680053\\
0.0331926620473319	0.00394262898390125\\
0.0339673858227221	0.00435713544429062\\
0.0347601918154198	0.00479071976714969\\
0.0355715020682139	0.00524427439354702\\
0.0364017484744614	0.00571873813587261\\
0.0372513730080021	0.00621509937958312\\
0.0381208279584389	0.00673439960059832\\
0.0390105761719099	0.00727773723562373\\
0.0399210912974805	0.00784627194717352\\
0.0408528580392856	0.00844122933002139\\
0.0418063724145576	0.00906390611129818\\
0.0427821420176762	0.00971567590264876\\
0.0437806862903817	0.01039799556957\\
0.0448025367982949	0.0111124122907491\\
0.045848237513891	0.011860571388544\\
0.046918345106078	0.0126442250210765\\
0.0480134292365345	0.0134652418368529\\
0.0491340728629636	0.0143256177041713\\
0.0502808725494248	0.0152274876404477\\
0.0514544387839093	0.0161731390805568\\
0.0526553963033276	0.0171650266389729\\
0.0538843844260822	0.0182057885376563\\
0.055142057392403	0.0192982648906833\\
0.0564290847126254	0.0204455180574531\\
0.0577461515235982	0.0216508552993316\\
0.0590939589534097	0.0229178539995193\\
0.0604732244946265	0.0242503897332783\\
0.0618846823862439	0.0256526675048993\\
0.0633290840045511	0.0271292564992341\\
0.0648071982631197	0.0286851287287128\\
0.0663198120221268	0.0303257019909681\\
0.067867730507233	0.0320568875869582\\
0.069451777738237	0.0338851432829737\\
0.0710727969677342	0.0358175320306229\\
0.0727316511300146	0.0378617869837888\\
0.0744292233004376	0.0400263833663147\\
0.0761664171655289	0.0423206177433665\\
0.0779441575040495	0.0447546952247338\\
0.0797633906792928	0.0473398250685401\\
0.0816250851428723	0.0500883250437207\\
0.0835302319502678	0.0530137347282212\\
0.0854798452884041	0.0561309376385353\\
0.0874749630155442	0.0594562916671726\\
0.089516647213783	0.0630077666963476\\
0.0916059847544371	0.066805087391072\\
0.09374408787663	0.070869877963569\\
0.0959320947793824	0.0752258040273625\\
0.0981711702275219	0.0798987043744347\\
0.100462506171734	0.0849167024247476\\
0.102807322383087	0.0903102829831028\\
0.105206867102362	0.0961123145204237\\
0.107662417704549	0.102357990169128\\
0.110175281378839	0.109084651675945\\
0.112746795824495	0.116331449440961\\
0.115378329962966	0.124138778420195\\
0.118071284666619	0.132547414409265\\
0.120827093504478	0.141597259146586\\
0.123647223505372	0.151325588207775\\
0.126533175938894	0.161764687459534\\
0.12948648711459	0.172938770005875\\
0.132508729199795	0.184860099200326\\
0.135601511056563	0.1975243241092\\
0.138766479098131	0.210905188844705\\
0.142005318165368	0.224949038944556\\
0.14531975242369	0.239569946266589\\
0.148711546280895	0.254646815675291\\
0.152182505326439	0.270024466601778\\
0.155734477292613	0.285521217185541\\
0.159369353038178	0.300945542120018\\
0.163089067554933	0.316123250842545\\
0.166895600997802	0.33093343569391\\
0.170790979738943	0.345345413402741\\
0.174777277446468	0.35944039781833\\
0.178856616188346	0.373393892440063\\
0.183031167562061	0.387395427629987\\
0.187303153850644	0.401502044015704\\
0.191674849205682	0.415465587498728\\
0.196148580857943	0.428623431896612\\
0.200726730356257	0.43995093221809\\
0.205411734835307	0.448301600575935\\
0.210206088313016	0.452737858427877\\
0.215112343018217	0.452783422484067\\
0.220133110749303	0.448476937177138\\
0.225271064264598	0.440226892964258\\
0.230528938705171	0.428556203438641\\
0.235909533050864	0.413836672715705\\
0.241415711610302	0.39608313053451\\
0.247050405545683	0.374855351822899\\
0.25281661443315	0.349320596673144\\
0.258717407859592	0.318531199568284\\
0.264755927056707	0.281906879276965\\
0.270935386573205	0.239748976422574\\
0.277259075986048	0.193466143940816\\
0.283730361651621	0.145285556036445\\
0.290352688497781	0.0975848990991829\\
0.297129581857733	0.0522572846733812\\
0.304064649346707	0.010422589613878\\
0.311161582782436	-0.0275067279126848\\
0.318424160150465	-0.0615830866044804\\
0.325856247615323	-0.0920712417891157\\
0.333461801578636	-0.119262782157295\\
0.341244870785289	-0.143366005757685\\
0.349209598478727	-0.164452535019998\\
0.357360224606579	-0.182440579162661\\
0.365701088077749	-0.197101952989193\\
0.374236629072198	-0.208087247847154\\
0.382971391404628	-0.214968037503409\\
0.39191002494334	-0.217296065854084\\
0.40105728808555	-0.21467707336854\\
0.410418050290471	-0.206851904948886\\
0.419997294671531	-0.193771446476111\\
0.42980012064908	-0.175647600893654\\
0.439831746665023	-0.152963181919433\\
0.450097512960805	-0.126431087047856\\
0.46060288442024	-0.0969059797824363\\
0.471353453478691	-0.065265222649809\\
0.482354943100147	-0.032284318135159\\
0.493613209823792	0.00146750701284098\\
0.505134246881676	0.0356926468269795\\
0.51692418738916	0.0703823210014467\\
0.528989307609815	0.105785745439942\\
0.541336030296538	0.142338331841785\\
0.55397092811064	0.180561445525004\\
0.566900727120744	0.220951362502262\\
0.580132310383338	0.263879294318247\\
0.593672721606913	0.309524944666496\\
0.607529168901607	0.35785994501654\\
0.621709028616383	0.408684546409258\\
0.636219849265749	0.461705637457751\\
0.651069355548146	0.516633521322101\\
0.666265452458115	0.573273654590203\\
0.681816229494448	0.631597057637326\\
0.697729964966554	0.691784420233188\\
0.71401513040134	0.754248711607576\\
0.73068039505295	0.819646616690528\\
0.747734630517759	0.88889060600401\\
0.765186915457083	0.963172792630214\\
0.783046540430119	1.04401101803824\\
0.801323012839689	1.13332846055646\\
0.820026061993413	1.23358175586124\\
0.839165644283016	1.34796067243113\\
0.858751948484518	1.48069741841009\\
0.878795401182132	1.63755071828631\\
0.899306672318762	1.8265785166246\\
0.920296680876041	2.05940072930958\\
0.941776600686952	2.35330359232207\\
0.963757866384109	2.73473622647753\\
0.986252179486878	3.24458919608992\\
1.00927151463057	3.9416803751188\\
1.03282812594103	4.86678127143666\\
1.05693455355799	5.71551852993712\\
1.08160363031071	4.68799807437252\\
1.10684848854941	-0.622763640168786\\
1.13268256713615	-5.14675612323312\\
1.15911961859889	-5.56447705737179\\
1.18617371645248	-4.65318748762329\\
1.21385926269063	-3.83122545756547\\
1.24219099545262	-3.2234019172624\\
1.27118399686903	-2.77640782766875\\
1.30085370109057	-2.43995744676313\\
1.33121590250431	-2.18021472075497\\
1.36228676414165	-1.97521259901127\\
1.39408282628258	-1.81047787811194\\
1.42662101526074	-1.67623195463904\\
1.45991865247398	-1.56569850574624\\
1.49399346360526	-1.47407428679301\\
1.52886358805873	-1.3978881865013\\
1.5645475886161	-1.33458978718961\\
1.60106446131832	-1.28227583354217\\
1.63843364557798	-1.2395007161522\\
1.6766750345277	-1.20513855059714\\
1.71580898561	-1.17827724080634\\
1.75585633141447	-1.15813326035489\\
1.79683839076772	-1.14398206729415\\
1.83877698008233	-1.13510417055507\\
1.88169442497056	-1.13075100679617\\
1.9256135721292	-1.13013727562217\\
1.9705578015018	-1.1324661250523\\
2.01655103872475	-1.13698992699965\\
2.06361776786386	-1.14310314893769\\
2.11178304444824	-1.15045765147122\\
2.16107250880838	-1.15908786930487\\
2.21151239972549	-1.16953483308257\\
2.26312956839953	-1.1829595885796\\
2.31595149274315	-1.20122638740253\\
2.37000629200933	-1.22689407203742\\
2.42532274176035	-1.26295928662364\\
2.48193028918626	-1.31206486619935\\
2.53985906878073	-1.37489131299775\\
2.59913991838293	-1.44801771990044\\
2.65980439559376	-1.52288302770828\\
2.72188479457518	-1.58812093900542\\
2.78541416324177	-1.63487661974236\\
2.85042632085343	-1.66077498144061\\
2.9169558760188	-1.66934624305262\\
2.98503824511873	-1.66664996315179\\
3.05470967115997	-1.65839647992522\\
3.1260072430687	-1.64876598894614\\
3.19896891543454	-1.64039785425891\\
3.27363352871525	-1.63479656702557\\
3.35004082991313	-1.63274930336128\\
3.42823149373397	-1.63463299624849\\
3.50824714423979	-1.64061091321217\\
3.59013037700707	-1.65074986562066\\
3.67392478180207	-1.6650874498277\\
3.75967496578547	-1.68366995677097\\
3.84742657725851	-1.7065738776647\\
3.93722632996348	-1.73391869355657\\
4.02912202795135	-1.7658754266849\\
4.12316259102975	-1.8026735706115\\
4.21939808080503	-1.84460797991223\\
4.31787972733202	-1.89204675220246\\
4.41865995638594	-1.94544087620348\\
4.5217924173707	-2.00533633743595\\
4.62733201187869	-2.07238940850695\\
4.73533492291712	-2.14738597553351\\
4.8458586448165	-2.231265956202\\
4.95896201383722	-2.32515414896307\\
5.07470523949047	-2.43039922125581\\
5.19314993659021	-2.54862299582856\\
5.31435915805324	-2.68178270126464\\
5.4383974284648	-2.8322493202456\\
5.56533077842765	-3.00290533355238\\
5.69522677971282	-3.19726435828666\\
5.82815458123085	-3.41961181894094\\
5.96418494584247	-3.67515606373509\\
6.10339028802862	-3.97015416685126\\
6.24584471243962	-4.31191424861568\\
6.39162405334401	-4.7084240068031\\
6.54080591499825	-5.16698840085486\\
6.69346971295866	-5.69039884974776\\
6.84969671635745	-6.26730728902011\\
7.00957009116563	-6.8505891526269\\
7.17317494446561	-7.31892918636255\\
7.34059836975721	-7.44872089906247\\
7.51192949332097	-7.00671628402739\\
7.68725952166374	-6.00554092659555\\
7.86668179007158	-4.78247043193099\\
8.05029181229598	-3.70556167589217\\
8.23818733139961	-2.8431307168107\\
8.43046837178897	-2.00053798431876\\
8.62723729246145	-1.16541907110699\\
8.82859884149515	-0.499472441702591\\
9.03466021181053	-0.030265817540883\\
9.24553109823357	0.321460786402192\\
9.46132375589077	0.620275056745148\\
9.68215305996709	0.884520957613844\\
9.90813656685868	1.10890106750019\\
10.1393945767529	1.28716456978381\\
10.3760501976691	1.41958835681125\\
10.6182294109938	1.51159895197343\\
10.866061138546	1.57066536301634\\
11.119677311207	1.60417006015566\\
11.3792129391532	1.61845037308395\\
11.6448061837269	1.61856214771538\\
11.9165984309856	1.60837582038341\\
12.1947343669674	1.59078296127561\\
12.4793620547131	1.56791124606034\\
12.7706330130864	1.54131005268425\\
13.0687022974335	1.51209822877675\\
13.373728582125	1.48107706010376\\
13.6858742450252	1.44881484746263\\
14.0053054539322	1.41570959824623\\
14.3321922550357	1.38203533993716\\
14.6667086634397	1.34797636653442\\
15.0090327557974	1.31365266045049\\
15.359346765109	1.27913887551184\\
15.7178371777316	1.24447861363903\\
16.0846948326536	1.20969524300773\\
16.4601150230855	1.17480015056217\\
16.8442976004232	1.13979906378743\\
17.2374470806362	1.1046968900651\\
17.6397727531405	1.06950138776327\\
18.0514887922111	1.03422588773588\\
18.4728143709963	0.998891217318446\\
18.9039737781922	0.963526934162472\\
19.3451965374405	0.92817194917671\\
19.7967175295133	0.892874602472314\\
20.2587771173502	0.85769225031853\\
20.7316212740123	0.822690421922327\\
21.2155017136245	0.787941609844568\\
21.7106760253726	0.753523764776405\\
22.2174078106288	0.71951857219402\\
22.7359668232772	0.686009593387986\\
23.2666291133146	0.653080355247313\\
23.8096771738036	0.620812471187659\\
24.3654000912547	0.589283869503323\\
24.9340936995188	0.558567195465002\\
25.5160607372718	0.528728440392324\\
26.1116110091746	0.499825835732\\
26.7210615507944	0.471909034065721\\
27.3447367973758	0.445018583162878\\
27.9829687565512	0.419185684758225\\
28.6360971850812	0.394432217516907\\
29.3044697697214	0.37077099419732\\
29.9884423123103	0.348206216577635\\
30.6883789191764	0.326734088232767\\
31.4046521949675	0.306343544464162\\
32.1376434410028	0.287017060163989\\
32.8877428582551	0.268731499618732\\
33.6553497550709	0.251458976682594\\
34.4408727597382	0.235167698866041\\
35.2447300380159	0.219822774242967\\
36.0673495157403	0.205386965322174\\
36.9091691066278	0.191821378889799\\
37.7706369453937	0.179086085134267\\
38.6522116263126	0.167140663018657\\
39.554362447347	0.155944671837142\\
40.4775696599732	0.14545805120056\\
41.422324724839	0.135641453394235\\
42.3891305733878	0.126456513211828\\
43.3785018755899	0.117866061075249\\
44.3909653139217	0.109834285587179\\
45.4270598637405	0.102326851709774\\
46.4873370802026	0.0953109805934308\\
47.5723613918789	0.0887554967559002\\
48.6827104012228	0.082630847887218\\
49.8189751920516	0.076909102072225\\
50.9817606442042	0.0715639267130098\\
52.1716857555435	0.0665705529232943\\
53.3893839714736	0.061905728673454\\
54.6355035221488	0.0575476635006328\\
55.9107077675529	0.0534759671709213\\
57.2156755506325	0.049671584293723\\
58.5511015586724	0.0461167265436085\\
59.9176966931062	0.0427948038418064\\
61.3161884479577	0.0396903555856465\\
62.7473212971158	0.0367889827878548\\
64.2118570906476	0.0340772817943443\\
65.7105754603627	0.0315427800867202\\
67.2442742348425	0.0291738745399218\\
68.8137698641567	0.0269597723932151\\
70.4198978544929	0.0248904351010241\\
72.0635132129305	0.0229565251559954\\
73.7454909025955	0.0211493559172794\\
75.466726308439	0.0194608444305501\\
77.2281357138864	0.0178834671901026\\
79.0306567886135	0.0164102187661308\\
80.8752490877046	0.0150345732001374\\
82.7628945624634	0.0137504480572923\\
84.6945980831458	0.0125521710150205\\
86.6713879738923	0.0114344488615525\\
88.6943165601471	0.0103923387753927\\
90.7644607288536	0.00942122175639655\\
92.882922501725	0.00851677808058918\\
95.0508296218951	0.0076749646535861\\
97.2693361542618	0.00689199414121751\\
99.5396230998423	0.00616431576031712\\
101.862899024469	0.00548859761749448\\
104.240400702156	0.00486171048866686\\
106.673393773486	0.00428071293740992\\
109.163173419361	0.00374283767522816\\
111.711065050482	0.00324547907198097\\
114.318425012915	0.00278618172969894\\
116.986641310131	0.00236263003770814\\
119.717134341897	0.00197263863174513\\
122.511357660412	0.00161414368403\\
125.370798744092	0.00128519495563101\\
128.296979789415	0.000983948546462178\\
131.291458521247	0.000708660282182918\\
134.355829022083	0.000457679681000108\\
137.491722580642	0.000229444447002971\\
140.700808560269	2.24754401305455e-05\\
143.984795287601	-0.000164627923752689\\
147.345430961983	-0.00033319188598257\\
150.784504586105	-0.000484472211489199\\
154.303846918356	-0.000619657745578889\\
157.905331447418	-0.000739873695163741\\
161.590875389592	-0.000846184665672966\\
165.362440709418	-0.000939597482092098\\
169.222035164104	-0.00102106381991668\\
173.171713372335	-0.00109148266922305\\
177.213577908036	-0.0011517026525965\\
181.34978041965	-0.00120252421536979\\
185.582522775552	-0.0012447017043475\\
189.914058236194	-0.00127894534918256\\
194.346692653602	-0.00130592315859319\\
198.882785698881	-0.00132626274183179\\
203.524752118358	-0.00134055306416947\\
208.275063019051	-0.00134934614365476\\
213.136247184144	-0.00135315869507387\\
218.110892419152	-0.00135247372581057\\
223.201646929523	-0.00134774208726276\\
228.411220730383	-0.00133938398454315\\
233.742387089182	-0.00132779044640696\\
239.197984002024	-0.00131332475669646\\
244.780915704444	-0.0012963238480134\\
250.49415421745	-0.00127709965794202\\
256.340740929651	-0.00125594044779267\\
262.323788216312	-0.00123311208352734\\
268.446481096197	-0.00120885927850172\\
274.712078927081	-0.0011834067973801\\
281.123917140846	-0.00115696062069478\\
287.685409019059	-0.00112970906943899\\
294.400047510003	-0.00110182388917749\\
301.271407088116	-0.00107346129320108\\
308.303145656828	-0.0010447629644276\\
315.499006495809	-0.00101585701580981\\
322.862820253673	-0.000986858909195882\\
330.398506987186	-0.000957872332715086\\
338.110078248068	-0.000928990036917957\\
346.001639218511	-0.00090029463005379\\
354.077390896527	-0.000871859332985975\\
362.341632332315	-0.000843748694413373\\
370.798762916817	-0.000816019267165957\\
379.453284723694	-0.000788720246473693\\
388.309804905961	-0.000761894071192004\\
397.37303814856	-0.000735576989071659\\
406.647809178185	-0.000709799587230679\\
416.139055331671	-0.000684587289045838\\
425.851829184341	-0.00065996081874416\\
435.791301239703	-0.000635936634984631\\
445.96276268191	-0.00061252733477916\\
456.371628192476	-0.000589742029107723\\
467.023438832733	-0.000567586691568116\\
477.92386499356	-0.000546064481421982\\
489.078709413959	-0.000525176042375604\\
500.493910270095	-0.000504919778427157\\
512.175544336424	-0.000485292108058072\\
524.129830220606	-0.00046628769804613\\
536.363131673925	-0.0004478996781431\\
548.881960978968	-0.000430119837778698\\
561.69298241638	-0.000412938805973802\\
574.803015812535	-0.000396346215538749\\
588.219040169996	-0.000380330852650812\\
601.948197382727	-0.000364880792783401\\
615.997796038017	-0.000349983523984544\\
630.375315307128	-0.000335626058406709\\
645.08840892677	-0.000321795032961018\\
660.144909273488	-0.000308476799898504\\
675.552831533164	-0.000295657508139328\\
691.320377967813	-0.000283323176019441\\
707.455942281988	-0.000271459756183931\\
723.968114091088	-0.000260053193258881\\
740.865683493957	-0.00024908947488048\\
758.157645752212	-0.000238554676681351\\
775.853206078783	-0.000228435001712498\\
793.961784538228	-0.000218716814828473\\
812.493021061405	-0.000209386672462064\\
831.456780577206	-0.000200431348207098\\
850.863158264058	-0.000191837854628891\\
870.722484923992	-0.000183593461607119\\
891.045332482152	-0.000175685711598803\\
911.842519614658	-0.00016810243209434\\
933.125117507824	-0.000160831745550118\\
954.904455751808	-0.000153862077069657\\
977.1921283718	-0.000147182160057025\\
1000	-0.000140781040071065\\
};
\addlegendentry{ROM ($q=33$)}

\end{axis}
\end{tikzpicture}%

%% file: ind1se_h2norm.tex
%
\definecolor{mycolor1}{rgb}{0.00000,0.40000,0.74000}%
\begin{tikzpicture}

\begin{axis}[%
width=0.951\mywidth,
height=\myheight,
at={(0\mywidth,0\myheight)},
scale only axis,
xmin=0,
xmax=16,
xlabel style={font=\color{white!15!black}},
xlabel={CURE iteration $k$},
ymin=0,
ymax=1,
ylabel style={font=\color{white!15!black}},
ylabel={$\|G_{r,tot}^{(k)}\|_{\mathcal{H}_2}$},
axis background/.style={fill=white}
]
\addplot [color=mycolor1, line width=1.0pt, forget plot]
  table[row sep=crcr]{%
1	1.16774996286596e-10\\
2	0.840948446980415\\
3	0.935455526812182\\
4	0.961774894004547\\
5	0.962020405874063\\
6	0.962064664212379\\
7	0.973224128907168\\
8	0.977066322613256\\
9	0.977575383173968\\
10	0.977580701409414\\
11	0.977585491064271\\
12	0.977963146214397\\
13	0.978894280222319\\
14	0.978962623254166\\
15	0.978966175574187\\
16	0.978966306627833\\
};
\end{axis}
\end{tikzpicture}%

%% file: ind1se_stopcrit.tex
%
\definecolor{mycolor1}{rgb}{0.00000,0.40000,0.74000}%
\begin{tikzpicture}

\begin{axis}[%
width=0.95\mywidth,
height=\myheight,
at={(0\mywidth,0\myheight)},
scale only axis,
xmin=2,
xmax=16,
xtick={ 2,  4,  6,  8, 10, 12, 14, 16},
xlabel style={font=\color{white!15!black}},
xlabel={CURE iteration $k$},
ymode=log,
ymin=1e-10,
ymax=10000000000,
ylabel style={font=\color{white!15!black}},
ylabel={$\frac{\|G_{r,tot}^{(k)}\|_{\mathcal{H}_2}-\|G_{r,tot}^{(k-1)}\|_{\mathcal{H}_2}}{\|G_{r,tot}^{(k)}\|_{\mathcal{H}_2}}$},
axis background/.style={fill=white},
legend style={legend cell align=left, align=left, draw=white!15!black}
]
\addplot [color=mycolor1, line width=1.0pt]
  table[row sep=crcr]{%
2	7201442719.80738\\
3	0.112381537977879\\
4	0.0281353484350613\\
5	0.000255269576120327\\
6	4.6005612818939e-05\\
7	0.0115994954496375\\
8	0.00394790222721109\\
9	0.000521009218033892\\
10	5.4402305311339e-06\\
11	4.89949816958141e-06\\
12	0.000386314193057622\\
13	0.000952115641091834\\
14	6.98165606112084e-05\\
15	3.62865745487624e-06\\
16	1.33869432745846e-07\\
};
\addlegendentry{value}

\addplot [color=green!36!orange, dashed, line width=1.5pt]
  table[row sep=crcr]{%
2	1e-06\\
16	1e-06\\
};
\addlegendentry{tolerance}

\end{axis}
\end{tikzpicture}%

%% file: ind2stokes_freqresp.tex
%
%
\definecolor{mycolor1}{rgb}{0.00000,0.40000,0.74000}%
\begin{tikzpicture}

\begin{axis}[%
width=8.558cm,
height=4cm,
at={(0cm,0cm)},
scale only axis,
xmode=log,
xmin=0.01,
xmax=10000,
xlabel style={font=\color{white!15!black}},
xlabel={Frequency (rad/s)},
ymin=-80,
ymax=-10,
ylabel style={font=\color{white!15!black}},
ylabel={Magnitude (dB)},
axis background/.style={fill=white},
legend style={at={(0.03,0.03)}, anchor=south west, legend cell align=left, align=left, draw=white!15!black}
]
\addplot [color=blue!10!orange, line width=1.0pt]
  table[row sep=crcr]{%
0.01	-14.9871213284681\\
0.0102807322383087	-14.9871214793744\\
0.0105693455355799	-14.9871216388719\\
0.010866061138546	-14.9871218074513\\
0.0111711065050482	-14.9871219856302\\
0.0114847154784029	-14.9871221739511\\
0.0118071284666619	-14.9871223729981\\
0.0121385926269063	-14.9871225833757\\
0.0124793620547131	-14.9871228057293\\
0.0128296979789415	-14.9871230407437\\
0.0131898689619867	-14.9871232891402\\
0.0135601511056563	-14.9871235516762\\
0.0139408282628258	-14.987123829161\\
0.0143321922550357	-14.9871241224451\\
0.0147345430961984	-14.9871244324222\\
0.0151481892225835	-14.9871247600552\\
0.0155734477292613	-14.987125106336\\
0.0160106446131832	-14.9871254723347\\
0.0164601150230855	-14.9871258591708\\
0.0169222035164104	-14.9871262680312\\
0.017397264323438	-14.9871267001705\\
0.0178856616188346	-14.9871271569119\\
0.0183877698008233	-14.9871276396553\\
0.0189039737781922	-14.9871281498901\\
0.0194346692653602	-14.9871286891717\\
0.0199802630857255	-14.987129259155\\
0.0205411734835306	-14.9871298615925\\
0.0211178304444824	-14.9871304983288\\
0.0217106760253726	-14.9871311713169\\
0.0223201646929523	-14.9871318826191\\
0.0229467636723194	-14.9871326344214\\
0.0235909533050864	-14.9871334290271\\
0.0242532274176035	-14.9871342688701\\
0.0249340936995188	-14.987135156532\\
0.0256340740929651	-14.9871360947331\\
0.0263537051926739	-14.9871370863461\\
0.0270935386573205	-14.9871381344177\\
0.0278541416324177	-14.9871392421583\\
0.0286360971850812	-14.9871404129708\\
0.0294400047510004	-14.9871416504411\\
0.0302664805939569	-14.9871429583645\\
0.0311161582782436	-14.9871443407505\\
0.0319896891543454	-14.9871458018469\\
0.0328877428582551	-14.9871473461282\\
0.0338110078248069	-14.9871489783301\\
0.0347601918154198	-14.9871507034613\\
0.0357360224606579	-14.9871525268068\\
0.0367392478180207	-14.9871544539673\\
0.0377706369453937	-14.9871564908483\\
0.0388309804905961	-14.987158643694\\
0.0399210912974805	-14.9871609191113\\
0.0410418050290471	-14.9871633240724\\
0.0421939808080503	-14.9871658659638\\
0.0433785018755899	-14.9871685525704\\
0.0445962762681909	-14.9871713921276\\
0.045848237513891	-14.9871743933631\\
0.0471353453478691	-14.9871775654637\\
0.048458586448165	-14.9871809181612\\
0.0498189751920516	-14.9871844617432\\
0.0512175544336424	-14.9871882070655\\
0.0526553963033276	-14.9871921656252\\
0.0541336030296538	-14.9871963495584\\
0.0556533077842765	-14.9872007716926\\
0.0572156755506325	-14.9872054455926\\
0.0588219040169996	-14.9872103855938\\
0.0604732244946265	-14.9872156068393\\
0.0621709028616383	-14.9872211253435\\
0.0639162405334401	-14.9872269580308\\
0.0657105754603627	-14.9872331227845\\
0.0675552831533165	-14.9872396385125\\
0.069451777738237	-14.9872465251941\\
0.071401513040134	-14.9872538039512\\
0.0734059836975721	-14.9872614970971\\
0.0754667263084389	-14.9872696282272\\
0.0775853206078784	-14.9872782222746\\
0.0797633906792928	-14.9872873055908\\
0.0820026061993413	-14.9872969060305\\
0.0843046837178897	-14.9873070530307\\
0.0866713879738923	-14.9873177777044\\
0.0891045332482151	-14.9873291129442\\
0.0916059847544371	-14.9873410934997\\
0.0941776600686952	-14.9873537561087\\
0.0968215305996709	-14.9873671395942\\
0.0995396230998423	-14.9873812849949\\
0.102334021219164	-14.9873962356802\\
0.105206867102362	-14.9874120374916\\
0.108160363031071	-14.9874287388747\\
0.11119677311207	-14.9874463910378\\
0.114318425012915	-14.9874650481032\\
0.117527711746295	-14.9874847672653\\
0.120827093504478	-14.9875056089793\\
0.124219099545262	-14.9875276371398\\
0.127706330130864	-14.9875509192794\\
0.131291458521247	-14.987575526768\\
0.134977233023394	-14.9876015350435\\
0.138766479098131	-14.9876290238319\\
0.142662101526074	-14.9876580773885\\
0.146667086634397	-14.9876887847754\\
0.150784504586105	-14.9877212401146\\
0.155017511733577	-14.9877555428816\\
0.159369353038178	-14.9877917982072\\
0.163843364557798	-14.9878301172136\\
0.168442976004232	-14.9878706173352\\
0.173171713372335	-14.9879134226839\\
0.178033201643011	-14.9879586644357\\
0.183031167562061	-14.9880064812228\\
0.188169442497056	-14.9880570195624\\
0.193451965374405	-14.988110434303\\
0.198882785698881	-14.9881668890973\\
0.204466066657912	-14.9882265569015\\
0.210206088313016	-14.9882896205064\\
0.216107250880838	-14.9883562730852\\
0.222174078106288	-14.9884267188\\
0.228411220730382	-14.9885011734127\\
0.234823460055428	-14.9885798649464\\
0.241415711610302	-14.9886630343814\\
0.248193028918626	-14.98875093638\\
0.255160607372719	-14.988843840089\\
0.262323788216312	-14.988942029926\\
0.269688062639069	-14.9890458064734\\
0.277259075986048	-14.9891554873702\\
0.285042632085343	-14.9892714083008\\
0.293044697697214	-14.9893939240019\\
0.301271407088116	-14.989523409341\\
0.309729066733141	-14.9896602604536\\
0.318424160150465	-14.989804895969\\
0.327363352871525	-14.9899577582395\\
0.336553497550709	-14.9901193147211\\
0.346001639218511	-14.9902900593706\\
0.355715020682139	-14.9904705141454\\
0.365701088077749	-14.9906612305827\\
0.375967496578547	-14.9908627914704\\
0.386522116263126	-14.9910758126042\\
0.397373038148561	-14.9913009446558\\
0.408528580392856	-14.9915388751252\\
0.419997294671531	-14.9917903304202\\
0.431787972733202	-14.9920560780441\\
0.443909653139217	-14.9923369288977\\
0.456371628192476	-14.9926337397235\\
0.46918345106078	-14.9929474156685\\
0.482354943100147	-14.9932789130164\\
0.495896201383721	-14.9936292420229\\
0.509817606442042	-14.9939994699628\\
0.524129830220606	-14.9943907242877\\
0.538843844260822	-14.9948041960194\\
0.55397092811064	-14.9952411432677\\
0.569522677971283	-14.9957028949811\\
0.585511015586724	-14.9961908548759\\
0.601948197382727	-14.9967065056047\\
0.618846823862439	-14.9972514131303\\
0.636219849265749	-14.9978272313362\\
0.654080591499826	-14.998435706885\\
0.672442742348425	-14.9990786843595\\
0.691320377967813	-14.9997581116415\\
0.710727969677342	-15.0004760456031\\
0.73068039505295	-15.0012346580851\\
0.751192949332097	-15.0020362422089\\
0.772281357138865	-15.0028832189938\\
0.793961784538228	-15.0037781443303\\
0.816250851428723	-15.0047237163314\\
0.839165644283016	-15.0057227830308\\
0.862723729246145	-15.0067783505227\\
0.886943165601471	-15.0078935914555\\
0.911842519614657	-15.0090718540056\\
0.9374408787663	-15.0103166712903\\
0.963757866384109	-15.0116317712292\\
0.990813656685867	-15.0130210869184\\
1.01862899024469	-15.014488767509\\
1.04722518988843	-15.0160391895969\\
1.07662417704549	-15.0176769691791\\
1.10684848854941	-15.019406974165\\
1.13792129391532	-15.0212343374924\\
1.16986641310131	-15.0231644708282\\
1.20270833476851	-15.0252030789101\\
1.23647223505372	-15.0273561745453\\
1.27118399686903	-15.0296300942505\\
1.30687022974335	-15.0320315145944\\
1.34355829022083	-15.0345674692217\\
1.38127630283201	-15.0372453666248\\
1.42005318165368	-15.0400730085925\\
1.45991865247398	-15.0430586094548\\
1.50090327557974	-15.0462108160385\\
1.54303846918356	-15.049538728405\\
1.58635653350859	-15.053051921333\\
1.63089067554933	-15.0567604665761\\
1.6766750345277	-15.0606749558778\\
1.72374470806362	-15.0648065247341\\
1.77213577908036	-15.0691668769116\\
1.82188534346517	-15.073768309682\\
1.87303153850644	-15.0786237397368\\
1.9256135721292	-15.0837467298155\\
1.97967175295133	-15.0891515159152\\
2.03524752118358	-15.0948530351345\\
2.09238348039698	-15.100866954017\\
2.15112343018217	-15.1072096973724\\
2.21151239972549	-15.1138984774849\\
2.27359668232772	-15.1209513236345\\
2.33742387089181	-15.1283871118435\\
2.40304289440697	-15.136225594679\\
2.47050405545683	-15.1444874310709\\
2.53985906878073	-15.1531942159593\\
2.61116110091746	-15.1623685096137\\
2.68446481096197	-15.1720338665037\\
2.75982639246618	-15.1822148634648\\
2.83730361651621	-15.1929371270672\\
2.9169558760188	-15.2042273598554\\
2.99884423123103	-15.216113365339\\
3.08303145656828	-15.2286240714248\\
3.16958208872612	-15.2417895520687\\
3.25856247615323	-15.2556410468936\\
3.35004082991313	-15.2702109784235\\
3.44408727597382	-15.2855329667486\\
3.54077390896527	-15.3016418412814\\
3.64017484744614	-15.3185736492838\\
3.74236629072198	-15.3363656608943\\
3.84742657725851	-15.3550563703939\\
3.9554362447347	-15.374685493357\\
4.06647809178186	-15.3952939594606\\
4.18063724145576	-15.4169239006769\\
4.2980012064908	-15.4396186345972\\
4.41865995638594	-15.4634226427266\\
4.54270598637405	-15.4883815434375\\
4.67023438832734	-15.5145420596105\\
4.80134292365345	-15.5419519807219\\
4.93613209823792	-15.5706601193177\\
5.07470523949047	-15.6007162619449\\
5.21716857555435	-15.6321711144542\\
5.36363131673924	-15.6650762418584\\
5.5142057392403	-15.6994840027867\\
5.66900727120743	-15.7354474787734\\
5.82815458123085	-15.7730203985837\\
5.99176966931062	-15.8122570578658\\
6.15997796038017	-15.8532122343816\\
6.33290840045511	-15.8959410992241\\
6.51069355548146	-15.9404991242957\\
6.69346971295866	-15.9869419864616\\
6.88137698641567	-16.0353254687339\\
7.07455942281988	-16.0857053587486\\
7.27316511300146	-16.13813734498\\
7.47734630517759	-16.1926769108654\\
7.68725952166374	-16.2493792270918\\
7.90306567886135	-16.3082990422531\\
8.12493021061405	-16.3694905719469\\
8.35302319502678	-16.4330073863621\\
8.58751948484517	-16.4989022963641\\
8.82859884149515	-16.5672272379768\\
9.07644607288536	-16.6380331551099\\
9.33125117507825	-16.7113698804569\\
9.59320947793824	-16.7872860142908\\
9.86252179486878	-16.8658288010352\\
10.1393945767529	-16.9470440034658\\
10.4240400702156	-17.0309757744279\\
10.7166764803286	-17.1176665261321\\
11.0175281378839	-17.207156797003\\
11.3268256713615	-17.2994851165021\\
11.6448061837269	-17.3946878681387\\
11.9717134341897	-17.4927991513912\\
12.3077980250667	-17.5938506431855\\
12.6533175938894	-17.6978714599457\\
13.0085370109057	-17.8048880212922\\
13.373728582125	-17.9149239167598\\
13.7491722580642	-18.0279997769122\\
14.135155848354	-18.1441331505101\\
14.531975242369	-18.2633383893786\\
14.9399346360526	-18.3856265426969\\
15.359346765109	-18.5110052625\\
15.7905331447418	-18.6394787220352\\
16.2338243161228	-18.771047548625\\
16.6895600997802	-18.9057087725221\\
17.1580898561	-19.0434557930334\\
17.6397727531405	-19.1842783630027\\
18.134978041965	-19.3281625924974\\
18.6440853397049	-19.4750909722114\\
19.1674849205682	-19.6250424168296\\
19.705578015018	-19.7779923283217\\
20.2587771173502	-19.9339126787036\\
20.8275063019052	-20.0927721116521\\
21.4122015481573	-20.2545360619299\\
22.0133110749303	-20.4191668913941\\
22.6312956839953	-20.5866240400967\\
23.2666291133146	-20.7568641907681\\
23.9197984002024	-20.929841444851\\
24.5913042546804	-21.1055075080892\\
25.281661443315	-21.2838118836345\\
25.9913991838293	-21.4647020705651\\
26.7210615507944	-21.648123765775\\
27.4712078927081	-21.83402106715\\
28.2424132607844	-22.0223366761346\\
29.0352688497781	-22.2130120978394\\
29.8503824511873	-22.405987836976\\
30.6883789191764	-22.6012035881318\\
31.5499006495809	-22.7985984190014\\
32.4356080723581	-22.9981109454288\\
33.3461801578637	-23.199679497289\\
34.2823149373397	-23.4032422744334\\
35.2447300380159	-23.608737492125\\
36.2341632332315	-23.8161035155674\\
37.2513730080021	-24.0252789832879\\
38.2971391404628	-24.2362029193473\\
39.3722632996348	-24.4488148344208\\
40.4775696599732	-24.663054815999\\
41.6139055331671	-24.8788636080179\\
42.7821420176762	-25.0961826803633\\
43.9831746665022	-25.3149542887375\\
45.217924173707	-25.5351215254887\\
46.4873370802026	-25.7566283620218\\
47.792386499356	-25.9794196834606\\
49.1340728629636	-26.2034413162814\\
50.5134246881677	-26.4286400496212\\
51.9314993659021	-26.6549636510188\\
53.3893839714736	-26.882360877321\\
54.8881960978967	-27.1107814815231\\
56.4290847126254	-27.3401762162719\\
58.0132310383338	-27.5704968348022\\
59.6418494584247	-27.8016960900118\\
61.3161884479577	-28.0337277324337\\
63.0375315307127	-28.2665465077806\\
64.8071982631197	-28.500108154772\\
66.6265452458115	-28.7343694038762\\
68.4969671635745	-28.9692879775967\\
70.4198978544929	-29.2048225928689\\
72.3968114091088	-29.4409329660694\\
74.4292233004376	-29.677579821084\\
76.5186915457083	-29.9147249007884\\
78.6668179007158	-30.1523309821908\\
80.8752490877046	-30.3903618954087\\
83.1456780577206	-30.6287825464814\\
85.4798452884041	-30.8675589439526\\
87.8795401182132	-31.106658228953\\
90.3466021181053	-31.3460487084527\\
92.882922501725	-31.5856998911241\\
95.4904455751809	-31.8255825252063\\
98.1711702275219	-32.0656686375775\\
100.927151463057	-32.3059315731589\\
103.760501976691	-32.5463460336706\\
106.673393773486	-32.786888114706\\
109.668059833687	-33.0275353400146\\
112.746795824495	-33.2682666919275\\
115.911961859889	-33.5090626368082\\
119.165984309856	-33.7499051445339\\
122.511357660412	-33.9907777010473\\
125.950646425836	-34.2316653131594\\
129.48648711459	-34.4725545049097\\
133.121590250431	-34.7134333049812\\
136.858742450252	-34.9542912248193\\
140.700808560269	-35.1951192273327\\
144.650733852165	-35.4359096862285\\
148.711546280895	-35.6766563362535\\
152.886358805873	-35.9173542148063\\
157.178371777316	-36.1579995955539\\
161.590875389592	-36.3985899148596\\
166.127252203429	-36.639123691958\\
170.790979738943	-36.8796004439377\\
175.585633141447	-37.1200205966483\\
180.514887922111	-37.3603853927255\\
185.582522775552	-37.6006967979348\\
190.792422476527	-37.840957407009\\
196.148580857943	-38.081170350157\\
201.655103872475	-38.3213392013111\\
207.316212740123	-38.5614678891371\\
213.136247184144	-38.801560611703\\
219.119668757815	-39.0416217556014\\
225.271064264598	-39.2816558202099\\
231.595149274315	-39.5216673476126\\
238.096771738036	-39.7616608586239\\
244.780915704444	-40.0016407951887\\
251.652705140539	-40.2416114693481\\
258.717407859592	-40.4815770188191\\
265.980439559376	-40.7215413691598\\
273.447367973758	-40.9615082023791\\
281.123917140846	-41.2014809317778\\
289.015971790951	-41.4414626827518\\
297.129581857733	-41.6814562792041\\
305.470967115997	-41.9214642351959\\
314.046521949675	-42.1614887514169\\
322.862820253673	-42.4015317160336\\
331.926620473319	-42.6415947094634\\
341.244870785289	-42.8816790126156\\
350.824714423979	-43.1217856181415\\
360.673495157403	-43.3619152442379\\
370.798762916817	-43.6020683505773\\
381.208279584389	-43.8422451559345\\
391.91002494334	-44.0824456571111\\
402.912202795135	-44.3226696488021\\
414.22324724839	-44.5629167440281\\
425.851829184341	-44.8031863948364\\
437.806862903817	-45.0434779129748\\
450.097512960805	-45.2837904902636\\
462.733201187869	-45.5241232184585\\
475.723613918789	-45.764475108371\\
489.078709413959	-46.0048451081066\\
502.808725494247	-46.2452321202394\\
516.92418738916	-46.4856350178365\\
531.435915805324	-46.7260526592195\\
546.355035221488	-46.9664839013965\\
561.69298241638	-47.2069276121313\\
577.461515235983	-47.4473826806125\\
593.672721606913	-47.6878480267304\\
610.339028802862	-47.9283226089752\\
627.473212971157	-48.1688054309859\\
645.08840892677	-48.4092955468029\\
663.198120221268	-48.6497920648847\\
681.816229494448	-48.8902941509675\\
700.957009116562	-49.1308010298409\\
720.635132129306	-49.3713119861477\\
740.865683493957	-49.6118263642882\\
761.664171655289	-49.8523435675471\\
783.046540430118	-50.0928630565401\\
805.029181229597	-50.3333843470836\\
827.628945624635	-50.5739070076092\\
850.863158264058	-50.8144306561648\\
874.749630155442	-51.0549549572069\\
899.306672318761	-51.2954796181641\\
924.553109823358	-51.536004385909\\
950.508296218951	-51.7765290432115\\
977.1921283718	-52.0170534052325\\
1004.62506171734	-52.2575773161116\\
1032.82812594103	-52.4981006456928\\
1061.82294109938	-52.7386232864515\\
1091.63173419361	-52.979145150613\\
1122.27735620851	-53.2196661675162\\
1153.78329962966	-53.4601862812002\\
1186.17371645248	-53.7007054482628\\
1219.47343669674	-53.9412236359478\\
1253.70798744092	-54.1817408204697\\
1288.90361239089	-54.4222569855723\\
1325.08729199795	-54.6627721212686\\
1362.28676414165	-54.9032862228355\\
1400.53054539322	-55.1437992898991\\
1439.84795287601	-55.3843113257383\\
1480.26912673951	-55.6248223366777\\
1521.82505326439	-55.8653323316233\\
1564.5475886161	-56.1058413216507\\
1608.46948326536	-56.3463493197462\\
1653.62440709418	-56.5868563404958\\
1700.04697520672	-56.8273623999176\\
1747.77277446468	-57.0678675153089\\
1796.83839076772	-57.3083717050608\\
1847.28143709964	-57.5488749885733\\
1899.14058236194	-57.7893773861281\\
1952.45558101686	-58.0298789187841\\
2007.26730356257	-58.2703796082852\\
2063.61776786386	-58.5108794769504\\
2121.55017136245	-58.7513785475855\\
2181.10892419152	-58.9918768433672\\
2242.33968321985	-59.2323743877965\\
2305.28938705171	-59.4728712045428\\
2370.00629200933	-59.7133673173873\\
2436.54000912547	-59.9538627501337\\
2504.9415421745	-60.1943575265\\
2575.2633267712	-60.4348516700684\\
2647.55927056707	-60.6753452042012\\
2721.88479457518	-60.9158381519687\\
2798.29687565512	-61.1563305360641\\
2876.85409019059	-61.3968223788245\\
2957.61665899326	-61.6373137021047\\
3040.64649346707	-61.8778045272588\\
3126.0072430687	-62.1182948751457\\
3213.76434410028	-62.3587847660348\\
3303.98506987186	-62.5992742196636\\
3396.73858227221	-62.8397632551522\\
3492.09598478727	-63.0802518910075\\
3590.13037700707	-63.3207401451829\\
3690.91691066277	-63.5612280349481\\
3794.53284723694	-63.8017155770252\\
3901.05761719099	-64.0422027874801\\
4010.5728808555	-64.2826896818232\\
4123.16259102975	-64.5231762749114\\
4238.91305733878	-64.7636625810658\\
4357.91301239703	-65.0041486140159\\
4480.25367982949	-65.2446343869264\\
4606.0288442024	-65.4851199124331\\
4735.33492291712	-65.7256052025543\\
4868.27104012228	-65.966090268919\\
5004.93910270095	-66.2065751225572\\
5145.44387839092	-66.4470597739946\\
5289.89307609816	-66.6875442334182\\
5438.3974284648	-66.928028510341\\
5591.07077675529	-67.1685126140548\\
5748.03015812535	-67.4089965532525\\
5909.39589534098	-67.6494803363161\\
6075.29168901608	-67.889963971184\\
6245.84471243962	-68.1304474654623\\
6421.18570906476	-68.370930826298\\
6601.44909273488	-68.6114140606148\\
6786.7730507233	-68.8518971749182\\
6977.29964966554	-69.0923801754209\\
7173.17494446561	-69.3328630680297\\
7374.54909025955	-69.5733458583234\\
7581.57645752212	-69.8138285516589\\
7794.41575040495	-70.0543111530839\\
8013.23012839689	-70.2947936673982\\
8238.1873313996	-70.5352760991595\\
8469.45980831459	-70.7757584527589\\
8707.22484923992	-71.0162407321796\\
8951.6647213783	-71.2567229414717\\
9202.96680876041	-71.4972050842242\\
9461.32375589078	-71.7376871639986\\
9726.93361542618	-71.9781691840947\\
10000	-72.2186511476918\\
};
\addlegendentry{FOM ($n=19039$)}

\addplot [color=mycolor1, dashed, line width=2.0pt]
  table[row sep=crcr]{%
0.01	-14.9851300219167\\
0.0102807322383087	-14.9851301817903\\
0.0105693455355799	-14.9851303507661\\
0.010866061138546	-14.9851305293625\\
0.0111711065050482	-14.9851307181272\\
0.0114847154784029	-14.9851309176391\\
0.0118071284666619	-14.9851311285101\\
0.0121385926269063	-14.9851313513868\\
0.0124793620547131	-14.9851315869529\\
0.0128296979789415	-14.9851318359308\\
0.0131898689619867	-14.9851320990841\\
0.0135601511056563	-14.9851323772198\\
0.0139408282628258	-14.985132671191\\
0.0143321922550357	-14.9851329818992\\
0.0147345430961984	-14.9851333102974\\
0.0151481892225835	-14.9851336573927\\
0.0155734477292613	-14.9851340242495\\
0.0160106446131832	-14.9851344119931\\
0.0164601150230855	-14.9851348218125\\
0.0169222035164104	-14.9851352549646\\
0.017397264323438	-14.9851357127779\\
0.0178856616188346	-14.9851361966563\\
0.0183877698008233	-14.9851367080839\\
0.0189039737781922	-14.9851372486292\\
0.0194346692653602	-14.9851378199499\\
0.0199802630857255	-14.9851384237982\\
0.0205411734835306	-14.9851390620259\\
0.0211178304444824	-14.9851397365904\\
0.0217106760253726	-14.9851404495605\\
0.0223201646929523	-14.9851412031228\\
0.0229467636723194	-14.9851419995882\\
0.0235909533050864	-14.9851428413994\\
0.0242532274176035	-14.9851437311381\\
0.0249340936995188	-14.9851446715328\\
0.0256340740929651	-14.9851456654677\\
0.0263537051926739	-14.9851467159908\\
0.0270935386573205	-14.9851478263238\\
0.0278541416324177	-14.985148999872\\
0.0286360971850812	-14.9851502402343\\
0.0294400047510004	-14.9851515512145\\
0.0302664805939569	-14.9851529368331\\
0.0311161582782436	-14.9851544013394\\
0.0319896891543454	-14.9851559492246\\
0.0328877428582551	-14.9851575852356\\
0.0338110078248069	-14.9851593143895\\
0.0347601918154198	-14.985161141989\\
0.0357360224606579	-14.9851630736389\\
0.0367392478180207	-14.9851651152627\\
0.0377706369453937	-14.9851672731213\\
0.0388309804905961	-14.9851695538321\\
0.0399210912974805	-14.985171964389\\
0.0410418050290471	-14.9851745121841\\
0.0421939808080503	-14.9851772050305\\
0.0433785018755899	-14.9851800511857\\
0.0445962762681909	-14.9851830593776\\
0.045848237513891	-14.9851862388305\\
0.0471353453478691	-14.9851895992943\\
0.048458586448165	-14.9851931510735\\
0.0498189751920516	-14.9851969050592\\
0.0512175544336424	-14.9852008727625\\
0.0526553963033276	-14.9852050663498\\
0.0541336030296538	-14.9852094986798\\
0.0556533077842765	-14.9852141833432\\
0.0572156755506325	-14.9852191347042\\
0.0588219040169996	-14.9852243679446\\
0.0604732244946265	-14.9852298991103\\
0.0621709028616383	-14.9852357451603\\
0.0639162405334401	-14.985241924019\\
0.0657105754603627	-14.9852484546306\\
0.0675552831533165	-14.9852553570178\\
0.069451777738237	-14.9852626523424\\
0.071401513040134	-14.9852703629707\\
0.0734059836975721	-14.9852785125417\\
0.0754667263084389	-14.9852871260395\\
0.0775853206078784	-14.98529622987\\
0.0797633906792928	-14.9853058519416\\
0.0820026061993413	-14.9853160217504\\
0.0843046837178897	-14.985326770471\\
0.0866713879738923	-14.9853381310516\\
0.0891045332482151	-14.9853501383149\\
0.0916059847544371	-14.9853628290645\\
0.0941776600686952	-14.9853762421979\\
0.0968215305996709	-14.9853904188252\\
0.0995396230998423	-14.9854054023949\\
0.102334021219164	-14.9854212388266\\
0.105206867102362	-14.985437976652\\
0.108160363031071	-14.9854556671626\\
0.11119677311207	-14.9854743645668\\
0.114318425012915	-14.9854941261558\\
0.117527711746295	-14.9855150124783\\
0.120827093504478	-14.9855370875255\\
0.124219099545262	-14.9855604189268\\
0.127706330130864	-14.9855850781561\\
0.131291458521247	-14.9856111407502\\
0.134977233023394	-14.9856386865393\\
0.138766479098131	-14.9856677998902\\
0.142662101526074	-14.9856985699643\\
0.146667086634397	-14.9857310909891\\
0.150784504586105	-14.9857654625454\\
0.155017511733577	-14.9858017898712\\
0.159369353038178	-14.9858401841816\\
0.163843364557798	-14.9858807630081\\
0.168442976004232	-14.9859236505563\\
0.173171713372335	-14.9859689780832\\
0.178033201643011	-14.9860168842972\\
0.183031167562061	-14.9860675157792\\
0.188169442497056	-14.9861210274285\\
0.193451965374405	-14.9861775829324\\
0.198882785698881	-14.9862373552638\\
0.204466066657912	-14.9863005272054\\
0.210206088313016	-14.986367291904\\
0.216107250880838	-14.9864378534553\\
0.222174078106288	-14.9865124275224\\
0.228411220730382	-14.9865912419874\\
0.234823460055428	-14.9866745376408\\
0.241415711610302	-14.9867625689085\\
0.248193028918626	-14.9868556046199\\
0.255160607372719	-14.9869539288178\\
0.262323788216312	-14.9870578416144\\
0.269688062639069	-14.9871676600943\\
0.277259075986048	-14.9872837192674\\
0.285042632085343	-14.987406373075\\
0.293044697697214	-14.9875359954509\\
0.301271407088116	-14.9876729814415\\
0.309729066733141	-14.9878177483875\\
0.318424160150465	-14.9879707371704\\
0.327363352871525	-14.9881324135273\\
0.336553497550709	-14.988303269438\\
0.346001639218511	-14.9884838245879\\
0.355715020682139	-14.9886746279098\\
0.365701088077749	-14.9888762592107\\
0.375967496578547	-14.989089330886\\
0.386522116263126	-14.9893144897268\\
0.397373038148561	-14.9895524188241\\
0.408528580392856	-14.9898038395765\\
0.419997294671531	-14.9900695138042\\
0.431787972733202	-14.9903502459768\\
0.443909653139217	-14.9906468855595\\
0.456371628192476	-14.9909603294839\\
0.46918345106078	-14.9912915247487\\
0.482354943100147	-14.9916414711582\\
0.495896201383721	-14.992011224204\\
0.509817606442042	-14.9924018980966\\
0.524129830220606	-14.9928146689554\\
0.538843844260822	-14.9932507781624\\
0.55397092811064	-14.9937115358899\\
0.569522677971283	-14.9941983248072\\
0.585511015586724	-14.9947126039768\\
0.601948197382727	-14.9952559129473\\
0.618846823862439	-14.9958298760527\\
0.636219849265749	-14.9964362069253\\
0.654080591499826	-14.9970767132342\\
0.672442742348425	-14.9977533016562\\
0.691320377967813	-14.9984679830906\\
0.710727969677342	-14.9992228781279\\
0.73068039505295	-15.0000202227815\\
0.751192949332097	-15.0008623744947\\
0.772281357138865	-15.0017518184326\\
0.793961784538228	-15.0026911740705\\
0.816250851428723	-15.0036832020902\\
0.839165644283016	-15.0047308115966\\
0.862723729246145	-15.0058370676646\\
0.886943165601471	-15.0070051992311\\
0.911842519614657	-15.008238607343\\
0.9374408787663	-15.0095408737747\\
0.963757866384109	-15.0109157700283\\
0.990813656685867	-15.0123672667308\\
1.01862899024469	-15.0138995434419\\
1.04722518988843	-15.015516998887\\
1.07662417704549	-15.0172242616321\\
1.10684848854941	-15.0190262012147\\
1.13792129391532	-15.0209279397497\\
1.16986641310131	-15.0229348640267\\
1.20270833476851	-15.0250526381182\\
1.23647223505372	-15.0272872165189\\
1.27118399686903	-15.0296448578376\\
1.30687022974335	-15.0321321390654\\
1.34355829022083	-15.0347559704451\\
1.38127630283201	-15.0375236109683\\
1.42005318165368	-15.0404426845323\\
1.45991865247398	-15.0435211967855\\
1.50090327557974	-15.0467675527001\\
1.54303846918356	-15.0501905749079\\
1.58635653350859	-15.0537995228417\\
1.63089067554933	-15.0576041127281\\
1.6766750345277	-15.0616145384797\\
1.72374470806362	-15.0658414935399\\
1.77213577908036	-15.0702961937369\\
1.82188534346517	-15.0749904012066\\
1.87303153850644	-15.079936449447\\
1.9256135721292	-15.0851472695714\\
1.97967175295133	-15.090636417826\\
2.03524752118358	-15.0964181044429\\
2.09238348039698	-15.1025072238929\\
2.15112343018217	-15.1089193866071\\
2.21151239972549	-15.1156709522243\\
2.27359668232772	-15.1227790644192\\
2.33742387089181	-15.1302616873526\\
2.40304289440697	-15.1381376437703\\
2.47050405545683	-15.1464266547583\\
2.53985906878073	-15.1551493811368\\
2.61116110091746	-15.164327466444\\
2.68446481096197	-15.1739835814221\\
2.75982639246618	-15.184141469875\\
2.83730361651621	-15.1948259957072\\
2.9169558760188	-15.2060631908964\\
2.99884423123103	-15.217880304071\\
3.08303145656828	-15.2303058492855\\
3.16958208872612	-15.2433696544843\\
3.25856247615323	-15.2571029090411\\
3.35004082991313	-15.2715382096368\\
3.44408727597382	-15.2867096036095\\
3.54077390896527	-15.3026526287658\\
3.64017484744614	-15.3194043484924\\
3.74236629072198	-15.3370033808461\\
3.84742657725851	-15.3554899201412\\
3.9554362447347	-15.3749057493947\\
4.06647809178186	-15.3952942418401\\
4.18063724145576	-15.4167003495973\\
4.2980012064908	-15.4391705774895\\
4.41865995638594	-15.4627529399594\\
4.54270598637405	-15.4874968990633\\
4.67023438832734	-15.5134532816477\\
4.80134292365345	-15.5406741740591\\
4.93613209823792	-15.5692127931371\\
5.07470523949047	-15.599123332821\\
5.21716857555435	-15.6304607864905\\
5.36363131673924	-15.6632807461818\\
5.5142057392403	-15.6976391810746\\
5.66900727120743	-15.7335921991255\\
5.82815458123085	-15.7711957973866\\
5.99176966931062	-15.8105056083152\\
6.15997796038017	-15.8515766511367\\
6.33290840045511	-15.894463098902\\
6.51069355548146	-15.9392180730765\\
6.69346971295866	-15.9858934780763\\
6.88137698641567	-16.0345398878826\\
7.07455942281988	-16.0852064954968\\
7.27316511300146	-16.1379411334002\\
7.47734630517759	-16.1927903693016\\
7.68725952166374	-16.2497996764186\\
7.90306567886135	-16.3090136716364\\
8.12493021061405	-16.3704764085945\\
8.35302319502678	-16.4342317067072\\
8.58751948484517	-16.5003234920337\\
8.82859884149515	-16.5687961224735\\
9.07644607288536	-16.6396946685487\\
9.33125117507825	-16.7130651223571\\
9.59320947793824	-16.7889545111804\\
9.86252179486878	-16.8674108983649\\
10.1393945767529	-16.9484832618687\\
10.4240400702156	-17.0322212494384\\
10.7166764803286	-17.1186748178569\\
11.0175281378839	-17.207893771233\\
11.3268256713615	-17.2999272192207\\
11.6448061837269	-17.3948229799471\\
11.9717134341897	-17.4926269541614\\
12.3077980250667	-17.593382496821\\
12.6533175938894	-17.6971298103433\\
13.0085370109057	-17.8039053805106\\
13.373728582125	-17.913741472027\\
13.7491722580642	-18.0266656964452\\
14.135155848354	-18.1427006610012\\
14.531975242369	-18.2618637030902\\
14.9399346360526	-18.384166711887\\
15.359346765109	-18.5096160360098\\
15.7905331447418	-18.6382124741816\\
16.2338243161228	-18.7699513444756\\
16.6895600997802	-18.9048226268712\\
17.1580898561	-19.0428111733736\\
17.6397727531405	-19.1838969797585\\
18.134978041965	-19.3280555129768\\
18.6440853397049	-19.4752580882912\\
19.1674849205682	-19.6254722902195\\
19.705578015018	-19.7786624312291\\
20.2587771173502	-19.9347900417716\\
20.8275063019052	-20.0938143845625\\
21.4122015481573	-20.2556929848916\\
22.0133110749303	-20.4203821670723\\
22.6312956839953	-20.5878375847843\\
23.2666291133146	-20.7580147299049\\
23.9197984002024	-20.9308694003762\\
24.5913042546804	-21.1063581026992\\
25.281661443315	-21.2844383589081\\
25.9913991838293	-21.465068881763\\
26.7210615507944	-21.648209576201\\
27.4712078927081	-21.8338213212621\\
28.2424132607844	-22.0218654870057\\
29.0352688497781	-22.2123031486221\\
29.8503824511873	-22.4050939791445\\
30.6883789191764	-22.6001948373417\\
31.5499006495809	-22.7975581218867\\
32.4356080723581	-22.9971300365935\\
33.3461801578637	-23.1988489972522\\
34.2823149373397	-23.4026444908817\\
35.2447300380159	-23.6084367439512\\
36.2341632332315	-23.8161375308703\\
37.2513730080021	-24.0256523263226\\
38.2971391404628	-24.2368837676924\\
39.3722632996348	-24.4497360827948\\
40.4775696599732	-24.6641198366648\\
41.6139055331671	-24.8799561681906\\
42.7821420176762	-25.0971797112804\\
43.9831746665022	-25.3157396453884\\
45.217924173707	-25.5355987270615\\
46.4873370802026	-25.7567305858371\\
47.792386499356	-25.979115889813\\
49.1340728629636	-26.2027381188644\\
50.5134246881677	-26.4275796252279\\
51.9314993659021	-26.6536184717559\\
53.3893839714736	-26.8808263000756\\
54.8881960978967	-27.1091672640516\\
56.4290847126254	-27.3385979078892\\
58.0132310383338	-27.5690677823416\\
59.6418494584247	-27.8005205662565\\
61.3161884479577	-28.0328954752956\\
63.0375315307127	-28.2661287764384\\
64.8071982631197	-28.5001552714942\\
66.6265452458115	-28.7349096561503\\
68.4969671635745	-28.9703276983182\\
70.4198978544929	-29.2063472089748\\
72.3968114091088	-29.4429088003874\\
74.4292233004376	-29.6799564415008\\
76.5186915457083	-29.9174378296154\\
78.6668179007158	-30.1553046025616\\
80.8752490877046	-30.393512417518\\
83.1456780577206	-30.6320209223781\\
85.4798452884041	-30.8707936438902\\
87.8795401182132	-31.1097978142767\\
90.3466021181053	-31.3490041550983\\
92.882922501725	-31.5883866340867\\
95.4904455751809	-31.8279222077332\\
98.1711702275219	-32.0675905597137\\
100.927151463057	-32.3073738428419\\
103.760501976691	-32.5472564301841\\
106.673393773486	-32.7872246792562\\
109.668059833687	-33.0272667118265\\
112.746795824495	-33.2673722107428\\
115.911961859889	-33.5075322343497\\
119.165984309856	-33.7477390484205\\
122.511357660412	-33.9879859750719\\
125.950646425836	-34.2282672578094\\
129.48648711459	-34.4685779416522\\
133.121590250431	-34.708913767171\\
136.858742450252	-34.9492710772253\\
140.700808560269	-35.1896467351901\\
144.650733852165	-35.4300380534963\\
148.711546280895	-35.6704427313714\\
152.886358805873	-35.9108588007387\\
157.178371777316	-36.1512845793188\\
161.590875389592	-36.3917186300589\\
166.127252203429	-36.6321597261043\\
170.790979738943	-36.8726068206044\\
175.585633141447	-37.1130590207277\\
180.514887922111	-37.3535155653272\\
185.582522775552	-37.5939758057688\\
190.792422476527	-37.8344391894915\\
196.148580857943	-38.0749052459252\\
201.655103872475	-38.3153735744377\\
207.316212740123	-38.5558438340268\\
213.136247184144	-38.7963157345124\\
219.119668757815	-39.0367890290134\\
225.271064264598	-39.2772635075268\\
231.595149274315	-39.51773899145\\
238.096771738036	-39.7582153289102\\
244.780915704444	-39.9986923907829\\
251.652705140539	-40.2391700673003\\
258.717407859592	-40.479648265162\\
265.980439559376	-40.7201269050747\\
273.447367973758	-40.9606059196582\\
281.123917140846	-41.2010852516628\\
289.015971790951	-41.4415648524532\\
297.129581857733	-41.6820446807174\\
305.470967115997	-41.9225247013702\\
314.046521949675	-42.1630048846187\\
322.862820253673	-42.4034852051695\\
331.926620473319	-42.6439656415528\\
341.244870785289	-42.8844461755491\\
350.824714423979	-43.124926791701\\
360.673495157403	-43.3654074768983\\
370.798762916817	-43.6058882200257\\
381.208279584389	-43.8463690116625\\
391.91002494334	-44.0868498438282\\
402.912202795135	-44.3273307097658\\
414.22324724839	-44.5678116037575\\
425.851829184341	-44.8082925209683\\
437.806862903817	-45.0487734573129\\
450.097512960805	-45.2892544093428\\
462.733201187869	-45.5297353741499\\
475.723613918789	-45.7702163492847\\
489.078709413959	-46.0106973326875\\
502.808725494247	-46.2511783226286\\
516.92418738916	-46.4916593176583\\
531.435915805324	-46.7321403165648\\
546.355035221488	-46.972621318337\\
561.69298241638	-47.2131023221345\\
577.461515235983	-47.453583327261\\
593.672721606913	-47.6940643331422\\
610.339028802862	-47.9345453393068\\
627.473212971157	-48.1750263453703\\
645.08840892677	-48.4155073510219\\
663.198120221268	-48.6559883560121\\
681.816229494448	-48.8964693601436\\
700.957009116562	-49.1369503632625\\
720.635132129306	-49.3774313652513\\
740.865683493957	-49.617912366023\\
761.664171655289	-49.8583933655161\\
783.046540430118	-50.0988743636898\\
805.029181229597	-50.3393553605209\\
827.628945624635	-50.5798363560004\\
850.863158264058	-50.8203173501309\\
874.749630155442	-51.0607983429247\\
899.306672318761	-51.3012793344013\\
924.553109823358	-51.5417603245867\\
950.508296218951	-51.7822413135112\\
977.1921283718	-52.0227223012089\\
1004.62506171734	-52.2632032877166\\
1032.82812594103	-52.5036842730731\\
1061.82294109938	-52.7441652573183\\
1091.63173419361	-52.9846462404928\\
1122.27735620851	-53.2251272226379\\
1153.78329962966	-53.4656082037945\\
1186.17371645248	-53.7060891840033\\
1219.47343669674	-53.9465701633043\\
1253.70798744092	-54.187051141737\\
1288.90361239089	-54.4275321193398\\
1325.08729199795	-54.6680130961502\\
1362.28676414165	-54.9084940722044\\
1400.53054539322	-55.1489750475377\\
1439.84795287601	-55.3894560221841\\
1480.26912673951	-55.6299369961763\\
1521.82505326439	-55.8704179695459\\
1564.5475886161	-56.1108989423231\\
1608.46948326536	-56.351379914537\\
1653.62440709418	-56.5918608862154\\
1700.04697520672	-56.832341857385\\
1747.77277446468	-57.0728228280712\\
1796.83839076772	-57.3133037982983\\
1847.28143709964	-57.5537847680897\\
1899.14058236194	-57.7942657374674\\
1952.45558101686	-58.0347467064526\\
2007.26730356257	-58.2752276750653\\
2063.61776786386	-58.5157086433248\\
2121.55017136245	-58.7561896112492\\
2181.10892419152	-58.996670578856\\
2242.33968321985	-59.2371515461615\\
2305.28938705171	-59.4776325131815\\
2370.00629200933	-59.7181134799309\\
2436.54000912547	-59.9585944464237\\
2504.9415421745	-60.1990754126735\\
2575.2633267712	-60.439556378693\\
2647.55927056707	-60.6800373444943\\
2721.88479457518	-60.9205183100889\\
2798.29687565512	-61.1609992754876\\
2876.85409019059	-61.4014802407009\\
2957.61665899326	-61.6419612057385\\
3040.64649346707	-61.8824421706097\\
3126.0072430687	-62.1229231353234\\
3213.76434410028	-62.3634040998878\\
3303.98506987186	-62.603885064311\\
3396.73858227221	-62.8443660286003\\
3492.09598478727	-63.084846992763\\
3590.13037700707	-63.3253279568058\\
3690.91691066277	-63.565808920735\\
3794.53284723694	-63.8062898845568\\
3901.05761719099	-64.0467708482768\\
4010.5728808555	-64.2872518119005\\
4123.16259102975	-64.5277327754329\\
4238.91305733878	-64.7682137388791\\
4357.91301239703	-65.0086947022436\\
4480.25367982949	-65.2491756655307\\
4606.0288442024	-65.4896566287447\\
4735.33492291712	-65.7301375918894\\
4868.27104012228	-65.9706185549685\\
5004.93910270095	-66.2110995179855\\
5145.44387839092	-66.4515804809438\\
5289.89307609816	-66.6920614438466\\
5438.3974284648	-66.9325424066967\\
5591.07077675529	-67.173023369497\\
5748.03015812535	-67.4135043322502\\
5909.39589534098	-67.6539852949589\\
6075.29168901608	-67.8944662576253\\
6245.84471243962	-68.1349472202518\\
6421.18570906476	-68.3754281828405\\
6601.44909273488	-68.6159091453935\\
6786.7730507233	-68.8563901079126\\
6977.29964966554	-69.0968710703997\\
7173.17494446561	-69.3373520328565\\
7374.54909025955	-69.5778329952846\\
7581.57645752212	-69.8183139576856\\
7794.41575040495	-70.058794920061\\
8013.23012839689	-70.299275882412\\
8238.1873313996	-70.53975684474\\
8469.45980831459	-70.7802378070463\\
8707.22484923992	-71.020718769332\\
8951.6647213783	-71.2611997315982\\
9202.96680876041	-71.501680693846\\
9461.32375589078	-71.7421616560764\\
9726.93361542618	-71.9826426182903\\
10000	-72.2231235804885\\
};
\addlegendentry{ROM ($q=8$)}

\end{axis}
\end{tikzpicture}%

%% file: ind2stokes_stopcrit.tex
%
\definecolor{mycolor1}{rgb}{0.00000,0.40000,0.74000}%
\begin{tikzpicture}

\begin{axis}[%
width=0.951\mywidth,
height=\myheight,
at={(0\mywidth,0\myheight)},
scale only axis,
xmin=2,
xmax=4,
xtick={2, 4},
xlabel style={font=\color{white!15!black}},
xlabel={CURE iteration $k$},
ymode=log,
ymin=1e-07,
ymax=0.0001,
ylabel style={font=\color{white!15!black}},
ylabel={$\frac{\|G_{r,tot}^{(k)}\|_{\mathcal{H}_2}-\|G_{r,tot}^{(k-1)}\|_{\mathcal{H}_2}}{\|G_{r,tot}^{(k)}\|_{\mathcal{H}_2}}$},
axis background/.style={fill=white},
legend style={legend cell align=left, align=left, draw=white!15!black}
]
\addplot [color=mycolor1, line width=1.0pt]
  table[row sep=crcr]{%
2	2.69425781256993e-05\\
3	1.8511224946141e-06\\
4	1.35964464671743e-07\\
};
\addlegendentry{value}

\addplot [color=green!36!orange, dashed, line width=1.5pt]
  table[row sep=crcr]{%
2	1e-06\\
4	1e-06\\
};
\addlegendentry{tolerance}

\end{axis}
\end{tikzpicture}%

%% file: conclusions.tex
\section{Conclusions}%
\label{sec:conclusions}%
In this contribution, we have derived a new formulation of the \HTIP\ of two strictly proper transfer functions in terms of their DAE realizations by introducing projected Sylvester equations. Based on this result, we have proved the validity of \HTPO\ reduction of the strictly proper part of a~given DAE. This makes it possible to extend the fully adaptive rational Krylov reduction method called CUREd SPARK to the DAE case. This method does not require a-priori knowledge of interpolation frequencies nor the reduced order and, in addition, guarantees stability preservation. We have exploited the special structure present in various DAE models for an~efficient numerical implementation. The approximation quality of the reduced-order models obtained by the presented adaptive model reduction method has been demonstrated through numerical examples. The algorithms related to this contribution are available as MATLAB code within the sssMOR toolbox.\par%